\documentclass{article}
\usepackage[a4paper,top=3cm,bottom=3cm,left=2.1cm,right=2.1cm,marginparwidth=1.75cm]{geometry}

\usepackage{graphicx} 

\usepackage{bm}
\usepackage{amsmath, amsfonts, amssymb, amsthm}
\usepackage{mathrsfs, dsfont}
\usepackage[dvipsnames]{xcolor}
\usepackage{graphicx}
\usepackage{verbatim}
\usepackage{float}
\usepackage{pifont} 
\usepackage{comment}
\usepackage{ulem}
\usepackage{bigints}

\usepackage[colorlinks=true, allcolors=blue]{hyperref}
\usepackage{xparse} 
\usepackage{hyperref}
\usepackage{tasks}
\settasks{style=itemize}
\usepackage{parskip}
\usepackage{todonotes}
\usepackage{xcolor} 
\usepackage{subcaption}
\usepackage{makecell}

\usetikzlibrary{decorations.pathmorphing}
\usetikzlibrary {patterns,patterns.meta}
\usetikzlibrary{decorations.pathreplacing}

\usepackage[round]{natbib}
\newtheorem{theorem}{Theorem}[section]
\newtheorem{lemma}[theorem]{Lemma}
\newtheorem{definition}[theorem]{Definition}
\newtheorem{assumption}[theorem]{Assumption}
\newtheorem{proposition}[theorem]{Proposition}

\newtheorem{corollary}[theorem]{Corollary}

\theoremstyle{remark}
\newtheorem{remark}[theorem]{Remark}
\newtheorem{example}[theorem]{Example}

\numberwithin{equation}{section}

\newenvironment{sqremark}{\begin{remark}}{\hfill \tiny $\blacksquare$ \end{remark}}
\newenvironment{sqexample}{\begin{example}}{\hfill \tiny $\blacksquare$ \end{example}}

\newcommand{\D}{\mathbf{D}}

\newcommand{\R}{\mathbb{R}}
\newcommand{\N}{\mathbb{N}}

\newcommand{\E}{\mathbb{E}}

\def\P{\mathbb{P}} 
\newcommand{\F}{\mathcal{F}}

\renewcommand{\Re}{\, \mathfrak{Re}}

\newcommand{\indic}[1]{\mathds{1}_{\left\{ #1 \right\}}}

\newcommand{\alphabet}[1][d]{A_{#1}}
\newcommand{\TA}[1][d]{T(\R^{#1})}
\newcommand{\eTA}[1][d]{T((\R^{#1}))}

\newcommand{\emptyword}{{\color{NavyBlue}{\mathbf{\varnothing}}}}
\newcommand{\word}[1]{{\color{NavyBlue}{\mathbf{#1}}}}
\newcommand{\proj}[1]{|_{\word{#1}}}

\newcommand{\bp}{\bm{p}}

\newcommand{\bsigma}{\bm{\sigma}}

\newcommand{\bell}{\bm{\ell}}

\newcommand{\shuprod}{\mathrel{\sqcup \mkern -3mu \sqcup}}
\newcommand{\shupow}[1]{^{\shuprod #1}}

\newcommand{\conpow}[1]{^{\otimes #1}}

\NewDocumentCommand{\sighat}{O{t} O{W}}{\widehat{\mathbb{#2}}_{#1}}
\NewDocumentCommand{\sigtilde}{O{t} O{W}}{\widetilde{\mathbb{#2}}_{#1}}
\NewDocumentCommand{\sig}{O{t} O{W}}{\sighat[#1][#2]}
\newcommand{\sigW}[1][t]{\widehat{\mathbb{W}}_{#1}}
\newcommand{\sigB}[1][t]{\widehat{\mathbb{B}}_{#1}}
\newcommand{\sigX}[1][t]{{\mathbb{X}}_{#1}}
\newcommand{\sigY}[1][t]{{\mathbb{Y}}_{#1}}
\newcommand{\sigXhat}[1][t]{{\widehat{\mathbb{X}}}_{#1}}

\newcommand{\bracket}[2]{\left \langle #1, #2 \right \rangle}
\NewDocumentCommand{\bracketsig}{O{t} O{W} m}{\bracket{#3}{\sig[#1][#2]}}
\newcommand{\bracketsigW}[2][t]{\bracket{#2}{\widehat{\mathbb{W}}_{#1}}}
\newcommand{\bracketsigX}[2][t]{\bracket{#2}{{\mathbb{X}}_{#1}}}

\NewDocumentCommand{\bracketsigtrunc}{O{M} O{t} O{W} m}{\bracket{#4}{\sig[#2][#3]^{\leq #1}}}
\NewDocumentCommand{\bracketsigtilde}{O{t} O{W} m}{\bracket{#3}{\sigtilde[#1][#2]}}

\usepackage{mathtools}
\usepackage{shuffle}
\usepackage{bbm}
\usepackage{authblk}
\mathtoolsset{showonlyrefs}
\definecolor{word_color}{HTML}{B56246}
\definecolor{alert_color}{HTML}{B56246}
\definecolor{title_color}{HTML}{75818c}
\definecolor{main_color}{HTML}{5b6c64}
\definecolor{alert_color}{HTML}{B56246}
\definecolor{my_blue}{HTML}{4699B5}
\definecolor{my_red}{HTML}{B54662}
\definecolor{my_purple}{HTML}{9A46B5}

\usepackage{enumitem}

\title{Malliavin calculus for signatures with applications to finance}

\author[1]{Eduardo Abi Jaber\thanks{eduardo.abi-jaber@polytechnique.edu. The first author is grateful for the financial support from the Chaires FiME-FDD, Financial Risks at Ecole Polytechnique.}}
\author[1]{Clément Rey\thanks{clement.rey@polytechnique.edu.}}
\author[1, 2]{Dimitri Sotnikov\thanks{dmitrii.sotnikov@polytechnique.edu. The third author is grateful for the financial support provided by Engie Global Markets.}}
\affil[1]{Ecole Polytechnique, CMAP}
\affil[2]{Engie Global Markets}

\begin{document}

\maketitle

\begin{abstract}
Malliavin calculus is a powerful and general framework for the analysis of square-integrable random variables, but it often suffers from a lack of tractability and explicit representations. To address this limitation, we focus on a subclass of random variables given by finite linear combinations of time-extended Brownian motion signatures. The class remains rich due to the universal approximation properties of signatures. Leveraging the algebraic structure of signatures, we first derive explicit formulas for the Malliavin derivative of signatures of continuous It\^o processes. As a consequence, we obtain closed-form expressions for the Clark--Ocone representation, the Ornstein--Uhlenbeck semigroup and its generator, as well as the integration-by-parts formula within the class of Brownian signature variables. These results provide purely algebraic formulations of the classical operators of Malliavin calculus. As an application, we compute Greeks for general path-dependent options under signature volatility models, and numerically compare different choices of Malliavin weights.
\end{abstract}

\begin{description}
\item[Mathematics Subject Classification (2010):] 60H07, 60H30 
\item[Keywords:] Path signatures, Malliavin calculus, Integration by parts, Greeks.
\end{description}

\section{Introduction}
A central problem in probability is to establish differentiability of the mapping
$
\theta \mapsto \E\big[f(G^\theta)\big],
$
where $G^\theta$ is a random variable depending on a parameter $\theta \in \R$, 
and to derive tractable representations of the associated parameter sensitivity
$$
\frac{\partial}{\partial \theta}\E\big[f(G^\theta)\big].
$$

The 
\cite{malliavin1978stochastic} calculus\footnote{We refer to \cite*{nualart2006malliavin, alos2021malliavin} for a comprehensive treatment of Malliavin calculus and its applications.},  which serves as a stochastic analogue of the calculus of variations,  provides a powerful tool to address this problem whenever $G^\theta$ is an $\mathcal{F}^W_T$-measurable random variable with $W$  a Brownian motion.   Under suitable regularity and non-degeneracy assumptions, the Malliavin integration-by-parts formula yields the representation
\begin{align}\label{eq:ippintro}
\frac{\partial}{\partial \theta}\E\big[f(G^\theta)\big]
=
\E\big[f(G^\theta)\pi_T\big],
\end{align}
where $\pi_T$ denotes the associated Malliavin weight. An important application of such formulas is the derivation of explicit expressions for the density of  random variables. In particular, \cite{bismut1981martingales}  introduced a probabilistic proof of H\"ormander's theorem based on integration by parts on the Wiener space, establishing the smoothness and lower bounds of densities for a broad class of diffusion processes \citep{kusuoka1987applicationsII,Kohatsu2003, nualart2006malliavin,Bally2006,NourdinViens2009}.  Furthermore, the representation \eqref{eq:ippintro} is particularly attractive for Monte Carlo methods, as it avoids differentiating the function $f$. 
Integration-by-parts techniques have proven useful in stochastic optimal control \citep{Gobet2005}, gradient estimation in machine learning \citep{oden2025backpropogation}, and the improvement of learning procedures \citep{pidstrigach2025conditioning}.
In mathematical finance, the computation of sensitivities (the so-called Greeks) via Malliavin integration by parts was first developed within the Black--Scholes framework \citep*{Fournie1999}. This approach was later extended to stochastic volatility models \citep{ElKhatib2009}, including hybrid \citep{yilmaz2018computation}, jump-diffusion \citep{Khatib2003, Davis2006, Petrou2008}, and non-Markovian settings \citep*{al2023computation}, and to more complex derivatives such as barrier, lookback \citep{Gobet2003}, and American options \citep*{Bally2003}.

However, in general, the Malliavin weight $\pi_T$ is either not explicit or difficult to compute numerically, except in very specific configurations where additional structure is available.
This naturally leads to the following question:
\begin{center}
\textit{Can one identify a structural and generic class of random variables $G^\theta$ for which the Malliavin weight admits an explicit and numerically tractable form?}
\end{center}

A natural solution comes from the theory of path signatures, introduced by \citet{chen1957integration} as sequences of iterated integrals, which provide a systematic way to represent functionals of continuous paths. 
More precisely, the signature $\widehat{\mathbb{W}}$ of the time-augmented Brownian motion $\widehat{W}_t = (t, W_t)$, $t \ge 0$, is defined as the sequence
\[
    \sig =
    \left(
        1,
        \begin{pmatrix}
            t \\ W_t
        \end{pmatrix},
        \begin{pmatrix}
            \frac{t^2}{2!} & \int_0^t s\, dW_s \\
            \int_0^t W_s\, ds & \frac{W_t^2}{2!}
        \end{pmatrix},
        \ldots
    \right),
\]
and serves as a universal approximator: it naturally appears in stochastic Taylor expansions of Markovian  \citep*{ benarous1989flots, kloeden1992stochastic} and non-Markovian systems  \citep*{fliess1983concept,  AbiJaber2024a}, and satisfies a Stone–Weierstrass–type theorem for functionals of Brownian motion. This highlights the role of signatures as analogues of polynomials on path space and underlines their foundational importance in the theory of rough paths initiated by Terry Lyons (see, e.g., \citealp*{lyons2007differential}).

The interplay between Malliavin calculus and the theory of rough paths and signatures  {has received considerable attention.} For instance, \citet[Chapter 20]{friz_victoir_2010} apply Malliavin calculus to establish the existence of smooth densities for rough differential equations (RDEs). Further developments include the study of Malliavin differentiability of RDE solutions by \citet{INAHAMA20141566}, the analysis of Malliavin differentiability and existence of smooth densities for rough SDEs by \citet*{bugini2024malliavin}, and the computation of the Wiener–Itô chaos expansion of Gaussian-process signatures by \citet{cass2024wiener}. Moreover, Malliavin calculus was used to establish the densities of the truncated path signatures themselves. See \citet{kusuoka1987applications} for the case of Brownian motion signature and \citet*{baudoin2020density} for the fractional Brownian motion case. Regarding the computation of Greeks, \citet{teichmann2006calculating} proposed an approach relying on the cubature on Wiener space introduced by \citet{Lyons04cubature}. {Most existing approaches start from an SDE, use a stochastic Taylor expansion to bring out the signature, and then, under conditions such as Hörmander’s, apply Malliavin calculus. A comprehensive framework for applying Malliavin calculus directly to the Brownian signature $\sigW$  is not yet fully developed. We instead isolate the Malliavin/signature component, aiming at a unified and generic treatment that exploits the algebraic structure of signatures and extends beyond the SDE setting to non-Markovian dynamics.}

Going back to our question, and motivated by the universal approximation properties of path signatures, we restrict our attention to random variables that are linear in the signature of Brownian motion, of the form
$$ G^\theta =\bracketsigW[T]{\bell^{\theta}}, $$
where the coefficients $\ell^\theta$ are deterministic. A first crucial observation is that this representation decouples the stochastic noise from the parameter dependence. As a consequence, the sensitivity analysis reduces to first understanding the Malliavin calculus of the generic object $\widehat{\mathbb{W}}$ and then applying integration-by-parts formulas to the class of linear functionals. Compared to the related literature, a key feature of our approach is that all fundamental Malliavin operations on the signature are systematically encoded algebraically as linear operators on the (extended) tensor algebra. This allows us to exploit the rich algebraic structure of signatures to derive explicit expressions for the Malliavin weights, while the computability of $\sigW[T]$ enables the design of numerically tractable implementations.

 To implement this strategy, we first develop a probabilistic perspective on the path signature of stochastic processes, which allows us to derive results for the Malliavin derivative of the signature of general Itô processes (Theorem~\ref{thm:malliavindersig}) echoing the calculus of variations for rough differential equations \citep{marie2015sensitivities}, but using a purely probabilistic approach. In Theorem~\ref{thm:strato_mal_der}, we further identify a class of processes for which the Stratonovich integral is Malliavin differentiable, and we derive a formula for its Malliavin derivative.

 By combining this viewpoint with the powerful  algebraic properties of the signature, we then 
show that several fundamental operations from Malliavin calculus admit explicit and tractable representations when applied to Brownian signatures.  In particular, we derive explicit expressions for the Skorokhod integral (Proposition~\ref{Prop:skorokhod}), the Ornstein--Uhlenbeck semigroup (Theorem~\ref{thm:ou_semigroup}), chaos expansions (Proposition~\ref{Prop:Chaos}), and ultimately to an integration-by-parts formula (Theorem~\ref{thm:ibp_sig}) for functionals of the form  $f(G^\theta)$, where the Malliavin weights are determined explicitly and belong to a class of random variables stable under integration by parts. More precisely, this stable  class turns out to be the class of rational functions of $\sigW[T]$, that is, ratios of two linear functionals of the signature.  

As an application, we use our results to derive and compute model sensitivities (Greeks) of path-dependent payoffs in signature volatility models in Section~\ref{sect:application}. 
Signature volatility models form a universal and flexible class of volatility models as shown by  \citet*{arribas2020sig, cuchiero2023spvix, abijaber2024signature}, capable of representing and approximating a wide range of (possibly non-Markovian) dynamics while preserving computational tractability for pricing and hedging derivatives. Within this setting, our  Malliavin integration-by-parts formula  yields explicit and reusable expressions for Greeks of a large class of path-dependent payoffs. 
In particular, any payoff of the form
$
f\!\left(\bracketsigX[T]{\ell}\right),
$
where $\sigX[T]$ denotes the signature of the time-augmented log-price process $(t, \log S_t)$, can be treated within a unified framework. 
The resulting sensitivity formula depends only on the deterministic coefficient $\ell$, without requiring additional model-specific derivations, which makes the approach particularly  robust and attractive for practical implementation. For completeness, we perform numerical comparisons of different choices of Malliavin weights for European and Asian options.

The results presented in this paper can be useful in a variety of settings where signatures are employed.  Beyond volatility modeling, signatures have found wide-ranging applications across various domains; see \citet{chevyrev2016primer} for a comprehensive overview of the signature method. In particular, they have been used   to solve stochastic control problems \citep*{bayer2023optimal, bayer2023primal, jaber2025signatureapproachpricinghedging} to capture memory effects and nonlinearities in machine learning for time series \citep*{levin2016learningpastpredictingstatistics, drobac2025slidingwindowsignaturestimeseries, jaber2025hedgingmemoryshallowdeep}, and in finance \citep*{arribas-nonparam, arribas-optiexec}.  

We finally mention that an alternative framework to Malliavin calculus is Functional Itô Calculus, introduced by \citet{Dupire2019} and subsequently employed by \citet{dupire2022functional} to prove the functional Taylor expansion. These results were then extended by \citet*{cuchiero2025functionalito} to functionals of c\`adl\`ag rough paths, leveraging signature approximation properties, and by \citet{bielert2025roughfunctionalitoformula} using rough integration theory. The connection between the functional (vertical) derivative and the Malliavin derivative was established in \citet{Cont2013}. Building on these ideas, \citet{Jazaerli2017} proposed an approach to Greek computation based on integration by parts with functional derivatives. However, their tractable formula applies only to so-called {weakly path-dependent} functionals, which exclude many path-dependent derivative contracts of practical interest.

\paragraph{Outline.} Section~\ref{sect:summarymain} provides a summary of our main results.  Section~\ref{sect:preliminaries} introduces the necessary preliminaries from path signatures. In Section~\ref{sect:der_ito_process}, we derive the formula for the Malliavin derivative of the signature of Itô processes. Section~\ref{sect:malliavin_sigW} is devoted to developing the framework for working with the signature of Brownian motion and provides explicit computations of certain Malliavin calculus operators for this case. In Section~\ref{sect:ibp}, we prove the integration by parts formula for random variables that depend on the Brownian signature. Finally, in Section~\ref{sect:application}, we apply these results to the computation of Greeks for possibly path-dependent options under the signature volatility model. Appendix~\ref{sect:strato_diff} contains the proof of the Malliavin differentiability of Stratonovich integrals.

\section{Summary of main results}\label{sect:summarymain}

We first introduce the notation  and basic framework in Section~\ref{sect:prelim_first}, and then present streamlined versions of our principal theorems in Section~\ref{sect:mainres}.

\subsection{Notations}  \label{sect:prelim_first}
From now on, let $T>0$ be fixed. Throughout this paper, we work on a filtered probability space \((\Omega, \mathcal{F}, (\mathcal{F}_t^W)_{t \in [0, T]}, \P)\) supporting a \(d\)-dimensional Brownian motion \(W\), where \((\mathcal{F}_t^W)_{t \in [0, T]}\) denotes the filtration generated by \(W\).

We will denote by $\circ$ the Stratonovich integrals. Namely, for two continuous semimartingales $X = (X_t)_{t\in[0, T]}$ and $Y = (Y_t)_{t\in[0, T]}$, the Stratonovich integral is defined by
$$
\int_0^\cdot Y_s\circ dX_s = \int_0^\cdot Y_s\, dX_s + \dfrac{1}{2}\left\langle X, Y\right\rangle_\cdot,
$$
where the integral on the right-hand side is the standard Itô integral $\langle X, Y \rangle$ denotes the quadratic covariation  between $X$ and $Y$.

We will denote by $C^k_p(V, W)$ (resp. $C^k_b(V, W)$) the space of $k$ times differentiable functions from $V$ to $W$ with all derivatives of polynomial growth (resp. bounded).

\paragraph{Path signature.} We introduce the key objects needed to present the main results of this work. For a more thorough introduction to path signatures, we refer to \citet{chevyrev2016primer}. We start with the definition of the path signature of a continuous $\R^d$-valued semimartingale $X = (X_t)_{t \in [0, T]}$, which will be denoted by $\mathbb{X}$. The \textit{path signature} $\sigX[s, t]$ of $X$ over $[s, t]$ for $0 \le s \le t \le T$ is the infinite sequence of iterated Stratonovich integrals:
\begin{equation}\label{eq:sig_def}
    \sigX[s, t] := \left( \sigX[s, t]^{\word{i_1\ldots i_n}} \right)_{\word{i_1\ldots i_n} \in V}
    = \left(\ \, \idotsint\limits_{s \le u_1 \le \ldots \le u_n \le t}\circ dX_{u_1}^{i_1}\ldots\circ dX_{u_n}^{i_n}\right)_{\word{i_1\ldots i_n} \in V}.
\end{equation}
Here, $V$ denotes the set of all words in the alphabet $A = \{\word{1}, \ldots, \word{d}\}$. We will also use the standard notation $\sigX[t] := \sigX[0, t]$.

\begin{sqremark}
    Throughout this paper, we will often consider the \textit{time-augmented} paths $\widehat{X}\colon t \mapsto (t, X_t) \in \R^{d+1}$ and their signatures, denoted by $\widehat{\mathbb{X}}_t$. In this case, it is convenient to consider the alphabet $\{\word{0}, \word{1}, \ldots, \word{d}\}$, where the letter $\word{0}$ corresponds to integration with respect to the time component, while the letters $\word{i}$ correspond to integration with respect to $X^i$ for $i = 1, \ldots, d$.
\end{sqremark}

\paragraph{Extended tensor algebra.} For $n \ge 0$, each word $\word{v} = \word{i_1 \ldots i_n} \in V$ of length $|\word{v}| = n$, with $\word{i_k} \in A,\, k = 1, \ldots, n$, can be identified with a basis vector $e_{i_1} \otimes \ldots \otimes e_{i_n}$ of $(\R^d)^{\otimes n}$, where $e_1, \ldots, e_d$ form an orthonormal basis of $\R^d$. We use the convention $(\R^d)^{\otimes 0} \cong \R$ and denote the corresponding empty word by $\emptyword$. This leads to the definition of the \textit{extended tensor algebra}
\begin{equation}
    \eTA := \prod_{n \ge 0} (\R^d)\conpow{n}
    = \left\{ \bell = (\bell_0, \bell_1, \ldots)\colon\ \bell_n \in (\R^d)\conpow{n},\ n \ge 0 \right\}.
\end{equation}
Since each element $\bell_n \in (\R^d)\conpow{n}$ admits a representation $\bell_n = \sum_{|\word{v}| = n} \bell^{\word{v}}\word{v}$, each element $\bell \in \eTA$ can be written as the formal sum
\begin{equation}
    \bell = \sum_{n \ge 0}\sum_{|\word{v}| = n} \bell^{\word{v}}\word{v}.
\end{equation}
In particular, the path signature defined in \eqref{eq:sig_def} belongs to $\eTA$ and admits the representation
\begin{equation*}
    \sigX[s, t] = \sum_{n \ge 0}\sum_{|\word{v}| = n} \sigX[s, t]^{\word{v}}\word{v}.
\end{equation*}

\paragraph{Tensor algebra and duality.} We denote by $\TA \subset \eTA$ the \textit{tensor algebra}
$$
\TA := \bigoplus_{n \ge 0} (\R^d)\conpow{n},
$$
    i.e. the subspace of sequences $\bell \in \eTA$ having only finitely many non-zero components. Note that $\TA$ is the topological dual of $\eTA$ (equipped with the product topology), with pairing given by
\begin{equation*}
    \langle \bell, \mathbb{X} \rangle
    = \sum_{n \ge 0}\sum_{|\word{v}| = n} \bell^{\word{v}}\mathbb{X}^{\word{v}},
    \quad \bell \in \TA,\quad \mathbb{X} \in \eTA,
\end{equation*}
where the sum is well-defined since only finitely many terms are nonzero. Moreover, given a continuous linear operator ${\bf A}\colon \eTA \to \eTA$, we call ${\bf A}^*\colon \TA \to \TA$ its \textit{adjoint} if
$$
    \langle \bell, {\bf A}\mathbb{X}\rangle = \langle {\bf A}^*\bell, \mathbb{X}\rangle,
$$
for all $\bell \in \TA$ and $\mathbb{X} \in \eTA$.
\paragraph{Algebraic structure.} The extended tensor algebra $\eTA$ becomes an algebra with unit element $\emptyword$ when equipped with the standard operations of addition, scalar multiplication, and the tensor product defined by
$$
\bell\otimes\bp \in \eTA, \qquad (\bell\otimes\bp)_n = \sum_{k = 0}^n \bell_k \otimes \bp_{n - k}, \quad n \geq 0.
$$
A key algebraic property of the path signature, known as the \citet{chen1957integration} identity, states that the signature of a path over $[s, t]$ is equal to the product of its signatures over $[s, u]$ and $[u, t]$ for all $s \leq u \leq t$, namely,
\begin{equation}\label{eq:chen}
    \sigX[s, t] = \sigX[s, u] \otimes \sigX[u, t].
\end{equation}
In particular, it follows from the uniqueness result of \citet{boedihardjo2016signature} that the signature of a path $X$ concatenated with its time-reversal $ \overleftarrow{X}^{[0, T]}\colon t \mapsto X_{T - t}$ is trivial. By Chen's identity,
$$
\sigX[T] \otimes \overleftarrow{\mathbb{X}}_T^{[0, T]} = \emptyword,
$$
and hence the inverse group element of \(\sigX[T]\) is given by the signature of the time-reversed path, \(\sigX[T]^{-1} = \overleftarrow{\mathbb{X}}_T^{[0, T]}\).

\paragraph{Malliavin derivative of path signatures.} 

We introduce the space of simple functionals as follows:
\begin{align*}
\mathcal{S} :=  \left\{F\colon \Omega \to \R\colon\  F(\omega)=f(W_{t_{k+1}}(\omega)-W_{t_k}(\omega),\ k=0,\ldots, 2^{n}-1),\ f \in \mathcal{C}^{\infty}_{p}((\mathbb{R}^{d})^{2^{n}}, \R),\ t_{k}=\frac{kT}{2^{n}},\ n \in \mathbb{N}  \right\}.
\end{align*}
For $F \in \mathcal{S}$, we define the \textit{Malliavin derivative} of $F$ as an element of $L^2(\Omega\times[0, T], \R^{d})$ with components
\begin{align}\label{eq:mal_deriv_def}
\mathbf{D}^{j}_{t} F = \partial_{x^{j}_l} f(W_{t_{k+1}} - W_{t_k}, k=0,\ldots, 2^{n}-1), \quad j = 1, \ldots, d,
\end{align}
for $t \in (t_l, t_{l+1}]$, where $\partial_{x^{j}_l}$denotes the partial derivative with respect to the $j$-th component of the argument $x_l$ , associated with the increment $W_{t_{l+1}} - W_{t_{l}}$.
We use the convention $\D_0^j F=0$. 

Higher-order Malliavin derivatives are defined recursively as
\begin{align*}
\mathbf{D}_{t_{1}}\ldots \mathbf{D}_{t_{\ell}}F \in (\R^d)^{\otimes \ell}.
\end{align*}

Notice that $\mathcal{S}$ is dense in $L^{2}(\Omega)$, and $\mathbf{D}$ is a closable operator from $L^{2}(\Omega)$ to $L^2(\Omega\times[0, T], \R^{d})$.
For $k = 1, 2, \ldots$, the space $\mathbb{D}^{k, 2}$ is defined as the closure of $\mathcal{S}$ with respect to the norm
\begin{align*}
\Vert F \Vert_{\mathbb{D}^{k,2}}^2 := \E\left[|F|^2\right] + \sum_{m=1}^k\E\left[\int_0^T\ldots\int_0^T |\D_{t_1}\ldots\D_{t_m} F|_{(\R^d)^{\otimes m}}^2\, dt_1\ldots dt_m\right].
\end{align*}
In particular, $\mathbb{D}^{k,2}$ is equipped with the inner product 
\begin{align*}
\langle F,G \rangle_{k,2} := \E\left[FG\right] + \sum_{m=1}^k\E\left[\int_0^T\ldots\int_0^T \langle\D_{t_1}\ldots\D_{t_m} F, \D_{t_1}\ldots\D_{t_m} G\rangle_{(\R^d)^{\otimes m}}\, dt_1\ldots dt_m\right].
\end{align*}
and hence forms a Hilbert space. We also define $\mathbb{D}^{\infty, 2} := \bigcap_{k \geq 1} \mathbb{D}^{k,2}$.

The Malliavin derivative satisfies the chain rule. For a random vector $F = (F_1, \dots, F_k)$ such that $F_i \in \mathbb{D}^{1,2}$ for all $i$, and for a function $\varphi \in C^1_b(\mathbb{R}^k, \mathbb{R})$, the random variable $\varphi(F)$ is also Malliavin differentiable, and
\begin{equation}\label{eq:chain_rule}
    \D(\varphi(F)) = \sum_{i=1}^{k} \partial_{x_i} \varphi(F) \D F_i.
\end{equation}

When the signature components \(\sigX[T]^{\word{v}}\) of an $\R^m$-valued process \(X\) belong to \(\mathbb{D}^{k,2}\) for all $\word{v} \in V$, we denote by \(\D_{t_1}^{j_1}\ldots\D_{t_k}^{j_k} \sigX[T]\) the element-wise application of the Malliavin derivative:
\begin{equation}
    \D_{t_1}^{j_1}\ldots\D_{t_k}^{j_k} \sigX[T] := \left( \D_{t_1}^{j_1}\ldots\D_{t_k}^{j_k} \sigX[T]^{\word{v}} \right)_{\word{v} \in V} \in \eTA[m].
\end{equation}
Then the element $\underbrace{\D \cdots \D}_{k\text{ times}} \sigX[T]$ belongs to
$
L^2\!\left(\Omega \times [0,T]^k,\, (\mathbb{R}^d)^{\otimes k} \otimes \eTA[m]\right).
$

\subsection{Main results}\label{sect:mainres}

We  now present our main results.

\paragraph{Malliavin derivative of signatures of Itô processes.} Consider a continuous $\R^m$-valued Itô process of the form
$$
dX_t = a_t^0\,dt + \sum_{k = 1}^d a_t^k \circ dW_t^k = A_t \circ d\widehat{W}_t,
$$
where \(a_t^j \in \R^m\), \(j = 0, 1, \ldots, d\), are square-integrable adapted processes satisfying additional Malliavin differentiability and regularity conditions specified in Assumption~\ref{assump:Strato_reg_assump}, which together form the columns of the matrix $A_t \in \R^{m \times (d+1)}$, and $\widehat{W}_t = (t, W_t)$. We show that the Malliavin derivative of $\sigX$, the signature of $X$, when exists, is given explicitly by
\begin{equation}
    \D_s \sigX[t] = \sigX[s] \otimes A_s \otimes \sigX[s, t] \;+\; \int_s^t \sigX[u] \otimes \left( \D_s A_u \circ d\widehat{W}_u \right) \otimes \sigX[u, t], \quad s \leq t,
\end{equation}
see Theorem~\ref{thm:malliavindersig}. In particular, when \(X_t = \widehat{W}_t = (t, W_t)\), this implies
\begin{equation}\label{eq:mal_multiple_der}
    \D_{s_1}^{i_1} \cdots \D_{s_n}^{i_n} \sigW[t]
    = \sigW[s_1] \otimes \word{i_1} \otimes \sigW[s_1, s_2] \otimes \word{i_2} \otimes \cdots \otimes \word{i_n} \otimes \sigW[s_n, t],
    \quad 0 \leq s_1 \leq \cdots \leq s_n \leq t.
\end{equation}
{The ``pierced'' formula \eqref{eq:mal_multiple_der} shows that iterated Malliavin derivatives of the Brownian signature correspond to inserting the letters 
$\word{i_1}, \ldots, \word{i_n}$ 
into the signature at the differentiation times. The formula admits a natural geometric interpretation, as explained in the next remark.
}
\begin{sqremark}
    By definition \eqref{eq:mal_deriv_def}, the Malliavin derivative 
    $\D^i_{s},\, s\geq 0$ of a functional $F((W_u)_{u \in [0,t]})$ can be formally viewed as the directional derivative of $F$ in the direction of the perturbation $\indic{[s, +\infty)}\word{i}$, that is,
    \[
        \D_s^i\sigW = \frac{d}{d\epsilon}\bigg|_{\varepsilon = 0}
        F\big((W_u + \varepsilon\,\indic{u \ge s}\word{i})_{u \in [0,t]}\big).
    \]
    Of course, the Malliavin derivative can be rigorously interpreted as a directional derivative only along directions in the Cameron–Martin space, which does not contain indicator functions. Therefore, this remark is purely formal and is intended only as a geometric illustration of the result.
    Geometrically, this variation corresponds to inserting a short ``vertical'' segment of length $\varepsilon$ in the direction $\word{i}$ between the two parts $(W_u)_{u \in [0,s]}$ and $(W_u)_{u \in [s,t]}$ of the Brownian path (see Figure~\ref{fig:pert_path}). 
    The signature of this vertical segment is $\exp^\otimes(\varepsilon\,\word{i})  := \sum\limits_{n\ge 0} \dfrac{\varepsilon^n\word{i}^{\otimes n}}{n!}$.
    Hence, by Chen's identity, the signature $\sigW[t]^\varepsilon$ of the perturbed path satisfies
    \[
        \sigW[t]^\varepsilon
        = \sigW[s] \otimes \exp^\otimes(\varepsilon\,\word{i}) \otimes \sigW[s,t].
    \]
    Differentiating at $\varepsilon = 0$ yields
    \[
        \D_s^i\sigW = \frac{d}{d\varepsilon}\bigg|_{\varepsilon = 0}
        \big(\sigW[s] \otimes \exp^\otimes(\varepsilon\,\word{i}) \otimes \sigW[s,t]\big)
        = \sigW[s] \otimes \word{i} \otimes \sigW[s,t].
    \]
    This provides a clear geometric link between the Malliavin derivative and the vertical (or pathwise) functional derivatives of the signature functionals considered by \citet*{cuchiero2025functionalito}.
\end{sqremark}

\begin{center}
\begin{figure}[H]
    \centering
\begin{tikzpicture}[scale=3]
        \draw[->] (-0.1, 0) -- (2.6, 0) node[right] {$\word{0}$};
        \draw[->] (0, -0.6) -- (0, 1) node[right] {$\word{i}$};

        \draw[dashed, -] (2.5, -0.43) -- (2.5, 0) node[above] {$t$};
        \draw[dashed, -] (1.0, 0.38) -- (1, 0) node[below] {$s$};
        
        \draw[very thick, alert_color, -] (0.0, 0.0) -- (0.05, 0.17) -- (0.1, 0.14) -- (0.15, 0.14) -- (0.2, 0.07) -- (0.25, 0.25) -- (0.3, 0.1) -- (0.35, 0.08) -- (0.4, 0.25) -- (0.45, 0.44) -- (0.5, 0.63) -- (0.55, 0.73) -- (0.6, 0.56) -- (0.65, 0.43) -- (0.7, 0.51) -- (0.75, 0.59) -- (0.8, 0.32) -- (0.85, 0.19) -- (0.9, 0.3) -- (0.95, 0.32) -- (1.0, 0.38) -- (1.05, 0.36) -- (1.1, 0.24) -- (1.15, 0.28) -- (1.2, 0.22) -- (1.25, 0.14) -- (1.3, 0.03) -- (1.35, 0.12) -- (1.4, 0.15) -- (1.45, -0.13) -- (1.5, -0.1) -- (1.55, -0.1) -- (1.6, -0.17) -- (1.65, -0.08) -- (1.7, -0.03) -- (1.75, -0.27) -- (1.8, -0.45) -- (1.85, -0.29) -- (1.9, -0.3) -- (1.95, -0.12) -- (2.0, 0.1) -- (2.05, 0.1) -- (2.1, -0.12) -- (2.15, -0.11) -- (2.2, -0.32) -- (2.25, -0.17) -- (2.3, -0.19) -- (2.35, -0.2) -- (2.4, -0.21) -- (2.45, -0.35) -- (2.5, -0.43);

        \draw[very thick, main_color, -] (1.0, 0.38 + 0.3) -- (1.05, 0.36 + 0.3) -- (1.1, 0.24 + 0.3) -- (1.15, 0.28 + 0.3) -- (1.2, 0.22 + 0.3) -- (1.25, 0.14 + 0.3) -- (1.3, 0.03 + 0.3) -- (1.35, 0.12 + 0.3) -- (1.4, 0.15 + 0.3) -- (1.45, -0.13 + 0.3) -- (1.5, -0.1 + 0.3) -- (1.55, -0.1 + 0.3) -- (1.6, -0.17 + 0.3) -- (1.65, -0.08 + 0.3) -- (1.7, -0.03 + 0.3) -- (1.75, -0.27 + 0.3) -- (1.8, -0.45 + 0.3) -- (1.85, -0.29 + 0.3) -- (1.9, -0.3 + 0.3) -- (1.95, -0.12 + 0.3) -- (2.0, 0.1 + 0.3) -- (2.05, 0.1 + 0.3) -- (2.1, -0.12 + 0.3) -- (2.15, -0.11 + 0.3) -- (2.2, -0.32 + 0.3) -- (2.25, -0.17 + 0.3) -- (2.3, -0.19 + 0.3) -- (2.35, -0.2 + 0.3) -- (2.4, -0.21 + 0.3) -- (2.45, -0.35 + 0.3) -- (2.5, -0.43 + 0.3);
        \draw[very thick, my_blue, -] (1.0, 0.38) -- (1, 0.38 + 0.3) node[midway, xshift=-1.3ex]{\textcolor{my_blue}{$\varepsilon$}};

        \draw[->,decorate,decoration={snake,amplitude=1mm,segment length=5mm, pre length=0.3ex,post length=0.3ex}] (0.7, 0.77) -- (1.0 - 0.1, 0.38 + 0.2);

        \draw [decorate,decoration={brace, amplitude=2ex, mirror}]
  (0.05,-0.1) -- (1 - 0.05,-0.1) node[midway, yshift=-4ex]{\textcolor{alert_color}{$\sigW[s]$}};
        \draw [decorate,decoration={brace, amplitude=2ex}]
          (1 + 0.05, 0.7) -- (2.5 - 0.05, 0.7) node[midway, yshift=4ex]{\textcolor{main_color}{$\sigW[s, t]$}};

        \node[above] at (0.7, 0.77) {$\textcolor{my_blue}{\exp^\otimes(\varepsilon\,\word{i})}$};
        \node[below] at (1.6, -0.18) {$\textcolor{alert_color}{W}$};
\end{tikzpicture}
    \caption{Perturbed path and signatures of its components.}
    \label{fig:pert_path}
\end{figure}
\end{center}

\paragraph{Integration of Malliavin derivatives.}
The algebraic structure of the ``pierced'' signature \eqref{eq:mal_multiple_der} allows us to rewrite certain functionals that appear in the Malliavin calculus applied to \(\sigW\) as linear functionals. Motivated by \eqref{eq:mal_multiple_der}, we show in Theorem~\ref{thm:iterated_psi} that for a general \(d\)-dimensional path \(X\),
\begin{equation}\label{eq:inserted_integration}
    \int_{0 \leq s_1 \leq \cdots \leq s_n \leq t}
    \left(
        \sigX[s_1] \otimes \word{i_1} \otimes \sigX[s_1, s_2]
        \otimes \cdots \otimes \word{i_n} \otimes \sigX[s_n, t]
    \right)
    \circ dX_{s_1}^{j_1} \cdots \circ dX_{s_n}^{j_n}
    =
    \mathbf{\Psi}_{\word{i_1}, \ldots, \word{i_n}}^{\word{j_1}, \ldots, \word{j_n}}(\sigX[t]),
\end{equation}
where \(\mathbf{\Psi}_{\word{i_1}, \ldots, \word{i_n}}^{\word{j_1}, \ldots, \word{j_n}} : \eTA \to \eTA\) is an operator acting on the extended tensor algebra by replacing the letters \(\word{i_1}, \ldots, \word{i_n}\) with \(\word{j_1}, \ldots, \word{j_n}\), as described in Definition~\ref{def:psi_def}.  
Its adjoint is shown to be \(\mathbf{\Psi}^{\word{i_1}, \ldots, \word{i_n}}_{\word{j_1}, \ldots, \word{j_n}}\), so that
\begin{equation*}
    \langle \bell,\, \mathbf{\Psi}_{\word{i_1}, \ldots, \word{i_n}}^{\word{j_1}, \ldots, \word{j_n}}(\sigX[]) \rangle
    = 
    \bracketsigX[]{
        \mathbf{\Psi}^{\word{i_1}, \ldots, \word{i_n}}_{\word{j_1}, \ldots, \word{j_n}}(\bell)
    },
    \qquad \bell \in \TA,\ \sigX[] \in \eTA,
\end{equation*}
as shown in Proposition~\ref{prop:psi_adjoint}.

This allows us to identify certain Malliavin-calculus operations applied to the Brownian signature \(\sigW\) with operators on \(\TA[d+1]\) and to show that the class of random variables
\begin{equation}\label{eq:sig_polynom}
    \mathcal{P}(\sigW)
    :=
    \left\{
        \bracketsigW[t]{\bell}
        \;\colon\;
        \bell \in \TA[d+1]
    \right\},
\end{equation}
defined as linear functionals of the Brownian signature with finitely many non-zero coefficients, is closed under these operators. For example, 
\begin{equation*}
    \int_{0 \leq s_1 \leq \cdots \leq s_n \leq t}
        \D_{s_1}^{i_1} \cdots \D_{s_n}^{i_n}
        \langle \bell, \sig \rangle
        \circ d\widehat{W}_{s_1}^{j_1} \cdots \circ d\widehat{W}_{s_n}^{j_n}
    =
    \bracketsigW{
        \mathbf{\Psi}^{\word{i_1 \ldots i_n}}_{\word{j_1 \ldots j_n}}(\bell)
    }.
\end{equation*}

The class \eqref{eq:sig_polynom} can be viewed as a class of non-commutative path-dependent polynomials of Brownian motion. It is important for applications due to its approximation properties (see, for instance, \citep{levin2016learningpastpredictingstatistics} or \citep[Subsection~2.2]{cuchiero2022theocalib} for the continuous semimartingale framework) and its ubiquitous role in stochastic Taylor expansions \citep{kloeden1992stochastic, benarous1989flots, AbiJaber2024a}.

\paragraph{Ornstein--Uhlenbeck semigroup and its generator.} We next consider the generalized Ornstein--Uhlenbeck semigroup, which is a family of operators acting on
\( L^{2}(\Omega, \mathcal{F}_T^W, \mathbb{P}) \).
Note that any random variable \( F \in L^{2}(\Omega, \mathcal{F}_T^W, \mathbb{P}) \)
can be represented as a measurable functional of the Brownian motion, that is,
\( F = F(W) \). For $\bm\kappa = (\kappa_1, \ldots, \kappa_d) \in \R^d_+$, the corresponding Ornstein--Uhlenbeck semigroup is then defined by
\begin{equation}\label{eq:ou_semi_def}
    {\bf T}_\vartheta^{\bm{\kappa}} F(W) = \E[F( B^{\vartheta, \bm{\kappa}}) \mid \F^W_T],
\end{equation}
where
$$
    B^{\vartheta, \bm{\kappa}}_t
    := \left(e^{-\kappa_i\vartheta} W_t^i + \sqrt{1 - e^{-2\kappa_i\vartheta}} W_t^{\perp, i}\right)_{i = 1, \ldots, d},\quad t \in [0,T], \quad \vartheta \geq 0,
$$
and $W^{\perp}$ is a standard Brownian motion on $\R^d$ independent of $W$.
The infinitesimal generator of the semigroup $\left({\bf T}_\vartheta^{\bm\kappa}\right)_{\vartheta \ge 0}$ is defined by
\begin{equation}\label{eq:OU_gen_def}
    {\bf L}^{\bm{\kappa}} = \dfrac{d}{d\vartheta}\bigg|_{\vartheta=0}{\bf T}_\vartheta^{\bm{\kappa}}.
\end{equation}

Theorem~\ref{thm:ou_semigroup} shows that, for \( F(W) = \sigW \), 
the Ornstein--Uhlenbeck semigroup and its generator, viewed as linear operators acting on \( L^{2}(\Omega, \mathcal{F}_T^W, \mathbb{P}) \), 
can be identified with linear operators on the tensor algebra:
\begin{align}
    {\bf T}^{\bm\kappa}_\vartheta\sigW &= \exp\left(\sum_{i=1}^d\frac{1 -e^{-2\kappa_i\vartheta}}{2}{{\bf\Psi}}^{\word{0}}_{\word{ii}}\right) \left(\prod_{i=1}^d{\bf J}_{\word{i}}^{\kappa_i\vartheta}\right)\sigW,  \\
    {\bf L}^{\bm\kappa}\sigW &= \sum_{i=1}^d\kappa_i({{\bf\Psi}}^{\word{0}}_{\word{ii}} - {\bf \Lambda}_{\word{i}})\sigW,
\end{align}
and that ${\bf T}_\vartheta^{\bm\kappa}$ and ${\bf L}^{\bm\kappa}$ map $\mathcal{P}(\sigW)$ to itself:
\begin{align*}
    {\bf T}_\vartheta^{\bm\kappa}\bracketsigW{\bell} &= \bracketsigW{\left(\prod_{i=1}^d{\bf J}_{\word{i}}^{\kappa_i\vartheta} \right)\exp\left(\sum_{i=1}^d\frac{1 -e^{-2\kappa_i\vartheta}}{2}{{\bf\Psi}}^{\word{ii}}_{\word{0}}\right)(\bell)}, \\ 
    {\bf L}^{\bm\kappa}\bracketsigW{\bell} &= \bracketsigW{\sum_{i=1}^d\kappa_i({{\bf\Psi}}^{\word{ii}}_{\word{0}} - {\bf \Lambda}_{\word{i}})(\bell)}.
\end{align*}
Here, the diagonal operator ${\bf \Lambda}_{\word{i}}$ multiplies each coordinate $\bell^{\word{v}}$ by the number $|\word{v}|_{\word{i}}$ of occurrences of $\word{i}$ in $\word{v}$, and ${\bf J}_{\word{i}}^{\vartheta}$ multiplies the coordinates $\bell^{\word{v}}$ by $e^{-|\word{v}|_{\word{i}}\vartheta}$.

We also discuss the relation
\[
{\bf L}^i := {\bf L}^{e_i} = -\delta^i \D^i,
\]
between the generator, the Malliavin derivative, and the Skorokhod integral (defined in Section~\ref{sect:skorokhod_and_repr}).

\paragraph{Integration by parts.} Integration by parts is one of the key applications of Malliavin calculus. Typically, it establishes formulae of the form
\begin{equation}\label{eq:IBP_general}
        \E\left[f'(G)F\right] = 
        \E\left[
        f(G)\pi_T\right],
\end{equation}
which hold for $f \in C^1_b$ and random variables $F, G$ satisfying suitable integrability conditions. However, the Malliavin weight $\pi_T$ is often intractable and cannot be computed explicitly. To address this issue, we introduce an operation 
\[
\diamond: \TA[d+1] \times \TA[d+1] \to \TA[d+1],
\] 
which allows for the linearization of the $L^2$ scalar product between the Malliavin derivatives of random variables 
$F = \bracketsigW[T]{\bell}$ and $F' = \bracketsigW[T]{\bell'}$ from $\mathcal{P}(\sigW[T])$:
\begin{equation}
    \bracket{\D F}{\D F'}_{L^2(\Omega\times[0, T], \R^d)} = \bracketsigW[T]{\bell \diamond \bell'}.
\end{equation}
This leads to the integration by parts formula \eqref{eq:IBP_general} established in
Theorem~\ref{thm:ibp_sig} for the case
\begin{equation}
    G = \bracketsigW[T]{\bell^G}, \quad F = \frac{\bracketsigW[T]{\bell^F_1}}{\bracketsigW[T]{\bell^F_2}}, \quad \bell^G, \bell^F_1, \bell^F_2 \in \TA[d+1].
\end{equation} 
{Furthermore, we show that the Malliavin weight $\pi_T$ belongs to the same class as $F$ and provide  explicit and tractable formula for it, in contrast with the typically implicit weights arising in classical Malliavin calculus.}

Finally, we illustrate this result by providing formulas for the sensitivities of path-dependent options within the signature volatility framework of \citet{abijaber2024signature}. More precisely, we consider the class of payoffs representable as linear functionals of the log-price signature, expressing the Malliavin weight as a rational function of the Brownian signature.

\section{Properties of path signatures}\label{sect:preliminaries}

In this subsection, we present some properties of path signatures that will be used throughout this paper.

\paragraph{Signature dynamics.}  
It follows immediately from the definition \eqref{eq:sig_def} that for each $n \geq 1$, every coordinate of the signature satisfies
\begin{equation*}
    \sigX^{\word{i_1 \ldots i_n}}
    = \int_0^t \sigX[s]^{\word{i_1 \ldots i_{n-1}}} \circ dX_s^{i_n}.
\end{equation*}
In compact tensor form, this yields the signature dynamics
\begin{equation}\label{eq:sig_dynamics}
    d\sigX = \sigX \otimes \circ dX_t, 
    \qquad \sigX[0] = \emptyword.
\end{equation}

As already mentioned in Subsection~\ref{sect:prelim_first}, the inverse element $\sigX^{-1}$ with respect to the tensor product is the signature of the time-reversed path $\overleftarrow{X}^{[0, t]}\colon u \mapsto X_{t-u}$, $u \in [0,t]$. 
The next lemma provides the dynamics of this inverse signature.

\begin{proposition}
    The inverse signature $\sigX^{-1}$ is a continuous semimartingale (elementwise) satisfying
    \begin{equation}\label{eq:inv_sig_dyn}
        d\sigX^{-1} = -\, dX_t \otimes \sigX^{-1}, 
        \qquad \sigX[0]^{-1} = \emptyword.
    \end{equation}
\end{proposition}

\begin{proof}
    Differentiating the identity
    \[
        \sigX \otimes \sigX^{-1} = \emptyword
    \]
    and using the product rule yields
    \[
        d\sigX \otimes \sigX^{-1} + \sigX \otimes d\sigX^{-1} 
        = \sigX \otimes dX_t \otimes \sigX^{-1} + \sigX \otimes d\sigX^{-1} 
        = 0.
    \]
    Multiplying on the left by $\sigX^{-1}$ gives \eqref{eq:inv_sig_dyn}.
\end{proof}

\paragraph{Variation of constants for signatures.}
Since the group increment
\[
    \sigX[s,t] = \sigX[s]^{-1} \otimes\sigX[t]
\]
plays the role of the fundamental solution, it can be used to obtain a variation-of-constants formula for $\eTA$-valued linear SDEs of the form
\begin{equation}\label{eq:G_val_sde}
    d\sigY[t] = \sigY[t] \otimes \circ dX_t + d\bsigma_t.
\end{equation}
We interpret \eqref{eq:G_val_sde} elementwise: for any truncation order $N$, the truncated equation is a finite-dimensional linear Stratonovich SDE whose coefficients depend only on the first $N$ levels of $\sigY[t]$.
Thus, if each truncated system admits a unique strong solution, we say that \eqref{eq:G_val_sde} admits a unique strong solution.

\begin{proposition}\label{prop:variation_of_const}
    Let $X = (X_t)_{t \in [0,T]}$ be a continuous $\R^d$-valued semimartingale with signature $\sigX$.  
    Suppose that each component $\bsigma^{\word{v}}$, $\word{v} \in V$, is a continuous semimartingale.
    Then \eqref{eq:G_val_sde} admits a unique strong solution $\sigY$, and this solution admits the following representation:
    \begin{equation}\label{eq:sig_variation_of_const}
        \sigY[t] 
        = \sigY[s]\,\sigX[s,t] 
        + \int_s^t \circ d\bsigma_u \otimes \sigX[u,t].
    \end{equation}
\end{proposition}

\begin{proof}
    For any $\sigY$ satisfying \eqref{eq:G_val_sde}, observe that
    \begin{align*}
        d(\sigY[t] \otimes \sigX^{-1})
        &= d\sigY[t] \otimes \sigX^{-1}
           + \sigY[t] \otimes d\sigX^{-1} \\
        &= (\sigY[t] \otimes dX_t + d\bsigma_t) \otimes \sigX^{-1}
           - \sigY[t] \otimes (dX_t \otimes \sigX^{-1}) \\
        &= d\bsigma_t \otimes \sigX^{-1}.
    \end{align*}
    Hence,
    \[
        \sigY[t] \otimes \sigX[t]^{-1}
        = \sigY[s] \otimes \sigX[s]^{-1}
          + \int_s^t d\bsigma_u \otimes \sigX[u]^{-1}.
    \]
    Multiplying on the right by $\sigX[t]$ and using the identity $\sigX[u]^{-1}\otimes\sigX[t] = \sigX[u,t]$ yields \eqref{eq:sig_variation_of_const}. This proves uniqueness. Existence follows from the fact that the process $\sigY$ defined by \eqref{eq:sig_variation_of_const} satisfies \eqref{eq:G_val_sde}, which can be verified by direct differentiation.
\end{proof}

\paragraph{Shuffle property.}  
Another key feature of signatures is that products of linear functionals is itself a linear functional, as a consequence of the shuffle property.  
The shuffle product $\shuprod$ is a bilinear mapping $\TA\times\TA \to \TA$. For words $\word{v}, \word{w} \in V$ and letters $\word{i}, \word{j} \in A$, the shuffle product is defined recursively by
\[
    \word{v}\word{i} \shuprod \word{w}\word{j}
    = (\word{v}\word{i} \shuprod \word{w})\word{j}
      + (\word{v} \shuprod \word{w}\word{j})\word{i},
\]
with the base case $\word{v} \shuprod \emptyword = \emptyword \shuprod \word{v} = \word{v}$.  
Since $\shuprod$ is bilinear, the definition extends to $\TA$: for $\bell, \bp \in \TA$,
\[
    \bell \shuprod \bp
    := \sum_{\word{v},\word{w} \in V} \bell^{\word{v}} \bp^{\word{w}} (\word{v} \shuprod \word{w}).
\]
\begin{proposition}[Shuffle property]\label{prop:shuffle}
    Let $X$ be a continuous $\R^d$-valued semimartingale with signature $\sigX[]$.  
    Then for all $\bell, \bp \in \TA$ and for all $t \in [0, T]$,
    \[
        \bracketsigX{\bell}\, \bracketsigX{\bp}
        = \bracketsigX{\bell \shuprod \bp}.
    \]
\end{proposition}

\begin{proof}
    See \cite[Theorem 2.15]{lyons2007differential}.
\end{proof}

The shuffle product also linearizes certain entire functions.  
For $\bell \in \TA$, the {shuffle exponential} is defined by
\begin{equation}\label{eq:shuexp_def}
    \exp^{\shuprod}(\bell)
    := \sum_{n \geq 0} \frac{\bell^{\shuprod n}}{n!}
    \in \eTA,
\end{equation}
where $\bell^{\shuprod n}$ is the $n$-fold shuffle product.  
Using the decomposition
\begin{equation}\label{eq:shuexp_computation}
    \exp^{\shuprod}(\bell)
    = e^{\bell^\emptyword}
      \sum_{n \geq 0} \frac{(\bell - \bell^\emptyword \emptyword)^{\shuprod n}}{n!},
\end{equation}
with the convention $\bp^{\shuprod 0}=\emptyword$,
we see that each component contains only finitely many terms, and the shuffle exponential is well defined.  
Formula \eqref{eq:shuexp_computation} also yields a more efficient computational method.

\paragraph{Expected signature of Brownian motion.}  
The expected signature  
\[
    \widehat{\mathcal{E}}_t := \E[\sigW], \quad t \in [0, T],
\]
of the time-extended Brownian path $\widehat{W}\colon t \mapsto (t, W_t^1,\ldots,W_t^d)$ can be computed explicitly.  
By \citet{fawcett2003problems},
\begin{equation}\label{eq:brownian_esig}
    \widehat{\mathcal{E}}_t
    = \exp^\otimes\left(
        \word{0}
        + \frac{1}{2}\sum_{i=1}^d \word{ii}
    \right),
\end{equation}
where $\exp^\otimes(\bell) := \sum\limits_{n\ge 0} \dfrac{\bell^{\otimes n}}{n!}$ for $\bell \in \TA[d+1]$.

\section{Malliavin derivative of an Itô process signature} \label{sect:der_ito_process}

Let $a^0, a^1, \ldots, a^d$ be adapted $\R^m$-valued adapted square-integrable processes. Let us define the $m$-dimensional Itô process $X$ by
\begin{equation}\label{eq:ito_proc_def}
    dX_t = a_t^0\,dt + \sum_{k=1}^d a_t^k \circ dW_t^k = A_t \circ d \widehat{W}_t,
\end{equation}
where $(a^j_t)_{j=0, \ldots, d}$ form the columns of the matrix $A_t \in \R^{m \times (d+1)}$ and $\widehat{W}_t = (t, W_t)$. 
We first aim to establish the following formula for the Malliavin derivative of $X$: for $\mathbb{P} \otimes ds \otimes dt\text{-almost every } (\omega, s, t) \in \Omega \times [0,T]^2,$
\begin{equation}\label{eq:strato_mal-derivative}
    \D_s^i X_t = a_s^i\mathds{1}_{[0,t]}(s) + \int_s^t \D_s^i a_u^0\,du + \sum_{k=1}^d \int_s^t \D_s^i a_u^k \circ dW_u^k,
\end{equation}
similar to the corresponding formula in the Itô integral case \citep[Section~3.2.1]{alos2021malliavin}. Treating Stratonovich integrals requires stronger regularity assumptions on the integrands $a^k$ than in the Itô case.

\begin{assumption}\label{assump:Strato_reg_assump}
We assume that
\begin{enumerate}
    \item $a^0 \in \mathbb{L}^{1,2}$, that is, $a^0 \in L^2(\Omega\times[0, T])$ such that $a^0_t \in \mathbb{D}^{1,2}$ for a.e. $t \in [0, T]$, and 
    \[
    \| a^0 \|_{\mathbb{L}^{1,2}}^2 := \E\left[\int_0^T |a_t^0|^2 \, dt
        + \int_0^T \int_0^T |\D_s a^0_t|^2\, ds\, dt\right] < \infty.
    \]
    \item For $k = 1, \ldots, d$, $a^k \in \mathbb{L}^{2,2}$, that is, $a^k \in L^2(\Omega \times [0,T])$,  $a_t^k \in \mathbb{D}^{2,2}$ for a.e.\ $t \in [0,T]$, and
    \begin{equation}\label{eq:norm_condition}
        \| a^k \|_{\mathbb{L}^{2,2}}^2
        :=
        \E\left[\int_0^T |a_t^k|^2 \, dt
        + \int_0^T \int_0^T |\D_s a_t^k|^2 \, ds\,dt
        + \int_0^T \int_0^T \int_0^T |\D_r \D_s a_t^k|^2 \, dr\, ds\, dt\right]
        < \infty.
    \end{equation}

    \item For $k = 1, \ldots, d$, there exist processes $\D^{+} a^k, \D^{-} a^k \in L^2(\Omega \times [0,T])$ such that
    \begin{equation}\label{eq:strato_reg_assump}
    \begin{aligned}
    \lim_{n \to \infty} \int_0^T
    \sup_{\substack{s < u < T \\ |u-s| \le 1/n}}
    \E\!\left[ \big| \D_s a_u^k - (\D^{+} a^k)_s \big|^2 \right] ds &= 0, \\
    \lim_{n \to \infty} \int_0^T
    \sup_{\substack{0 < u < s \\ |s-u| \le 1/n}}
    \E\!\left[ \big| \D_s a_u^k - (\D^{-} a^k)_s \big|^2 \right] ds &= 0.
    \end{aligned}
    \end{equation}

    \item For $k = 1, \ldots, d$, the processes $\D^{+} a^k$ and $\D^{-} a^k$ belong to $\mathbb{L}^{1,2}$, and their Malliavin derivatives satisfy
    \begin{equation}\label{eq:strato_reg_assump_second}
    \begin{aligned}
    \lim_{n \to \infty} \int_0^T \int_0^T
    \sup_{\substack{s < u < T \\ |u-s| \le 1/n}}
    \E\!\left[
        \big| \D_s \D_r a_u^k - \D_r (\D^{+} a^k)_s \big|^2
    \right] ds\, dr &= 0, \\
    \lim_{n \to \infty} \int_0^T \int_0^T
    \sup_{\substack{0 < u < s \\ |s-u| \le 1/n}}
    \E\!\left[
        \big| \D_s \D_r a_u^k - \D_r (\D^{-} a^k)_s \big|^2
    \right] ds\, dr &= 0.
    \end{aligned}
    \end{equation}
\end{enumerate}
\end{assumption}

\begin{sqremark}\label{rmk:L_classes}
Condition \eqref{eq:strato_reg_assump_second} is a natural extension of the regularity class $\mathbb{L}^{1,2}_2$, that is, $\mathbb{L}^{1,2}$ together with the Malliavin derivative trace regularity condition \eqref{eq:strato_reg_assump}, which is required for the Stratonovich integrability of $a^k$ (see \citep[Section~3.1.1]{nualart2006malliavin}). In the present setting, we aim to prove not only the Stratonovich integrability of $a^k$, but also the  differentiability of the associated Stratonovich integral, and the well-posedness of the map
\[
(r,t) \mapsto \int_0^t \D_r a_u^k \circ dW_u^k
\]
as an $L^2(\Omega \times [0,T]^2)$-valued process rather than merely as a random variable. Accordingly, condition \eqref{eq:strato_reg_assump_second} arises naturally as a derivative trace condition for the integrand $\D_r a_t^k$. This extension of the results of \citet{nualart2006malliavin} will be used in Theorem~\ref{thm:malliavindersig} to derive the Malliavin derivatives of iterated Stratonovich integrals, for which a process-based point of view appears more natural.
\end{sqremark}

\begin{theorem}\label{thm:strato_mal_der}
Under Assumption~\ref{assump:Strato_reg_assump}, we have $X \in \mathbb{L}^{1,2}$, and \eqref{eq:strato_mal-derivative} holds.
\end{theorem}
\begin{proof}
The proof is given in Appendix~\ref{sect:strato_diff}.
\end{proof}

We next show that a formula analogous to \eqref{eq:strato_mal-derivative} also holds for the signature $\sigX[\cdot]$ of $X$.

\begin{theorem}\label{thm:malliavindersig}
    Let $X$ be an Itô process defined by \eqref{eq:ito_proc_def}. Assume that $a^{\sigX[]}_\cdot := (\sigX[\cdot] \otimes a^{i}_\cdot)_{i=0, 1, \ldots, d}$ satisfies  Assumption~\ref{assump:Strato_reg_assump} elementwise, that is, coordinatewise with respect to each $\word{v} \in V$.
    Then, $\sigX[] \in \mathbb{L}^{1,2}$ elementwise and   for $\mathbb{P} \otimes ds \otimes dt\text{-almost every } (\omega, s, t) \in \Omega \times [0,T]^2$ with $s\leq t$,
    \begin{equation}
        \D_s \sigX = \sigX[s] \otimes A_s \otimes \sigX[s, t]
        + \int_s^t \sigX[u] \otimes \big(\D_s A_u \circ d\widehat{W}_u\big) \otimes \sigX[u, t],
    \end{equation}
    or, in coordinate form, for $i = 1, \ldots, d$,
    \begin{equation}\label{eq:ito_proc_sig_der_coord}
        \D_s^i \sigX[t]
        = \sigX[s] \otimes a_s^i \otimes \sigX[s, t]
        + \sum_{k=0}^d \int_s^t 
        \big(\sigX[u] \otimes \D_s^i a_u^k \otimes \sigX[u, t]\big) \circ d\widehat{W}_u^k.
    \end{equation}
\end{theorem}

\begin{proof}
The dynamics of the signature $\sigX[\cdot]$ is given by \eqref{eq:sig_dynamics}, so that
\[
\sigX = \emptyword + \int_0^t \sigX[u] \otimes \circ dX_u
= \emptyword + \sum_{k=0}^d \int_0^t \sigX[u] \otimes a_u^k \circ d\widehat{W}_u^k.
\]
In particular $\sigX[\cdot]$ satisfies the assumptions from Theorem \ref{thm:strato_mal_der} elementwise and $\sigX[\cdot] \in \mathbb{L}^{1,2}$. Taking the Malliavin derivative $\D_s^i$ and applying the formula \eqref{eq:strato_mal-derivative} for the Malliavin derivative of Itô processes elementwise yields
\begin{align*}
    \D_s^i \sigX
    &= \sigX[s] \otimes a_s^i
    + \int_s^t \D_s^i \sigX[u] \otimes \circ dX_u
    + \sum_{k=0}^d \int_s^t \sigX[u] \otimes \D_s^i a_u^k \circ d\widehat{W}_u^k.
\end{align*}
Applying the variation-of-constants formula from Proposition~\ref{prop:variation_of_const} to $\sigY[t] = \D_s^i\sigX$ yields \eqref{eq:ito_proc_sig_der_coord}.
\end{proof}

\begin{sqremark}
     Theorem~\ref{thm:malliavindersig} is primarily a structural result: it identifies the form of $\D\sigX[]$ and connects it to the known expression \eqref{eq:strato_mal-derivative} for Itô processes, provided the derivative exists. However, it does not furnish tractable conditions ensuring this existence. Nevertheless, as demonstrated in Corollary~\ref{cor:brown_sig_def} below, the theorem applies to the time-augmented Brownian motion $\widehat{W}$, the main process of interest in this paper, since the integrability of the Brownian signature $\sigW[]$ and its derivatives can be verified directly.
\end{sqremark}

\begin{sqremark}
    In the expression~\eqref{eq:ito_proc_sig_der_coord}, the stochastic integrals are to be understood as standard Stratonovich integrals with predictable integrands, namely
    \begin{equation}\label{eq:stoch_int_correct}
        \left[
            \int_s^t 
            \big(\sigX[u] \otimes \D_s^i a_u^k \otimes \sigX[u]^{-1} \big)
            \circ d\widehat{W}_u^k
        \right] \otimes \sigX.
    \end{equation}
    However, we will often write expressions in which anticipative terms appear inside the integrals, as in \eqref{eq:ito_proc_sig_der_coord}. Such Stratonovich integrals with anticipative integrands can still be defined as limits of Riemann sums (see Appendix~\ref{sect:strato_diff} and \citet[Section~3.1]{nualart2006malliavin}).  Since the resulting integrals coincide, we will use~\eqref{eq:stoch_int_correct} and~\eqref{eq:ito_proc_sig_der_coord} interchangeably throughout the paper.
\end{sqremark}

\begin{corollary}\label{cor:brown_sig_def}
    If $X$ is a time-augmented Brownian motion $X_t = \widehat{W}_t$, then $a^k_t \equiv \word{k}$, so that $\D_s a^k_t \equiv 0$, and, for $\mathbb{P} \otimes ds \otimes dt\text{-almost every } (\omega, s, t) \in \Omega \times [0,T]^2$ with $s \leq t$,
    \begin{equation}\label{eq:malliavin_der_solution}
        \D_s^i \sig[t] = \sig[s] \otimes \word{i} \otimes \sig[s, t].
    \end{equation}
    Moreover, $\sigW \in \mathbb{D}^{\infty, 2}$ for a.e. $t \in [0, T]$, and,  for $\mathbb{P} \otimes dt^{\otimes (n+1)}\text{-almost every } (\omega, s_1, \ldots, s_n, t) \in \Omega \times [0,T]^{n+1}$ with $s_1 \leq \ldots \leq s_n  \leq t$, we have
    \begin{equation}\label{eq:higher_order_md}
        \D_{s_1}^{i_1} \cdots \D_{s_n}^{i_n} \sig[t]
        = \sig[s_1] \otimes \word{i_1} \otimes \sig[s_1, s_2]
          \otimes \word{i_2} \otimes \cdots \otimes \word{i_n} \otimes \sig[s_n, t].
    \end{equation}
\end{corollary}
\begin{proof}
    The formula \eqref{eq:higher_order_md} follows from a recursive application of Theorem~\ref{thm:malliavindersig}. For $n = 1$, \eqref{eq:malliavin_der_solution} is given directly by Theorem~\ref{thm:malliavindersig}. Suppose now that \eqref{eq:higher_order_md} holds for $n - 1$. Then, applying the product rule for the Malliavin derivative, we obtain
    \begin{align}
        \D_{s_1}^{i_1} \cdots \D_{s_n}^{i_n} \sig[t] &= \D_{s_1}^{i_1}(\sig[s_2] \otimes \word{i_2} \otimes \sig[s_2, s_3]
          \otimes \word{i_3} \otimes \cdots \otimes \word{i_n} \otimes \sig[s_n, t]) \\
          &= (\D_{s_1}^{i_1}\sig[s_2]) \otimes \word{i_2} \otimes \sig[s_2, s_3]
          \otimes \word{i_3} \otimes \cdots \otimes \word{i_n} \otimes \sig[s_n, t] \\
          &= \sig[s_1] \otimes \word{i_1} \otimes \sig[s_1, s_2]
          \otimes \word{i_2} \otimes \sig[s_2, s_3]
          \otimes \word{i_3} \otimes \cdots \otimes \word{i_n} \otimes \sig[s_n, t],
    \end{align}
    where the second equality follows from the fact that $\sigW[s_k, s_{k+1}]$ is measurable with respect to $\sigma((W_u - W_{s_k}), {u \in [s_k, s_{k+1}]})$, so that $\D_{s_1}\sigW[s_k, s_{k+1}] = 0$ for all $k \geq 2$; see \cite[Corollary~1.2.1]{nualart2006malliavin}.

    It remains to prove the integrability of the Malliavin derivatives.
    Indeed, for all $\word{v} \in V$,
    \begin{align*}
        \E\left[(\D_{s_1}^{i_1} \cdots \D_{s_n}^{i_n} \sig[t]^{\word{v}})^2\right]
        &=\E\left[\left((\sig[s_1] \otimes \word{i_1} \otimes \sig[s_1, s_2]
          \otimes \word{i_2} \otimes \cdots \otimes \word{i_n} \otimes \sig[s_n, t])^{\word{v}}\right)^2\right] \\
        &= \E\left[\left(\sum_{\substack{\word{w_0}, \ldots, \word{w_n} \\ \word{w_0i_1w_1 \ldots w_n i_n} = \word{v}}}\sigW[s_1]^{\word{w_0}}\sigW[s_1, s_2]^{\word{w_1}}\ldots\sigW[s_n, t]^{\word{w_n}}\right)^2\right] \\
        &\leq \E\left[\left(\sum_{\substack{\word{w_0}, \ldots, \word{w_n} \\ |\word{w_0}|, \ldots, |\word{w_n}| < |\word{v}|}}|\sigW[s_1]^{\word{w_0}}\sigW[s_1, s_2]^{\word{w_1}}\ldots\sigW[s_n, t]^{\word{w_n}}|\right)^2\right] \\
        &= \E\left[\left(\sum_{|\word{w_0}| < |\word{v}|}|\sigW[s_1]^{\word{w_0}}|\right)^2\right]\ldots \E\left[\left(\sum_{|\word{w_n}| < |\word{v}|}|\sigW[s_n, t]^{\word{w_n}}|\right)^2\right] \\
        &\leq \E\left[\left(\sum_{|\word{w}| < |\word{v}|}|\sigW[T]^{\word{w}}|\right)^2\right]^{n+1} < \infty,
    \end{align*}
    where, in the last inequality, we used the diffusive scaling of the signature $(\sigW^{\word{v}})_{\word{v} \in V} \overset{d}{=} \left(t^{\frac{|\word{v}| + |\word{v}|_{\word{0}}}{2}}\sigW[1]^{\word{v}}\right)_{\word{v} \in V}$, where $|\word{v}|_{\word{0}}$ denotes the number of occurrences of the letter $\word{0}$ in the word $\word{v}$. Therefore, $\E\left[(\D_{s_1}^{i_1} \cdots \D_{s_n}^{i_n} \sig[t]^{\word{v}})^2\right]$
    is bounded uniformly in $s_1, \ldots, s_n, t \in [0, T]$ and $i_1, \ldots, i_n \in \{1, \ldots, d\}$, so that $\sigW[t] \in \mathbb{D}^{n, 2}$ for all $n \geq 1$.
\end{proof}

\section{Malliavin calculus for $\sigW[]$}\label{sect:malliavin_sigW}

We now focus on the signature of the time-augmented Brownian motion $\widehat{W}\colon\, t \mapsto (t, W_t)$. This object plays the role of Brownian motion on path space in the sense of \citep[Definition 1.2]{Baudoin2004} and hence inherits many of its nice properties as we will show in this section.

\subsection{Iterated integrals of Malliavin derivatives of $\sig$}

We have already seen that the higher-order Malliavin derivative of the Brownian signature reduces to the ``insertion'' of letters in Chen's identity, see \eqref{eq:higher_order_md}. Our goal is to show that this operation can be, in some sense, inverted by integrating $\D_{s_1}^{i_1} \cdots \D_{s_n}^{i_n} \sigW[t]$ over the simplex $0 \le s_1 \le \cdots \le s_n \le t$ with respect to the components of $d\widehat{W}$. 

In particular, this will allow us to show in Theorem~\ref{thm:iter_int_of_mal_der} that if $F = \bracketsigW{\bell} \in \mathcal{P}(\sigW)$, recall \eqref{eq:sig_polynom}, then
\[
    \int_{0 \le s_1 \le \cdots \le s_n \le t} \D_{s_1}^{i_1}\cdots \D_{s_n}^{i_n} \left\langle \bell, \sigW \right\rangle \circ d\widehat{W}_{s_1}^{j_1} \cdots \circ d\widehat{W}_{s_n}^{j_n} \in \mathcal{P}(\sigW),
\]
and its coefficients can be expressed explicitly in terms of $\bell$. This property will be crucial for deriving closed-form formulae for the Ornstein--Uhlenbeck semigroup and integration by parts.

We start with the simplest case $\sig[s]\otimes\word{i}\otimes\sig[s,t]$, which can be written as
\begin{equation}\label{eq:product_inserted}
    \sig[s]\otimes\word{i}\otimes\sig[s, t] 
    = \sum_{\word{v_1}, \word{v_2} \in V} \sig[s]^{\word{v_1}} \sig[s,t]^{\word{v_2}} (\word{v_1}\word{i}\word{v_2}).
\end{equation}
The behavior of the products $\sig[s]^{\word{v_1}}\sig[s,t]^{\word{v_2}}$ under integration with respect to $s$ is described in the following lemma, which holds for general path signatures.

\begin{lemma}\label{lem:inserted_int}
    Let $X = (X_t)_{t \in [0,T]}$ be an $\R^d$-valued continuous semimartingale. Then, for all $\word{v_1}, \word{v_2} \in V$ and $j \in \{1,\ldots,d\}$, 
    \begin{equation}\label{eq:sig_prod_int}
        \int_0^t \sigX[s]^{\word{v_1}} \sigX[s,t]^{\word{v_2}} \circ dX_s^j = \sigX[t]^{\word{v_1 j v_2}}, \quad t \in [0,T].
    \end{equation}
    Consequently, for $\bell_1, \bell_2 \in \TA$,
    \begin{equation}
        \int_0^t \bracketsigX[s]{\bell_1} \bracketsigX[s,t]{\bell_2} \circ dX_s^j = \bracketsigX{\bell_1 \otimes\word{j}\otimes \bell_2}.
    \end{equation}
\end{lemma}

\begin{proof}
    For $\word{v_1}, \word{v_2} \in V$, set 
    \(\sigY^{\word{v_1 j v_2}} := \int_0^t \sigX[s]^{\word{v_1}} \sigX[s,t]^{\word{v_2}} \circ dX_s^j\). 
    We proceed by induction on $|\word{v_2}|$ to show that
    \begin{equation}\label{eq:YX_eq}
        \sigY^{\word{v_1 j v_2}} = \sigX^{\word{v_1 j v_2}}, \quad t \in [0, T].
    \end{equation}
    If $\word{v_2} = \emptyword$, \eqref{eq:YX_eq} holds by definition of the signature. Otherwise, suppose that $\word{v_2} = \word{v_2' k}$ for some $\word{k} \in A$, and that \eqref{eq:YX_eq} holds for $\sigY^{\word{v_1 j v_2'}}$. Then,
    \begin{align*}
        \sigY^{\word{v_1 j v_2}} 
        &= \int_0^t \sigX[s]^{\word{v_1}} \sigX[s,t]^{\word{v_2}} \circ dX_s^j 
        = \int_0^t \sigX[s]^{\word{v_1}} \left( \int_s^t \sigX[s,u]^{\word{v_2'}} \circ dX_u^k \right) \circ dX_s^j \\
        &= \int_0^t \left( \int_0^u \sigX[s]^{\word{v_1}} \sigX[s,u]^{\word{v_2'}} \circ dX_s^j \right) \circ dX_u^k 
        = \int_0^t \sigY[u]^{\word{v_1 j v_2'}} \circ dX_u^k \\
        &= \int_0^t \sigX[u]^{\word{v_1 j v_2'}} \circ dX_u^k 
        = \sigX[t]^{\word{v_1 j v_2}},
    \end{align*}
    where we applied Fubini's theorem and the induction hypothesis. 
\end{proof}

Integrating \eqref{eq:product_inserted} and applying Lemma~\ref{lem:inserted_int}, we obtain
\begin{equation*}
    \int_0^t (\sig[s] \otimes \word{i} \otimes \sig[s, t]) \circ d\widehat{W}_s^j 
    = \sum_{\word{v_1}, \word{v_2} \in V} \sig[t]^{\word{v_1 j v_2}} (\word{v_1} \word{i} \word{v_2}).
\end{equation*}
This motivates the following definition.

\begin{definition}\label{def:psi_def}
    Fix $n \in \N$ and words $\word{w_1}, \ldots, \word{w_n} \in V$ and $\word{u_1}, \ldots, \word{u_n} \in V$. We define the switching operator 
    ${\bf\Psi}_{\word{u_1}, \ldots, \word{u_n}}^{\word{w_1}, \ldots, \word{w_n}}: \eTA \to \eTA$ by
    \begin{equation}\label{eq:psi_def}
        {\bf\Psi}_{\word{u_1}, \ldots, \word{u_n}}^{\word{w_1}, \ldots, \word{w_n}} \colon 
        \mathbbm{x} \mapsto \sum_{\word{v_0}, \ldots, \word{v_n}} \mathbbm{x}^{\word{v_0 w_1 v_1 \ldots w_n v_n}} 
        (\word{v_0} \word{u_1} \word{v_1} \ldots \word{u_n} \word{v_n}).
    \end{equation}
\end{definition}

Note that ${\bf\Psi}_{\word{u_1}, \ldots, \word{u_n}}^{\word{w_1}, \ldots, \word{w_n}}$ can be equivalently defined by its action on the basis vectors:
\begin{equation}
    {\bf\Psi}_{\word{u_1}, \ldots, \word{u_n}}^{\word{w_1}, \ldots, \word{w_n}} (\word{v}) 
    = \sum_{\substack{\word{v_0}, \ldots, \word{v_n} \\ \word{v_0 w_1 v_1 \ldots w_n v_n} = \word{v}}} 
      \word{v_0 u_1 v_1 \ldots u_n v_n}, \quad \word{v} \in V.
\end{equation}

Hence, the switching operator ${\bf\Psi}_{\word{u_1}, \ldots, \word{u_n}}^{\word{w_1}, \ldots, \word{w_n}}$ admits a natural interpretation: it finds all possible sequences of {consecutive} and {non-intersecting} sub-strings 
$\word{w_1}, \ldots, \word{w_n}$ in the word $\word{v}$, and replaces them with $\word{u_1}, \ldots, \word{u_n}$. 
The result is the sum of all such modified words. In particular, ${\bf \Psi}^{\word{j}}_{\word{i}}$ replaces each occurrence of the letter $\word{j}$ by the letter $\word{i}$, one at a time. 
For example,
\[
    {\bf \Psi}^{\word{1}}_{\word{0}} (\word{01101}) = \word{00101} + \word{01001} + \word{01100}.
\]

The switching operator allows us to rewrite compactly the iterated integrals we are interested in.

\begin{theorem}\label{thm:iterated_psi}
    Let $X = (X_t)_{t \in [0, T]}$ be an $\R^d$-valued continuous semimartingale and let $\sigX[]$ denote its signature. Then,
    \begin{equation}\label{eq:iterated_md_int}
        \int_{0 \le s_1 \le \ldots \le s_n \le t} 
        \Big( \sigX[s_1] \otimes \word{i_1} \otimes \sigX[s_1, s_2] \otimes \ldots \otimes \word{i_n} \otimes \sigX[s_n, t] \Big) 
        \circ dX_{s_1}^{j_1} \ldots \circ dX_{s_n}^{j_n} 
        = {\bf\Psi}_{\word{i_1}, \ldots, \word{i_n}}^{\word{j_1}, \ldots, \word{j_n}} (\sigX[t]).
    \end{equation}
\end{theorem}

\begin{proof}
     It follows directly from the definition of the tensor product that
    \[
        \sigX[s_1] \otimes \word{i_1} \otimes \sigX[s_1, s_2] \otimes \ldots \otimes \word{i_n} \otimes \sigX[s_n, t]
        = \sum_{\word{v_0}, \ldots, \word{v_n} \in V} 
        \sigX[s_1]^{\word{v_0}} \sigX[s_1, s_2]^{\word{v_1}} \ldots \sigX[s_n, t]^{\word{v_n}} 
        (\word{v_0 i_1 v_1 \ldots i_n v_n}).
    \]
    Integrating and applying Lemma~\ref{lem:inserted_int} $n$ times concludes the proof.
\end{proof}

\begin{sqremark}
    Theorem~\ref{thm:iterated_psi} still applies when the order of the iterated integral differs from the number of ``holes'' in the ``pierced'' signature. More precisely, if the integral is taken only over a subset of consecutive variables $s_{k_0 + 1} \le \ldots \le s_{k_0 + m}$ for some $0 \le k_0 \le \ldots \le k_0 + m \le n$, we have
    \begin{align*}
        &\int_{s_{k_0} \le s_{k_0 + 1} \le \ldots \le s_{k_0 + m} \le s_{k_0 + m + 1}}
        \Big( \sigX[s_1] \otimes \word{i_1} \otimes \sigX[s_1, s_2] \otimes \ldots \otimes \word{i_n} \otimes \sigX[s_n, t] \Big) 
        \circ dX_{s_{k_0+1}}^{j_{k_0+1}} \ldots \circ dX_{s_{k_0 + m}}^{j_{k_0 + m}} \\
        &= \sigX[s_1] \otimes \word{i_1} \otimes \ldots \otimes \word{i_{k_0}} \otimes 
        {\bf \Psi}_{\word{i_{k_0 + 1}}, \ldots, \word{i_{k_0 + m}}}^{\word{j_{k_0 + 1}}, \ldots, \word{j_{k_0 + m}}} (\sigX[s_{k_0}, s_{k_0 + m + 1}]) 
        \otimes \word{i_{k_0 + m + 1}} \otimes \ldots \otimes \word{i_n} \otimes \sigX[s_n, t],
    \end{align*}
    where we set $s_0 = 0$ and $s_{n + 1} = t$.
\end{sqremark}

We will also be interested in computing the linear functionals of objects of the form \eqref{eq:iterated_md_int} and to write them as linear functionals of $\sigX$. In other words, we would like to compute the adjoint of ${\bf \Psi}_{\word{i_1, \ldots, i_n}}^{\word{j_1, \ldots, j_n}}$. The following proposition answers this question.

\begin{proposition}\label{prop:psi_adjoint}
    The adjoint operator $({\bf\Psi}_{\word{u_1}, \ldots, \word{u_n}}^{\word{w_1}, \ldots, \word{w_n}})^*$ is given by
    \begin{equation}
        ({\bf\Psi}_{\word{u_1}, \ldots, \word{u_n}}^{\word{w_1}, \ldots, \word{w_n}})^* = {\bf\Psi}^{\word{u_1}, \ldots, \word{u_n}}_{\word{w_1}, \ldots, \word{w_n}}.
    \end{equation}
    That is, for $\bell \in \TA$, we have
    \begin{equation*}
        \langle \bell,\, {\bf\Psi}_{\word{u_1}, \ldots, \word{u_n}}^{\word{w_1}, \ldots, \word{w_n}} (\sigX) \rangle 
        = \bracketsigX{ {\bf\Psi}^{\word{u_1}, \ldots, \word{u_n}}_{\word{w_1}, \ldots, \word{w_n}} (\bell) }.
    \end{equation*}
\end{proposition}

\begin{proof}
    The proof follows from observing that
    \begin{equation*}
        \langle \bell,\, {\bf\Psi}_{\word{u_1}, \ldots, \word{u_n}}^{\word{w_1}, \ldots, \word{w_n}} (\sigX) \rangle
        = \sum_{\word{v_0}, \ldots, \word{v_n}} \bell^{\word{v_0 u_1 v_1 \ldots u_n v_n}} \sigX^{\word{v_0 w_1 v_1 \ldots w_n v_n}} 
        = \bracketsigX{{\bf\Psi}^{\word{u_1}, \ldots, \word{u_n}}_{\word{w_1}, \ldots, \word{w_n}} (\bell) }.
    \end{equation*}
\end{proof}
Theorem~\ref{thm:iterated_psi} and Proposition~\ref{prop:psi_adjoint} imply immediately the following result.

\begin{theorem}\label{thm:iter_int_of_mal_der}
    For all $i_1, \ldots, i_n \in \{1, \ldots, d\}$ and $j_1, \ldots, j_n \in \{0, 1, \ldots, d\}$, we have
    \begin{equation*}
        \int_{0 \le s_1 \le \ldots \le s_n \le t} 
        (\D_{s_1}^{i_1} \ldots \D_{s_n}^{i_n} \sig) \circ d\widehat{W}_{s_1}^{j_1} \ldots \circ d\widehat{W}_{s_n}^{j_n} 
        = {\bf \Psi}_{\word{i_1, \ldots, i_n}}^{\word{j_1, \ldots, j_n}} (\sig).
    \end{equation*}
    Moreover, for all $\bell \in \TA[d+1]$,
    \begin{equation*}
        \int_{0 \le s_1 \le \ldots \le s_n \le t} 
        \D_{s_1}^{i_1} \ldots \D_{s_n}^{i_n} \langle \bell, \sig \rangle 
        \circ d\widehat{W}_{s_1}^{j_1} \ldots \circ d\widehat{W}_{s_n}^{j_n} 
        = \bracketsigW{ {\bf \Psi}^{\word{i_1 \ldots i_n}}_{\word{j_1 \ldots j_n}} (\bell) }.
    \end{equation*}
\end{theorem}

\subsection{Skorokhod integration and the representation of random variables in $\mathcal{P}(\sigW[T])$}\label{sect:skorokhod_and_repr}
We now illustrate how the derived results can be applied to obtain explicit representations of the (anticipative) Skorokhod integral, the Clark--Ocone formula, and the chaos expansion for random variables $F \in \mathcal{P}(\sigW[T])$.

\paragraph{Skorokhod integration.}

The Skorokhod integral $\delta$ is defined by the dual relationship
\begin{align}
\label{def:integ_skorohod_dual}
\mathbb{E}\left[\int_{0}^{T} \langle \mathbf{D}_{t}F, h_{t} \rangle_{\R^{d}}\, dt\right] = \mathbb{E}[F\,\delta(h)],
\end{align}
for $F \in \mathbb{D}^{1,2}$ and $h \in \mathrm{Dom}(\delta)$, where
\[
\mathrm{Dom}(\delta) := \left\{ h \in L^{2}(\Omega \times [0,T], \R^{d})\colon\ \left| \mathbb{E}\left[\int_{0}^{T} \langle \mathbf{D}_{t}F , h_{t} \rangle_{\R^{d}}\, dt \right] \right| \leqslant C \Vert F \Vert_{L^{2}(\Omega)},\ \forall F \in \mathbb{D}^{1,2}\right\}.
\]
Moreover, for $h=(h^{1},\ldots,h^{d}) \in \mathrm{Dom}(\delta)$ we may write
\begin{align*}
\delta(h)=\sum_{j=1}^{d} \delta^{j}(h^{j}),
\end{align*}
where, for each $j \in \{1,\ldots,d\}$ and all $F \in \mathbb{D}^{1,2}$,
\begin{align*}
\mathbb{E}\left[\int_{0}^{T} \mathbf{D}^{j}_{t} F \, h^{j}_{t}\, dt\right] = \mathbb{E}\left[F\, \delta^{j}(h^{j})\right].
\end{align*}

If $F \in \mathbb{D}^{1,2}$ and $h \in \mathrm{Dom}(\delta)$ are such that $Fh \in L_{2}(\Omega \times [0,T], \R^{d})$, then $Fh \in \mathrm{Dom}(\delta)$ and the following equality \citep[Proposition 1.3.3]{nualart2006malliavin} holds:
\begin{align}\label{eq:IBP_formula}
\delta(Fh) = F\,\delta(h) - \int_{0}^{T} \langle \mathbf{D}_{s}F , h_{s} \rangle_{\R^{d}}\, ds,
\end{align}
provided the right-hand side is square integrable. This equality plays an important role, in particular, for proving the integration by parts results discussed in Section~\ref{sect:ibp}.

We show that the Skorokhod integral of $\sigW[T]$, and hence of any random variable in $\mathcal{P}(\sigW[T])$, can be expressed using the switching operator ${\bf \Psi}^{\word{0}}_{\word{i}}$.

\begin{proposition}\label{Prop:skorokhod}
    For $i = 1, \ldots, d$, we have
    \begin{equation}\label{eq:skorokhod_sig}
        \delta^i(\sigW[T]) = W_T^i \sigW[T] - {\bf \Psi}^{\word{0}}_{\word{i}}(\sigW[T]).
    \end{equation}
    Moreover, if $\bell \in \TA[d+1]$, then
    \begin{equation}\label{eq:skorokhod_lin_func}
        \delta^i\left( \bracketsigW[T]{\bell} \right) = \bracketsigW[T]{\bell \shuprod \word{i} - {\bf \Psi}^{\word{i}}_{\word{0}}(\bell)}.
    \end{equation}
\end{proposition}
\begin{proof}
    Applying the integration-by-parts formula \eqref{eq:IBP_formula} element-wise with $F = \sigW[T]$ and $h_t \equiv 1$ yields
    \begin{align*}
         \delta^i(\sigW[T]) 
         &= W_T^i \sigW[T] - \int_0^T \D_t^i \sigW[T]\, dt \\
         &= W_T^i \sigW[T] - {\bf \Psi}^{\word{0}}_{\word{i}}(\sigW[T]),
    \end{align*}
    where we used Theorem~\ref{thm:iter_int_of_mal_der} to compute the integral.  
    Equation~\eqref{eq:skorokhod_lin_func} then follows from Proposition~\ref{prop:shuffle} and Proposition~\ref{prop:psi_adjoint}.
\end{proof}

\paragraph{Clark--Ocone formula.}
The classical Clark--Ocone formula refines the martingale representation theorem. It states that any $F \in \mathbb{D}^{1, 2}$ admits the representation
\begin{equation}\label{eq:clark-ocone}
    F = \E[F] + \sum_{i=1}^d \int_0^T \E[\D_t^i F \, | \, \F_t^W]\, dW_t^i.
\end{equation}
If the random variable $F$ lies in $\mathcal{P}(\sigW[T])$, the integrand in \eqref{eq:clark-ocone} can be computed explicitly.
\begin{proposition}
    Let $F = \bracketsig[T]{\bell}$ with $\bell \in \TA[d+1]$. Then, $F \in \mathbb{D}^{1, 2}$ and
    \begin{equation}\label{eq:clark-ocone_sig}
        F = \left\langle \bell, \widehat{\mathcal{E}}_{T} \right\rangle 
        + \sum_{i=1}^d \int_0^T 
        \left\langle \bell, \sig[t] \otimes \word{i} \otimes \widehat{\mathcal{E}}_{T - t} \right\rangle \, dW_t^i,
    \end{equation}    
    where $\widehat{\mathcal{E}}_{t}$ denotes the expected signature given by \eqref{eq:brownian_esig}.
\end{proposition}
\begin{proof}
Equation~\eqref{eq:malliavin_der_solution} yields
\begin{align}
    \E[\D_t^i F \, | \, \F_t^W] 
    = \E\left[
        \left\langle \bell, \sig[t] \otimes \word{i} \otimes \sig[t, T] \right\rangle 
        \Big| \F_t^W
      \right]
    = \left\langle 
        \bell, 
        \sig[t] \otimes \word{i} \otimes \widehat{\mathcal{E}}_{T - t}
      \right\rangle,
\end{align}
since $\sigW[t]$ is $\F_t^W$-measurable and $\sigW[t, T]$ is independent of $\F_t^W$, with $\sigW[t, T] \overset{d}{=} \sigW[T - t]$.
\end{proof}

\paragraph{Chaos expansion.}
The chaos expansion \citep[Theorem 1.1.2]{nualart2006malliavin} expresses any square-integrable random variable $F \in L^2(\Omega, \F^W, \P)$ as an infinite series of iterated Wiener--Itô integrals:
$$
F = \E[F] + \sum_{n \geq 1} n! 
    \sum_{\word{i_1 \ldots i_n} \in V_n}\ \,
    \idotsint\limits_{0 < s_1 < \ldots < s_n < t}
    f_{\word{i_1\ldots i_n}}(s_1, \ldots, s_n)\,
    dW_{s_1}^{i_1} \cdots dW_{s_n}^{i_n}.
$$
Moreover, when $F$ is infinitely Malliavin differentiable (i.e.\ $F \in \mathbb{D}^{\infty,2}$), the kernel functions $f_{\word{i_1\ldots i_n}}$ can be computed \citep[Exercise 1.2.6]{nualart2006malliavin} via
$$
f_{\word{i_1\ldots i_n}}(s_1, \ldots, s_n)
:= \frac{1}{n!}\,
\mathbb{E}\!\left[
    \D_{s_1}^{i_1} \cdots \D_{s_n}^{i_n} F
\right],
\qquad 
\word{i_1\ldots i_n} \in V \setminus \{ \word{0} \}.
$$
 
It follows from Corollary~\ref{cor:brown_sig_def} that $\mathcal{P}(\sigW[T]) \subset \mathbb{D}^{\infty, 2}$. 
The following proposition shows that for $F \in \mathcal{P}(\sigW[T])$, the kernel functions can be expressed explicitly.
\begin{proposition}\label{Prop:Chaos}
    Let $F = \bracketsig[T]{\bell}$ with $\bell \in \TA$. Then $F \in \mathbb{D}^{\infty,2}$ and the kernel functions $f_{\word{i_1\ldots i_n}}$ in the chaos expansion are given by
    \begin{equation}
        f_{\word{i_1\ldots i_n}}(s_1, \ldots, s_n)
        = \frac{1}{n!}
        \left\langle
            \bell,\,
            \widehat{\mathcal{E}}_{s_1} 
            \otimes \word{i_1}
            \otimes \widehat{\mathcal{E}}_{s_2 - s_1}
            \otimes \word{i_2}
            \otimes \cdots
            \otimes \word{i_n}
            \otimes \widehat{\mathcal{E}}_{s_n - s_{n-1}}
        \right\rangle.
    \end{equation}
\end{proposition}
 
In particular, recalling that $|\word{v}|_{\word{i}}$ denotes the number of occurrences of the letter $\word{i}$ in $\word{v} \in V$, we have
$$
f_{\word{i_1\ldots i_n}}(s_1, \ldots, s_n) \equiv 0, \quad n  >  N = \sup\left\{\sum_{i=1}^d|\word{v}|_{\word{i}}\colon\ \word{v} \in V, \ \bell^{\word{v}} \neq 0\right\},
$$
so that $F$ has a finite chaos expansion of order $N$. 
\begin{proof}
    The result follows from \eqref{eq:higher_order_md} and the independence of 
    $\sigW[s_1],\ \sigW[s_1, s_2],\ \ldots,\ \sigW[s_{n-1}, s_n]$.
\end{proof}

\subsection{The Ornstein--Uhlenbeck semigroup and its generator}

In this section we establish explicit formulas for the action of the Ornstein--Uhlenbeck semigroup and its generator on the Brownian signature.
The Ornstein--Uhlenbeck semigroup ${\bf T}_\vartheta$ (see, for instance, \citet[Section 1.4]{nualart2006malliavin}) and its generator are usually introduced using the splitting of the Brownian motion
    $$
    B^{\vartheta} := e^{-\vartheta} W + \sqrt{1 - e^{-2\vartheta}} W^\perp, \quad \vartheta \geq 0,
    $$
    where $W^\perp$ is an independent copy of $W$. The semigroup is then defined by
    \begin{equation*}
        {\bf T}_\vartheta F(W) = \E[F( B^{\vartheta}) \mid \F^W_T], \quad \vartheta \geq 0,
    \end{equation*}
    and its generator is given by
    \begin{equation}\label{eq:gen_dcmp}
        {\bf L} = \dfrac{d}{d\vartheta}\bigg|_{\vartheta=0}{\bf T}_\vartheta= -\sum_{i=1}^d \delta^i \D^i,
    \end{equation}
    where the last identity follows from \citep[Proposition~1.4.3]{nualart2006malliavin}.
    
    However, for later applications to integration-by-parts formulae, it may be convenient to split to all components of $W$, but only some of them. This leads to the following generalization of ${\bf T}_\vartheta$.

\begin{definition}
    Fix $\bm{\kappa} = (\kappa_1, \ldots, \kappa_d) \in \R_+^d$ and set 
    \[
    B^{\vartheta, \bm{\kappa}}_t
    := \left(e^{-\kappa_i\vartheta} W_t^i + \sqrt{1 - e^{-2\kappa_i\vartheta}} W_t^{\perp, i}\right)_{i = 1, \ldots, d},\quad t \in [0,T], \quad \vartheta \geq 0.
    \]
    The $\bm{\kappa}$-Ornstein--Uhlenbeck semigroup and its generator are defined by
        \begin{equation}\label{eq:OU_semi_std}
        {\bf T}_\vartheta^{\bm{\kappa}} F(W) = \E[F( B^{\vartheta, \bm{\kappa}}) \mid \F^W_T], \qquad 
        {\bf L}^{\bm{\kappa}} = \dfrac{d}{d\vartheta}\bigg|_{\vartheta=0}{\bf T}_\vartheta^{\bm{\kappa}}.
    \end{equation}
\end{definition}

In particular, when $\bm{\kappa} = (1, \ldots, 1)$, the definition coincides with the standard Ornstein--Uhlenbeck semigroup \eqref{eq:OU_semi_std}. When $\bm{\kappa} = e_i \in \R^d_+$ for some $i = 1, \ldots, d$, this corresponds to splitting only the $i$-th component of $W$. We will denote the corresponding semigroup and generator by $ {\bf T}^i_\vartheta$ and ${\bf L}^{i}$.
Note that the operators \( \mathbf{T}_\vartheta^i \) and \( \mathbf{L}^i \),
    can be viewed as the standard Ornstein--Uhlenbeck semigroup and generator associated with the
    one-dimensional Brownian motion \( W^i \), conditional on the remaining components
    \( W^1, \ldots, W^{i-1}, W^{i+1}, \ldots, W^d \). In particular, this implies that 
    $$
    \mathbf{L}^i = - \delta^i\D^i,
    $$
    which leads to the linear decomposition of the generator:
    $$
    \mathbf{L} = \sum_{i=1}^d\mathbf{L}^i.
    $$

\begin{sqremark}
    Note, however, that the splitting of the Ornstein--Uhlenbeck semigroup itself is not linear:
    \[
    \mathbf{T}_\vartheta \neq \sum_{i=1}^d \mathbf{T}_\vartheta^i.
    \]
    For instance, one can easily verify that, for \( F = W_T^1 W_T^2 \),
    \[
    \mathbf{T}_\vartheta \bigl(W_T^1 W_T^2\bigr)
    = e^{-2\vartheta} W_T^1 W_T^2
    \neq 2 e^{-\vartheta} W_T^1 W_T^2
    = (\mathbf{T}_\vartheta^1 + \mathbf{T}_\vartheta^2)\bigl(W_T^1 W_T^2\bigr).
    \]
\end{sqremark}

Our goal is to compute ${\bf T}_\vartheta^{\bm{\kappa}} F(W)$ and ${\bf L}^{\bm{\kappa}} F(W)$ in the special case $F(W) = \sigW$. We recall that in this case ${\bf T}_\vartheta^{\bm{\kappa}} \sigW$ and ${\bf L}^{\bm{\kappa}} \sigW$ are understood as elements of $\eTA$, with the operators applied component-wise:
$$
{\bf T}_\vartheta^{\bm{\kappa}}\sigW 
= \left({\bf T}_\vartheta^{\bm{\kappa}} \sigW^{\word{v}}\right)_{\word{v} \in V}, \qquad {\bf L}^{\bm{\kappa}} \sigW 
= \left({\bf L}^{\bm{\kappa}} \sigW^{\word{v}}\right)_{\word{v} \in V}.
$$
\begin{sqremark}
    The domain of \( \mathbf{L} \) coincides with the space $\mathbb{D}^{2, 2}$ of random variables that are twice Malliavin differentiable \citep[Section~1.4.2]{nualart2006malliavin}. 
    Therefore, \( \mathbf{L} \sigW \) is well-defined, since
    \( \sigW \in \mathbb{D}^{\infty,2} \) by Corollary~\ref{cor:brown_sig_def}. We will show that
    $$
    {\bf L}^{\bm{\kappa}} = -\sum_{i = 1}^d \kappa_i \delta^i\D^i,
    $$
    so that $\mathrm{Dom}({\bf L}^{\bm{\kappa}}) = \mathrm{Dom}({\bf L})$ and ${\bf L}^{\bm{\kappa}}\sigW$ is well-defined as well.
\end{sqremark}

We introduce two linear operators on \( \eTA[d+1] \) that will be useful for computing
\( \mathbf{T}^{\bm{\kappa}}_\vartheta \sigW \) and \( \mathbf{L}^{\bm{\kappa}} \sigW \).

\begin{definition}
    Fix $\vartheta \geq 0$ and a letter $\word{i} \in A$. We define the operators 
    ${\bf J}_{\word{i}}^\vartheta \colon \eTA[d+1] \to \eTA[d+1]$ 
    and 
    ${\bf \Lambda}_{\word{i}} \colon \eTA[d+1] \to \eTA[d+1]$ 
    by
\begin{align*}
    &{\bf J}_{\word{i}}^\vartheta\colon
    \sum_{n \ge 0}\sum_{|\word{v}| = n} \bell^{\word{v}} \word{v}
    \mapsto 
    \sum_{n \ge 0}\sum_{|\word{v}| = n} 
        e^{-|\word{v}|_{\word{i}}\,\vartheta}\,
        \bell^{\word{v}} \word{v}, 
    \\
    &{\bf \Lambda}_{\word{i}}\colon
    \sum_{n \ge 0}\sum_{|\word{v}| = n} \bell^{\word{v}} \word{v}
    \mapsto 
    \sum_{n \ge 0}\sum_{|\word{v}| = n}
        |\word{v}|_{\word{i}}\,
        \bell^{\word{v}} \word{v},
\end{align*}
where $|\word{v}|_{\word{i}}$ denotes the number of occurrences of the letter $\word{i}$ in the word $\word{v}$.
\end{definition}

We note that both operators are diagonal, and they satisfy
$
    {\bf \Lambda}_{\word{i}}
    = -\dfrac{d}{d\vartheta} {\bf J}_{\word{i}}^{\vartheta}\big|_{\vartheta = 0}.
$
 
\begin{theorem}\label{thm:ou_semigroup}
The $\bm{\kappa}$-Ornstein--Uhlenbeck semigroup operator of the Brownian signature is given by
\begin{equation}\label{eq:ou_semi_op}
    {\bf T}^{\bm{\kappa}}_\vartheta\sigW 
    = \exp\!\left(\sum_{i=1}^d\frac{1 - e^{-\kappa_i\vartheta}}{2}\,{{\bf\Psi}}^{\word{0}}_{\word{ii}}\right)
      \left(\prod_{i=1}^d{\bf J}_{\word{i}}^{\kappa_i\vartheta}\right)\sigW.
\end{equation}
Moreover, we have $\sigW \in \mathrm{Dom}({\bf L}^{\bm{\kappa}})$ and
\begin{equation}\label{eq:ou_semi_gen}
    {\bf L}^{\bm{\kappa}}\sigW = 
    \sum_{i=1}^d\kappa_i\bigl({{\bf\Psi}}^{\word{0}}_{\word{ii}} - {\bf \Lambda}_{\word{i}}\bigr)\sigW.
\end{equation}
Their adjoint operators are
\begin{equation}\label{eq:ou_semi_adj}
    ({\bf T}_\vartheta^{\bm{\kappa}})^* 
    = \left(\prod_{i=1}^d{\bf J}_{\word{i}}^{\kappa_i\vartheta} \right)
      \exp\!\left(\frac{1 - e^{-2\kappa_i\vartheta}}{2}\,{{\bf\Psi}}^{\word{ii}}_{\word{0}}\right),
    \qquad 
    ({\bf L}^{\bm{\kappa}})^* 
    = \sum_{i=1}^d\kappa_i({{\bf\Psi}}^{\word{ii}}_{\word{0}} - {\bf \Lambda}_{\word{i}}).
\end{equation}
\end{theorem}

\begin{proof}
\textbf{Step 1. Semigroup ${\bf T}_\vartheta^{\bm{\kappa}}$.}
We denote the time-augmented signature of $B^{\vartheta, \bm{\kappa}}$ by \( \sigB[]^{\vartheta, \bm\kappa} \) and write the dynamics equation \eqref{eq:sig_dynamics} for it:
\begin{equation}\label{eq:sig_B_theta_dyn}
    d \sigB[t]^{\vartheta, \bm\kappa} 
    = (\sigB[t]^{\vartheta, \bm\kappa} \otimes \word{0}) \, dt
      + \sum_{i=1}^d(\sigB[t]^{\vartheta, \bm\kappa} \otimes \word{i}) \circ (e^{-\kappa_i\vartheta} \, dW_t^i 
      + \sqrt{1 - e^{-2\kappa_i\vartheta}} \, dW_t^{\perp, i}).
\end{equation}
Writing the dynamics of \( (\sigB[t]^{\vartheta, \bm\kappa} \otimes \word{i}) \) in a similar way, one can easily verify that 
\[
d \bigl\langle \sigB[]^{\vartheta, \bm\kappa} \otimes \word{i},\, \sqrt{1 - e^{-2\kappa_i\vartheta}} W^{\perp, i} \bigr\rangle_t 
= ({1 - e^{-2\kappa_i\vartheta}}) (\sigB[t]^{\vartheta, \bm\kappa} \otimes \word{ii}) \, dt, \quad i = 1, \ldots, d.
\]
Using the relationship between Itô and Stratonovich integrals for 
\( (\sigB[t]^{\vartheta, \bm\kappa} \otimes \word{i}) \circ \sqrt{1 - e^{-2\kappa_i\vartheta}} \, dW_t^{\perp, i} \) in \eqref{eq:sig_B_theta_dyn} 
and rearranging the terms, we obtain
\[
d \sigB[t]^{\vartheta, \bm\kappa} 
= \sigB[t]^{\vartheta, \bm\kappa}\otimes\!\left(\word{0} 
   + \sum_{i=1}^d\frac{{1 - e^{-2\kappa_i\vartheta}}}{2}\word{ii}\right) dt
   + \sum_{i=1}^d(\sigB[t]^{\vartheta, \bm\kappa}\otimes\word{i}) \circ e^{-\kappa_i\vartheta}\, dW_t^i 
   + \sum_{i=1}^d(\sigB[t]^{\vartheta, \bm\kappa}\otimes\word{i})\,\sqrt{1 - e^{-2\kappa_i\vartheta}}\, dW_t^{\perp, i}.
\]
Taking the conditional expectation $\E[\,\cdot\, | \F^W_T]$ yields a linear SDE for 
\(
{\bf T}^{\bm\kappa}_\vartheta\sigW = \E[\sigB[t]^{\vartheta, \bm\kappa}\, | \F^W_T]:
\)
\[
    d {\bf T}^{\bm\kappa}_\vartheta\sigW
    = {\bf T}^{\bm\kappa}_\vartheta\sigW\otimes\!\left(\word{0} 
      + \sum_{i=1}^d\frac{1 - e^{-2\kappa_i\vartheta}}{2}\word{ii}\right) dt
      + \sum_{i=1}^d{\bf T}^{\bm\kappa}_\vartheta\sigW\otimes(e^{-\kappa_i\vartheta} \word{i}) \circ dW_t^i
      .
\]
Applying Picard iteration, one obtains a solution linear in $\sig$:
\[
    {\bf T}^{\bm\kappa}_\vartheta\sigW 
    = \sum_{\word{v} = \word{i_1 \ldots i_n}}
      \bigl(\bm{u}^{\vartheta, {\bm\kappa}}(\word{i_1}) \otimes \cdots \otimes \bm{u}^{\vartheta, {\bm\kappa}}(\word{i_n})\bigr)
      \sigW^{\word{i_1 \ldots i_n}},
\]
where  
\begin{equation}\label{eq:u_formula}
    \bm{u}^{\vartheta, {\bm\kappa}}(\word{0}) 
= \word{0} + \sum_{i=1}^d\frac{1 - e^{-\kappa_i\vartheta}}{2}\word{ii}, 
\qquad 
\bm{u}^{\vartheta, {\bm\kappa}}(\word{i}) = e^{-\kappa_i\vartheta}\word{i},
\qquad i = 1, \ldots, d.
\end{equation}

Hence, we can define an operator ${\bf T}^{\bm\kappa}_\vartheta\colon \eTA[d+1] \to \eTA[d+1]$ by
\[
{\bf T}^{\bm\kappa}_\vartheta(\word{i_1\ldots i_n}) 
= \bigotimes_{j=1}^n \bm{u}^{\vartheta, {\bm\kappa}}(\word{i_j}),
\]
so that ${\bf T}^{\bm\kappa}_\vartheta(\sigW)$ coincides with the action of the Ornstein--Uhlenbeck semigroup.
 
Since for all $\word{j} \in \{\word{0},\word{1},\ldots,\word{d}\}$,
$$
 {\bf\Psi}^{\word{0}}_{\word{ii}}(\word{j}) = \indic{j = 0}\word{ii}, \qquad  {\bf J}_{\word{i}}^{\kappa_i\vartheta}(\word{j}) = \begin{cases}
     e^{-\kappa_i\vartheta}\word{j}, &j = i, \\ \word{j}, & j \neq i,
\end{cases} 
$$
and
$$
\exp\!\left(\sum_{i=1}^d\frac{1 - e^{-\kappa_i\vartheta}}{2}{{\bf\Psi}}^{\word{0}}_{\word{ii}}\right)(\word{j}) = \begin{cases}
    \word{0} + \sum_{i=1}^d\frac{1 - e^{-\kappa_i\vartheta}}{2}\word{ii}, & j = 0, \\
    \word{j}, &j \neq 0,
\end{cases}
$$
equation \eqref{eq:u_formula} reads
\[
\bm{u}^{\vartheta, {\bm\kappa}}(\word{j}) = \exp\!\left(\sum_{i=1}^d\frac{1 - e^{-\kappa_i\vartheta}}{2}{{\bf\Psi}}^{\word{0}}_{\word{ii}}\right)
\left(\prod_{i=1}^d{\bf J}_{\word{i}}^{\kappa_i\vartheta}\right)(\word{j}) = {\bf T}^{\bm\kappa}_\vartheta(\word{j}),
\qquad \word{j} \in \{\word{0},\word{1},\ldots,\word{d}\}.
\]

Moreover, each ${{\bf\Psi}}^{\word{0}}_{\word{ii}}$ satisfies the Leibniz rule,
\[
{{\bf\Psi}}^{\word{0}}_{\word{ii}}(\word{u}\otimes\word{v})
= {{\bf\Psi}}^{\word{0}}_{\word{ii}}(\word{u})\otimes\word{v}
  + \word{u}\otimes{{\bf\Psi}}^{\word{0}}_{\word{ii}}(\word{v}),
\]
and the same holds for $\sum_{i=1}^d\frac{1 - e^{-\kappa_i\vartheta}}{2}{{\bf\Psi}}^{\word{0}}_{\word{ii}}$, so that
its exponential is a tensor algebra homomorphism:
\[
\exp\left(\sum_{i=1}^d\frac{1 - e^{-\kappa_i\vartheta}}{2}{{\bf\Psi}}^{\word{0}}_{\word{ii}}\right)(\word{u}\otimes\word{v})
= \exp\left(\sum_{i=1}^d\frac{1 - e^{-\kappa_i\vartheta}}{2}{{\bf\Psi}}^{\word{0}}_{\word{ii}}\right)(\word{u})
  \otimes
  \exp\left(\sum_{i=1}^d\frac{1 - e^{-\kappa_i\vartheta}}{2}{{\bf\Psi}}^{\word{0}}_{\word{ii}}\right)(\word{v}).
\]

Moreover, for all $i = 1, \ldots, d$, ${\bf J}_{\word{i}}^{\kappa_i\vartheta}$ is a homomorphism as well.  
Thus, the identity
\[
{\bf T}^{\bm\kappa}_\vartheta
= \exp\!\left(\sum_{i=1}^d\frac{1 - e^{-\kappa_i\vartheta}}{2}{{\bf\Psi}}^{\word{0}}_{\word{ii}}\right)
  \left(\prod_{i=1}^d{\bf J}_{\word{i}}^{\kappa_i\vartheta}\right)
\]
is proved, yielding \eqref{eq:ou_semi_op}.

\textbf{Step 2. Generator ${\bf L}^{\bm\kappa}$.} 
By definition,
\[
{\bf L}^{\bm\kappa} 
= \left[
\frac{d}{d\vartheta}
\left(\exp\!\left(\sum_{i=1}^d\frac{1 - e^{-2\kappa_i\vartheta}}{2}{{\bf\Psi}}^{\word{0}}_{\word{ii}} \right)\right)
\prod_{i=1}^d{\bf J}_{\word{i}}^{\kappa_i\vartheta}
+ 
\sum_{j=1}^d\exp\!\left(\sum_{i=1}^d\frac{1 - e^{-2\kappa_i\vartheta}}{2}{{\bf\Psi}}^{\word{0}}_{\word{ii}} \right)
\left(\frac{d}{d\vartheta}{\bf J}_{\word{j}}^{\kappa_j\vartheta}\right)\prod_{k\neq j}{\bf J}_{\word{k}}^{\kappa_k\vartheta}
\right]_{\vartheta=0}.
\]

A direct computation gives
\[
\frac{d}{d\vartheta}
\exp\!\left(\sum_{i=1}^d\frac{1 - e^{-2\kappa_i\vartheta}}{2}{{\bf\Psi}}^{\word{0}}_{\word{ii}}\right)\Bigg|_{\vartheta=0}
=
\sum_{i=1}^d\kappa_i{{\bf\Psi}}^{\word{0}}_{\word{ii}},
\qquad 
{\bf J}_{\word{i}}^{\kappa_i\vartheta}\big|_{\vartheta=0}
= \exp\!\left(\sum_{i=1}^d\frac{1 - e^{-2\kappa_i\vartheta}}{2}{{\bf\Psi}}^{\word{0}}_{\word{ii}}\right)\Bigg|_{\vartheta=0} = \mathrm{Id}.
\]
Recall that $ \frac{d}{d\vartheta}{\bf J}_{\word{j}}^{\kappa_j\vartheta}\Big|_{\vartheta = 0} = -\kappa_j {\bf \Lambda_{\word{j}}}$
so \eqref{eq:ou_semi_gen} follows.

Finally, the adjoint formulas \eqref{eq:ou_semi_adj} follow from Proposition~\ref{prop:psi_adjoint} and the fact that ${\bf J}_{\word{i}}^{\kappa_i\vartheta}$ and ${\bf \Lambda}_{\word{i}}$ are diagonal operators.
\end{proof}

\begin{corollary}
    The standard Ornstein--Uhlenbeck semigroup ${\bf T}_\vartheta$ and generator ${\bf L}$ are given by 
    \begin{equation}
        {\bf T}_\vartheta\sigW 
        = \exp\!\left(\frac{1 - e^{-2\vartheta}}{2}\,\sum_{i=1}^d{{\bf\Psi}}^{\word{0}}_{\word{ii}}\right)
          \left(\prod_{i=1}^d {\bf J}_{\word{i}}^{\vartheta}\right)\sigW, \qquad     {\bf L}\sigW = 
    \sum_{i=1}^d\bigl({{\bf\Psi}}^{\word{0}}_{\word{ii}} - {\bf \Lambda}_{\word{i}}\bigr)\sigW.
    \end{equation}
    The adjoint operators are given by
    \begin{equation}\label{eq:statndard_ou_dual}
    {\bf T}_\vartheta^* 
    = \left(\prod_{i=1}^d {\bf J}_{\word{i}}^{\vartheta}\right)
      \exp\!\left(\frac{1 - e^{-2\vartheta}}{2}\,\sum_{i=1}^d{{\bf\Psi}}^{\word{ii}}_{\word{0}}\right) \qquad     {\bf L}^* = 
    \sum_{i=1}^d\bigl({{\bf\Psi}}^{\word{ii}}_{\word{0}} - {\bf \Lambda}_{\word{i}}\bigr).
    \end{equation}
\end{corollary}
\begin{remark}
The semigroup ${\mathbf T}_\vartheta$ can equivalently be defined in terms of chaos expansions; see \citep[Definition 1.4.1]{nualart2006malliavin}. Namely, if $F \in L^2(\Omega, \F^W, \P)$, then
$
F = \sum_{n \ge 0} J_n F,
$
where $J_n F$ denotes the projection of $F$ onto the $n$-th Wiener chaos. The semigroup then acts as
\[
{\mathbf T}_\vartheta F = \sum_{n \ge 0} e^{-n\vartheta} J_n F.
\]
Using the recursive definition of the Wiener chaos and the It\^{o} signature, a simple induction argument based on Fubini's theorem shows that $\widehat{\mathbb{W}}_t^{\mathrm{Ito}, \word{v}}$ belongs to the Wiener chaos of order $\sum_{i=1}^d |\word{v}|_i$.
\[
{\mathbf T}_\vartheta\bracket{\bell}{\widehat{\mathbb{W}}_t^{\mathrm{Ito}}}
=
\bracket{\prod_{i=1}^d {\bf J}_{\word{i}}^{\vartheta}\bell}{\widehat{\mathbb{W}}_t^{\mathrm{Ito}}}.
\]
Comparing this with \eqref{eq:statndard_ou_dual}, we observe that the operator
\(
\exp\!\left(\frac{1 - e^{-2\vartheta}}{2}\sum_{i=1}^d {{\bf\Psi}}^{\word{ii}}_{\word{0}}\right)
\)
is related to the conversion of the Stratonovich signature into the It\^o signature. Indeed, by writing the dynamics of $\sigW[\cdot]$ in terms of It\^o integrals and proceeding as in the proof of Theorem~\ref{thm:ou_semigroup}, one can show that
\[
\bracket{\bell}{\widehat{\mathbb{W}}_t^{\mathrm{Ito}}}
=
\bracket{\exp\!\left(\frac{1}{2}\sum_{i=1}^d {{\bf\Psi}}^{\word{ii}}_{\word{0}}\right)\bell}{\widehat{\mathbb{W}}_t}, \qquad 
\]
which provides a concise operator representation of the It\^o--Stratonovich relationship for signatures similar to \citet[Proposition 1]{benarous1989flots}.
\end{remark}

\begin{corollary}\label{cor:coordintate_OU}
    The ``coordinate'' Ornstein--Uhlenbeck semigroup ${\bf T}_\vartheta^i$ and generator ${\bf L}^i$ are given by 
    \begin{equation}
        {\bf T}^i_\vartheta\sigW 
        = \exp\!\left(\frac{1 - e^{-2\vartheta}}{2}\,{{\bf\Psi}}^{\word{0}}_{\word{ii}}\right)
          {\bf J}_{\word{i}}^{\vartheta}\sigW, \qquad     {\bf L}^i\sigW = 
    \bigl({{\bf\Psi}}^{\word{0}}_{\word{ii}} - {\bf \Lambda}_{\word{i}}\bigr)\sigW.
    \end{equation}
    The adjoint operators are given by
    \begin{equation}
    ({\bf T}^i_\vartheta)^* 
    = {\bf J}_{\word{i}}^{\vartheta}
      \exp\!\left(\frac{1 - e^{-2\vartheta}}{2}\,{{\bf\Psi}}^{\word{ii}}_{\word{0}}\right) \qquad     ({\bf L}^i)^* = 
    {{\bf\Psi}}^{\word{ii}}_{\word{0}} - {\bf \Lambda}_{\word{i}}.
    \end{equation}
\end{corollary}

An alternative way to compute the generators ${\bf L}^i$ consists in combining the formula \(\mathbf{L}^i = -\delta^i\D^i\) and Theorem~\ref{thm:iter_int_of_mal_der}. We start by writing
\begin{equation}\label{eq:skor_of_der}
    \mathbf{L}^i\sigW[t] = -\delta^i(\D^i\sigW) = -\delta^i\Big(\sigW[\cdot]\otimes\word{i}\otimes\sigW[\cdot, t]\Big),
\end{equation}

To proceed, one needs to express the Skorokhod integral in terms of the Stratonovich integral, as established in \cite[Theorem~3.1.1]{nualart2006malliavin}:
\begin{equation}\label{eq:skor_strato}
    \delta^i(X)
    = \int_0^t X_u \circ dW_u^i
    - \frac{1}{2}\int_0^t \left((\D^{+,i} X)_u + (\D^{-,i} X)_u\right) \, du,
\end{equation}
where $\D^{+,i} X$ and $\D^{-,i} X$ denote the $i$-th components of the $\D^{+} X$ and $\D^{-} X$ defined in \eqref{eq:strato_reg_assump}.

For the processes $X$ appearing in the Skorokhod integral \eqref{eq:skor_of_der}, $\D^{+,i} X$ and $\D^{-,i} X$ can be computed explicitly.

\begin{lemma}\label{lem:mal_trace}
Let $\word{v} \in V$ and consider the $T_2(\R^{d+1})$-valued process
\[
\sigX[s] = \sigW[s]\otimes\word{v}\otimes\sigW[s,t], \qquad s \leq t.
\]
Then $\sigX[\cdot] \in \mathbb{L}_2^{1,2}$ and, for $s \leq t$,
\[
(\D^{+,i} \sigX[\cdot])_s
= \sigW[s]\otimes\word{iv}\otimes\sigW[s,t],
\qquad
(\D^{-,i} \sigX[\cdot])_s
= \sigW[s]\otimes\word{vi}\otimes\sigW[s,t].
\]
\end{lemma}

\begin{proof}
For $0 < s < u < t$, we have
\[
\D_s^i \big(\sigW[u]\otimes\word{v}\otimes\sigW[u,t]\big)
= \sigW[s]\otimes\word{i}\otimes\sigW[s,u]\otimes\word{v}\otimes\sigW[u,t].
\]
We claim that for all $\word{w} \in V$,
\[
\E\left[
\left(
\D_s^i \big(\sigW[u]\otimes\word{v}\otimes\sigW[u,t]\big)^{\word{w}}
- (\sigW[s]\otimes\word{iv}\otimes\sigW[s,t])^{\word{w}}
\right)^2
\right]
\to 0,
\]
as $|u-s|\to 0$. Indeed, we have
$$
\D_s^i \big(\sigW[u]\otimes\word{v}\otimes\sigW[u,t]\big)
- \sigW[s]\otimes\word{iv}\otimes\sigW[s,t]
=
\sigW[s]\otimes\word{i}\otimes
\big(\sigW[s,u]\otimes\word{v}-\word{v}\otimes\sigW[s,u]\big)
\otimes\sigW[u,t],
$$
so that
\begin{align}
&\E\left[
\left(
\D_s^i \big(\sigW[u]\otimes\word{v}\otimes\sigW[u,t]\big)^{\word{w}}
- (\sigW[s]\otimes\word{iv}\otimes\sigW[s,t])^{\word{w}}
\right)^2
\right] \\
&= \E\left[
\left(\sum_{\substack{\word{w_1}, \word{w_2'}, \word{w_2''}, \word{w_3} \\
\word{w_1iw_2'vw_3 = w_1ivw_2''w_3 = w}}}
\sigW[s]^{\word{w_1}}
\big(\sigW[s,u]^{\word{w_2'}}-\sigW[s,u]^{\word{w_2''}}\big)
\sigW[u,t]^{\word{w_3}}
\right)^2
\right] \\
&\leq \E\left[
\left(\sum_{\substack{\word{w_1}, \word{w_2'}, \word{w_2''}, \word{w_3} \\
\word{|\word{w_1}|, |\word{w_3}| < |w|} \\
\word{0 < |\word{w_2'}|, |\word{w_2''}| < |w|}
}}
\big|\sigW[s]^{\word{w_1}}\big|
\big|\sigW[s,u]^{\word{w_2'}}-\sigW[s,u]^{\word{w_2''}}\big|
\big|\sigW[u,t]^{\word{w_3}}\big|
\right)^2
\right] \\
&\leq 
4\E\left[
\left(\sum_{
\word{|\word{w_1}| < |w|}}
\big|\sigW[s]^{\word{w_1}}\big|\right)^2\right] 
\E\left[
\left(\sum_{
\word{|\word{w_3}| < |w|}}
\big|\sigW[u, t]^{\word{w_3}}\big|\right)^2\right] 
\E\left[
\left(\sum_{
\word{0 < |\word{w_2}| < |w|}}
\big|\sigW[s, u]^{\word{w_2}}\big|\right)^2\right] \\
&\leq 
4\E\left[
\left(\sum_{
\word{|\word{w_1}| < |w|}}
\big|\sigW[T]^{\word{w_1}}\big|\right)^2\right]^2
\E\left[
\left(\sum_{
\word{0 < |\word{w_2}| < |w|}}
\big|\sigW[s, u]^{\word{w_2}}\big|\right)^2\right],
\end{align}
where, in the last inequality, we used the diffusive scaling of the signature.
The last expression depends only on $|u-s|$ and converges to $0$ as $|s-u|\to 0$. Hence, the limit in \eqref{eq:strato_reg_assump} vanishes. The proof for $0<u<s<t$ and $\D^{-,i}$ is analogous.
\end{proof}

Applying Lemma~\ref{lem:mal_trace} to $\sigW[\cdot]\otimes\word{i}\otimes\sigW[\cdot,t]$ yields
\[
\left(\D^{+,i} \Big(\sigW[\cdot]\otimes\word{i}\otimes\sigW[\cdot,t]\Big)\right)_u + \left(\D^{-,i} \Big(\sigW[\cdot]\otimes\word{i}\otimes\sigW[\cdot,t]\Big)\right)_u 
= 2\big(\sigW[u]\otimes\word{ii}\otimes\sigW[u,t]\big).
\]

The Skorokhod integral \eqref{eq:skor_of_der} can therefore be expressed in terms of the Stratonovich integral using \eqref{eq:skor_strato}:
\begin{align}\label{eq:rmk_skorokhod}
\delta^i\Big(\sigW[\cdot]\otimes\word{i}\otimes\sigW[\cdot,t]\Big)
&= \int_0^t
\big(\sigW[u]\otimes\word{i}\otimes\sigW[u,t]\big)\circ dW_u^i
- \int_0^t
\big(\sigW[u]\otimes\word{ii}\otimes\sigW[u,t]\big)\,du  \\
&= {\bf \Psi}^{\word{i}}_{\word{i}}(\sigW)
- {\bf \Psi}^{\word{0}}_{\word{ii}}(\sigW) \\
&= {\bf \Lambda}_{\word{i}}(\sigW)
- {\bf \Psi}^{\word{0}}_{\word{ii}}(\sigW),
\end{align}
where we used Theorem~\ref{thm:iter_int_of_mal_der} and observed that
${\bf \Psi}^{\word{i}}_{\word{i}} = {\bf \Lambda}_{\word{i}}$.
This recovers the formula obtained in Corollary~\ref{cor:coordintate_OU}.

\section{Integration by parts}\label{sect:ibp}
For conciseness of notation, we will denote by $\langle \cdot, \cdot \rangle_{L^2}$ the scalar product in the space $L^{2}(\Omega \times [0,T], \R^{d})$. That is, for $h = (h^1, \ldots, h^d)$ and $g = (g^1, \ldots, g^d)$ in $L^{2}(\Omega \times [0,T], \R^{d})$, we set
$$
\langle h, g \rangle_{L^2} = \int_0^T \langle h_t, g_t \rangle_{\R^d}\, dt.
$$

The following proposition (a multidimensional version of \citet[Proposition 6.2.1]{nualart2006malliavin}) allows for computation of Malliavin weights in the integration-by-parts formula \eqref{eq:IBP_general}.

\begin{proposition}\label{ibp_proposition}
    Let $G$ and $F$ be random variables such that $G \in \mathbb{D}^{1,2}(\mathbb{R})$, and let $h \in L^{2}(\Omega \times [0,T], \R^{d})$ be such that $\langle \D G, h\rangle_{L^2} \not= 0$ a.s. and $\dfrac{Fh}{\langle \D G, h\rangle_{L^2}} \in \mathrm{Dom}(\delta)$. Then, for any $f\in C^1_b(\R, \R)$ we have
    \begin{equation}\label{eq:delta_IBP}
        \E[f'(G)F] = \E\left[f(G)\delta\left(\dfrac{Fh}{\langle \D G, h \rangle_{L^2}}\right)\right].
    \end{equation}
    Moreover, if $\dfrac{F}{\langle \D G, h\rangle_{L^2}} \in \mathbb{D}^{1, 2}$ and $h \in \mathrm{Dom}(\delta)$, then
    \begin{equation}\label{eq:universal_delta}
        \E\left[f'(G)F\right] = 
        \E\left[
        f(G)\left(
        \dfrac{F\delta(h)}{\langle \D G, h \rangle_{L^2}}
        - \dfrac{\langle \D F, h\rangle_{L^2}}{\langle \D G, h\rangle_{L^2}}
        + \dfrac{F\left\langle \D\langle \D G, h\rangle_{L^2}, h\right\rangle_{L^2}}{\langle \D G, h\rangle_{L^2}^2}
        \right)\right].
    \end{equation}
\end{proposition}

\begin{proof}
    By the chain rule, $\bracket{\D f(G)}{h }_{L^2} = f'(G)\bracket{\D G}{h}_{L^2}$, so that
    \[
    f'(G) = \dfrac{\bracket{\D f(G)}{h}_{L^2}}{\bracket{\D G}{h}_{L^2}},
    \]
    and
    \begin{align}
        \E[f'(G)F] 
        &= \E\left[\dfrac{\bracket{\D f(G)}{h}_{L^2}}{\bracket{\D G}{h }_{L^2}}F\right] 
        = \E\left[f(G)\delta\left(\dfrac{F h}{\bracket{\D G}{h }_{L^2}}\right)\right],
    \end{align}
    where we used the definition of the Skorokhod integral \eqref{def:integ_skorohod_dual}. The formula \eqref{eq:IBP_formula} yields
    \begin{align}
        \E[f'(G)F]
        &= \E\left[f(G)\left(
        \dfrac{F\delta(h)}{\bracket{\D G}{h}_{L^2}}
        - \bracket{\D\left(\dfrac{F}{\bracket{\D G}{h}_{L^2}}\right)}{h}_{L^2}
        \right)\right],
    \end{align}
    and \eqref{eq:universal_delta} follows from a direct computation of the Malliavin derivative inside the integral.
\end{proof}

\subsection{Linearization of the $L^2$ scalar product}

Typically, computing $\langle \D G , h\rangle_{L^2}$ and $\left\langle \D \langle \D G, h\rangle_{L^2}, h\right\rangle_{L^2}$, which appear in \eqref{eq:universal_delta}, is non-trivial, and their numerical evaluation may present a significant challenge. However, this computation is significantly simplified if the process $h$ is chosen in such a way that it allows one to linearize $\langle \D G , h\rangle_{L^2}$, i.e., to express it as a linear functional of $\sigW[T]$. Assuming additionally that both $G$ and $F$ are linear or rational (i.e., ratios of two linear functionals) functions of $\sigW[T]$ makes it possible to rewrite the weight given by \eqref{eq:universal_delta} as a rational function of $\sigW[T]$ as well.

We start by considering two random variables $F, F' \in \mathcal{P}(\sigW[T])$ with corresponding coefficients $\bell, \bell' \in \TA[d+1]$, and we introduce an operation $\diamond\colon \TA\times\TA \to \TA$ that allows us to linearize the scalar product $\bracket{\D F}{\D F'}_{L^2}$, i.e., to express it as
\begin{equation}
    \bracket{\D F}{\D F'}_{L^2}
    = \bracket{\D\langle\bell, \sigW[T]\rangle}{\D \langle\bell', \sigW[T]\rangle}_{L^2} 
    = \bracketsigW[T]{\bell \diamond \bell'}.
\end{equation}

\begin{definition}
    Fix three words $\word{u_1}, \word{u_2}, \word{w} \in V$. For two words $\word{v}, \word{v'}$, define $\diamond^{\word{u_1}, \word{u_2}}_{\word{w}}: \TA\times\TA\to\TA$ by
    \begin{equation}\label{eq:diamond_def}
        \word{v} \diamond^{\word{u_1}, \word{u_2}}_{\word{w}}\word{v'} 
        = \sum_{\substack{\word{v_1}, \word{v_2}:\\\word{v_1u_1v_2} = \word{v}}}
          \sum_{\substack{\word{v_1'}, \word{v_2'}:\\\word{v_1'u_2v_2'} = \word{v'}}}
          (\word{v_1}\shuprod\word{v_1'})\otimes\word{w}\otimes(\word{v_2}\shuprod\word{v_2'}).
    \end{equation}
    For two tensor sequences $\bell, \bell' \in \TA$, the operator is extended by bilinearity. We also define $\diamond := \sum_{i = 1}^d \diamond^{\word{i}, \word{i}}_{\word{0}}$.
\end{definition}

The proposition below clarifies the link between this algebraic operation and the integration of the product of linear functionals of two ``pierced'' signatures $\sigX[t]\otimes\word{u}\otimes\sigX[t, T]$.

\begin{proposition}\label{prop:diamond} 
Let $X$ be an $\R^d$-valued continuous semimartingale. Then, for all words $\word{u_1}, \word{u_2} \in V$, letter $\word{k} \in \alphabet[]$, and $\bell, \bell' \in \TA$,
\begin{equation}\label{eq:diamond_prod_prop}
    \int_0^T\langle\bell, \sigX[t]\otimes\word{u_1}\otimes\sigX[t, T]\rangle \,\langle\bell', \sigX[t]\otimes\word{u_2}\otimes\sigX[t, T]\rangle \circ dX_t^k 
    = \bracketsigX[T]{\bell \diamond^{\word{u_1}, \word{u_2}}_{\word{k}} \bell'}.
\end{equation}
\end{proposition}

\begin{proof}
    Since the integral in \eqref{eq:diamond_prod_prop} is a bilinear functional of $\bell$ and $\bell'$, it suffices to consider the case $\bell = \word{v} \in V$ and $\bell' = \word{v'} \in V$. Then the integrand expands as
    \begin{align}
        &\int_0^T\langle\word{v}, \sigX[t]\otimes\word{u_1}\otimes\sigX[t, T]\rangle \,\langle\word{v'}, \sigX[t]\otimes\word{u_2}\otimes\sigX[t, T]\rangle \circ dX_t^k \\
        &\quad=\sum_{\substack{\word{v_1}, \word{v_2}:\\\word{v_1u_1v_2 = v}}}
                \sum_{\substack{\word{v_1'}, \word{v_2'}:\\\word{v_1'u_2v_2' = v'}}}
                \int_0^T \sigX[t]^{\word{v_1}}\sigX[t]^{\word{v_1'}}\sigX[t,T]^{\word{v_2}}\sigX[t,T]^{\word{v_2'}} \circ dX_t^k \\
        &\quad= \sum_{\substack{\word{v_1}, \word{v_2}:\\\word{v_1u_1v_2 = v}}}
                \sum_{\substack{\word{v_1'}, \word{v_2'}:\\\word{v_1'u_2v_2' = v'}}}
                \int_0^T \bracketsigX{\word{v_1}\shuprod\word{v_1'}}\,\bracketsigX[t, T]{\word{v_2}\shuprod\word{v_2'}} \circ dX_t^k \\
        &\quad= \sum_{\substack{\word{v_1}, \word{v_2}:\\\word{v_1u_1v_2 = v}}}
                \sum_{\substack{\word{v_1'}, \word{v_2'}:\\\word{v_1'u_2v_2' = v'}}}
                \bracketsigX[T]{(\word{v_1}\shuprod\word{v_1'})\word{k}(\word{v_2}\shuprod\word{v_2'})} \\
        &\quad= \bracketsigX[T]{\word{v}\diamond^{\word{u_1}, \word{u_2}}_{\word{k}}\word{v'}},
    \end{align}
    where Lemma~\ref{lem:inserted_int} is used in the third equality. This proves \eqref{eq:diamond_prod_prop}.
\end{proof}

\begin{corollary}\label{cor:scal_prod_lin}
    Let $F, F' \in \mathcal{P}(\sigW[T])$ with the coefficients $\bell, \bell' \in \TA[d+1]$. Then
    \begin{equation}\label{eq:L2_linearization}
        \langle \D F, \D F' \rangle_{L^2} = \bracketsigW[T]{\bell\diamond\bell'}.
    \end{equation}
\end{corollary}

\begin{proof}
    We have
    \begin{align*}
        \langle \D F, \D F' \rangle_{L^2} 
        &= \sum_{i = 1}^d\int_0^T \langle\bell, \D^i_t\sigW[T]\rangle \,\langle\bell', \D^i_t\sigW[T]\rangle \,dt \\
        &= \sum_{i = 1}^d\int_0^T \langle\bell, \sigW[t]\otimes\word{i}\otimes\sigW[t,T]\rangle \,\langle\bell', \sigW[t]\otimes\word{i}\otimes\sigW[t,T]\rangle \,dt \\
        &= \sum_{i = 1}^d\bracketsigW[T]{\bell \diamond^{\word{i}, \word{i}}_{\word{0}} \bell'} = \bracketsigW[T]{\bell \diamond\bell'},
    \end{align*}
    where we used Proposition~\ref{prop:diamond} and the definition of $\diamond$.
\end{proof}

\begin{corollary}\label{cor:diamond_letter}
    Fix $\word{i}, \word{j}, \word{k} \in A$ and $\bell \in \TA[d+1]$. Then $\bell\diamond^{\word{i}, \word{j}}_{\word{k}}\word{j} = {\bf \Psi}^{\word{i}}_{\word{k}}(\bell)$.
\end{corollary}
\begin{proof}
    It is sufficient to observe that for all continuous semimartingales $X$, we have
    $$
    \bracketsigX[T]{\bell \diamond^{\word{i}, \word{j}}_{\word{k}} \word{j}}
    = \int_0^T \langle\bell, \sigX[t]\otimes\word{i}\otimes\sigX[t, T]\rangle \,\underbrace{\langle\word{j}, \sigX[t]\otimes\word{j}\otimes\sigX[t, T]\rangle}_{=1} \circ dX_t^k
    = \bracketsigX[T]{{\bf \Psi}^{\word{i}}_{\word{k}}(\bell)},
    $$
    where the last equality follows from Theorem~\ref{thm:iterated_psi}.
\end{proof}

Although it enjoys the desired linearization property \eqref{eq:L2_linearization}, the direct computation of $\diamond$ via its definition \eqref{eq:diamond_def} is rather cumbersome and numerically intricate. The following proposition shows that $\diamond$ can be expressed in terms of the switching operators ${\bf \Psi}^{\word{ii}}_{\word{0}}$, enabling a significantly more efficient numerical implementation.

\begin{proposition}
    If $\word{i}, \word{k} \in A$ and $\bell, \bell' \in \TA[d+1]$, then the diamond product $\diamond^{\word{i}, \word{i}}_{\word{k}}$ can be written in carré du champ form:
\begin{equation}\label{eq:diamond_simplified}
    \bell \diamond^{\word{i}, \word{i}}_{\word{k}} \bell'
    = \dfrac{1}{2}\left({\bf \Psi}^{\word{ii}}_{\word{k}}(\bell\shuprod\bell') - {\bf \Psi}^{\word{ii}}_{\word{k}}(\bell)\shuprod\bell' - \bell\shuprod{\bf \Psi}^{\word{ii}}_{\word{k}}(\bell')\right).
\end{equation}
In particular, we have
\begin{equation}
    \bell \diamond \bell'
    = \dfrac{1}{2}\sum_{i=1}^d\left({\bf \Psi}^{\word{ii}}_{\word{0}}(\bell\shuprod\bell') - {\bf \Psi}^{\word{ii}}_{\word{0}}(\bell)\shuprod\bell' - \bell\shuprod{\bf \Psi}^{\word{ii}}_{\word{0}}(\bell')\right).
\end{equation}
\end{proposition}
\begin{proof}
    For $t \in [0,T]$ and $\varepsilon > 0$, take an arbitrary continuous semimartingale $X$ and define
    \[
        h_t(\varepsilon) := \langle \bell \shuprod \bell',\, \sigX \otimes \exp^\otimes(\varepsilon\word{i}) \otimes \sigX[t, T]\rangle.
    \]
    Since \(\sigX \otimes \exp^\otimes(\varepsilon\word{i}) \otimes \sigX[t, T]\) is group-like (i.e., corresponds to the signature of a path as a product of signatures), it satisfies the shuffle property:
    \[
        h_t(\varepsilon)
        = \langle \bell \shuprod \bell',\, \sigX \otimes \exp^\otimes(\varepsilon\word{i}) \otimes \sigX[t, T]\rangle
        = \langle \bell,\, \sigX \otimes \exp^\otimes(\varepsilon\word{i}) \otimes \sigX[t, T]\rangle
          \langle \bell',\, \sigX \otimes \exp^\otimes(\varepsilon\word{i}) \otimes \sigX[t, T]\rangle.
    \]
    Differentiating both sides twice with respect to $\varepsilon$ at $\varepsilon=0$ gives 
    \begin{align*}
        \dfrac{d^2}{d\varepsilon^2}\bigg|_{\varepsilon = 0} h_t(\varepsilon) 
        &= \langle \bell \shuprod \bell',\, \sigX \otimes \word{ii}\otimes\sigX[t, T]\rangle \\
        &= \langle \bell,\, \sigX \otimes \word{ii}\otimes\sigX[t, T]\rangle\langle \bell',\, \sigX[T]\rangle 
        + \langle \bell,\, \sigX[T]\rangle \langle \bell',\, \sigX \otimes \word{ii}\otimes\sigX[t, T]\rangle \\
        &\quad + 2\langle \bell,\, \sigX \otimes \word{i}\otimes\sigX[t, T]\rangle\langle \bell',\, \sigX \otimes \word{i}\otimes\sigX[t, T]\rangle.
    \end{align*}
    Integrating this identity with respect to $\circ dX_t^k$ over $[0,T]$ and applying Theorem~\ref{thm:iterated_psi}, Proposition~\ref{prop:psi_adjoint}, and \eqref{eq:diamond_prod_prop} yields 
    $$
    \bracketsigX[T]{{\bf \Psi}^{\word{ii}}_{\word{k}}(\bell\shuprod\bell')} = \bracketsigX[T]{{\bf \Psi}^{\word{ii}}_{\word{k}}(\bell)}\bracketsigX[T]{\bell'} + \bracketsigX[T]{\bell}\bracketsigX[T]{{\bf \Psi}^{\word{ii}}_{\word{k}}(\bell')} +
    2\bracketsigX[T]{\bell \diamond^{\word{i}, \word{i}}_{\word{k}} \bell'},
    $$
    so that
    $$
    \bracketsigX[T]{\bell \diamond^{\word{i}, \word{i}}_{\word{k}} \bell'}
    = \dfrac{1}{2}\bracketsigX[T]{{\bf \Psi}^{\word{ii}}_{\word{k}}(\bell\shuprod\bell') - {\bf \Psi}^{\word{ii}}_{\word{k}}(\bell)\shuprod\bell' - \bell\shuprod{\bf \Psi}^{\word{ii}}_{\word{k}}(\bell')}
    $$
Finally, noting that this equality remains true for all possible \(X\), we conclude that
\eqref{eq:diamond_simplified} holds.
\end{proof}

\subsection{Integration by parts formula}

The linearization property \eqref{eq:L2_linearization}, proved in the previous subsection, suggests that taking $h = \D H$ for some $H \in \mathcal{P}(\sigW[T])$ is a natural choice for the function $h$ in Proposition~\ref{ibp_proposition}. However, one can obtain a much larger class of admissible functions $h$ by taking 
\begin{equation}\label{eq:h_class}
    h = \left(\D^i \bracketsigW[T]{\bell^h_i}\right)_{i = 1, \ldots, d},
    \qquad \bell^h_i \in \TA[n + 1], \quad i = 1, \ldots, d,
\end{equation}
so that we associate with each function $h$ a collection of coefficients 
\[
\vec{\bell}^h = (\bell_1^h, \ldots, \bell_d^h) \in \TA[d+1]^d.
\]
A natural extension of the ``diamond'' product $\diamond$ to vectors of coefficients is
\[
\vec{\bell} \diamond \vec{\bell'} := \sum_{i=1}^d \bell_i \diamond^{\word{i}, \word{i}}_{\word{0}}\bell'_i,
\qquad \vec{\bell}, \vec{\bell'} \in \TA[d+1]^d.
\]
Hence, for $G = \bracketsigW[T]{\bell^G}$ and $h$ of the form \eqref{eq:h_class}, we obtain
\begin{equation}\label{eq:DG_h_linearization}
    \bracket{\D G}{h}_{L^2}
    = \bracketsigW[T]{\bell^G \diamond \vec{\bell}^h},
\end{equation}
where we identify $\bell^G$ with the vector $\vec{\bell}^G = (\bell^G, \ldots, \bell^G)$.

Before stating the main theorem of this section, we introduce the class $\mathcal{R}(\sigW[T])$ of random variables that can be written as a ratio of two linear functionals of the Brownian signature.

\begin{definition}
    We say that a random variable $F$ belongs to $\mathcal{R}(\sigW[T])$ if there exist $\bell_1, \bell_2 \in \TA$ such that $\bracketsigW[T]{\bell_2} \neq 0$ a.s.\ and
    \[
        F = \dfrac{\bracketsigW[T]{\bell_1}}{\bracketsigW[T]{\bell_2}} \quad \text{a.s.}
    \]
\end{definition}

Note that $\mathcal{R}(\sigW[T])$ forms an algebra.
\begin{theorem}\label{thm:ibp_sig}
    Under the assumptions of Proposition~\ref{ibp_proposition}, suppose that 
    \begin{enumerate}
        \item $G \in \mathcal{P}(\sigW[T])$ with the coefficient $\bell^G \in \TA$.
        \item $F \in \mathcal{R}(\sigW[T])$, i.e., $F = \dfrac{\bracketsigW[T]{\bell^F_1}}{\bracketsigW[T]{\bell^F_2}}$ for some $\bell^F_1, \bell^F_2 \in \TA$.
        \item $h$ is a process of the form \eqref{eq:h_class}.
        \item The coefficients $\bell^G, \bell^{F_1}, \bell^{F_2}, \vec\bell^h$ are such that $\dfrac{F}{\bracket{\D G}{h}_{L^2}} = \dfrac{\bracketsigW[T]{\bell^F_1}}{\bracketsigW[T]{\bell^F_2}\bracketsigW[T]{\bell^G \diamond \vec\bell^{h}}} \in \mathbb{D}^{1,2}$.
    \end{enumerate} 
    Then, 
    \begin{equation}\label{eq:universal_delta_linearized}
        \E\left[f'(G)F\right] = 
        \E\left[
        f(G)\pi_T\right],
    \end{equation}
    where $\pi_T \in \mathcal{R}(\sigW[T])$ is given by
    \begin{equation}\label{eq:universal_weight_linearized}
    \begin{aligned}
        \pi_T &= 
        \dfrac{\bracketsigW[T]{\bell^F_1}\bracketsigW[T]{\sum_{i=1}^d({\bf \Lambda}_{\word{i}} - {{\bf\Psi}}^{\word{ii}}_{\word{0}})(\bell^{h}_i)}}{\bracketsigW[T]{\bell^F_2}\bracketsigW[T]{\bell^G \diamond \vec\bell^{h}}}
        - \dfrac{\bracketsigW[T]{\bell^F_1 \diamond \vec\bell^{h}}}{\bracketsigW[T]{\bell^F_2}\bracketsigW[T]{\bell^G \diamond \vec\bell^{h}}} \\
        &\quad+ \dfrac{\bracketsigW[T]{\bell^F_1}\bracketsigW[T]{\bell^F_2 \diamond\vec\bell^{h}}}{\bracketsigW[T]{\bell^F_2}^2\bracketsigW[T]{\bell^G \diamond \vec\bell^{h}}} 
        + \dfrac{\bracketsigW[T]{\bell^F_1}\bracketsigW[T]{\bell^G \diamond\vec\bell^{h}\diamond \vec\bell^{h}}}{\bracketsigW[T]{\bell^F_2}\bracketsigW[T]{\bell^G \diamond \vec\bell^{h}}^2}.
    \end{aligned}       
    \end{equation}
\end{theorem}

\begin{proof}
    The proof relies on formula \eqref{eq:universal_delta}, obtained in Proposition~\ref{ibp_proposition}. Our goal is to show that each term in the Malliavin weight given by \eqref{eq:universal_delta} can be written as a fraction of two random variables in $\mathcal{P}(\sigW[T])$.
    
    First, all the terms in \eqref{eq:universal_delta} containing $\bracket{\D G}{h}_{L^2}$ can be computed via \eqref{eq:DG_h_linearization}. 

    To compute $\bracket{\D F}{h}$, we apply the chain rule \eqref{eq:chain_rule}: 
    \[
    \D F 
    = \D \left(\dfrac{\bracketsigW[T]{\bell^F_1}}{\bracketsigW[T]{\bell^F_2}}\right) 
    = \dfrac{\D \bracket{\bell^F_1}{\sigW[T]}}{\bracketsigW[T]{\bell^F_2}}
      - \dfrac{\bracketsigW[T]{\bell^F_1}\D \bracket{\bell^F_2}{\sigW[T]}}{\bracketsigW[T]{\bell^F_2}^2},
    \]
    so that, {using (\ref{eq:DG_h_linearization}),}
    \begin{equation}\label{eq:proof_IBP_2}
        \bracket{\D F}{h}_{L^2} 
        = \dfrac{\bracket{\bell^F_1 \diamond \vec\bell^{h}}{\sigW[T]}}{\bracketsigW[T]{\bell^F_2}}
        - \dfrac{\bracketsigW[T]{\bell^F_1}\bracket{\bell^F_2\diamond \vec\bell^{h}}{\sigW[T]}}{\bracketsigW[T]{\bell^F_2}^2}.
    \end{equation}
    As for the term $\delta(h)$, we observe that
    \begin{equation}\label{eq:proof_IBP_3}
        \delta^i(h^i) 
        = \delta^i\!\left(\D^i\bracketsigW[T]{\bell^{h}_i}\right)
        = -{\bf L}^i\bracketsigW[T]{\bell^{h}_i}
        = \bracketsigW[T]{({\bf \Lambda}_{\word{i}} - {{\bf\Psi}}^{\word{ii}}_{\word{0}})(\bell^{h^i})}, 
    \end{equation}
    as follows from Corollary~\ref{cor:coordintate_OU}.

    Combining \eqref{eq:DG_h_linearization}, \eqref{eq:proof_IBP_2}, and \eqref{eq:proof_IBP_3} and substituting the results into \eqref{eq:universal_delta} yields \eqref{eq:universal_weight_linearized}.

    It remains to notice that $\bracketsigW[T]{\bell^F_2} \neq 0$ a.s. since $F \in \mathcal{R}(\sigW[T])$, and that 
    \[
        \bracketsigW[T]{\bell^G \diamond \vec\bell^{h}} 
        = \langle \D G, h\rangle_{L^2} \neq 0 
        \quad (a.s.),
    \]
    by the assumption of Proposition~\ref{ibp_proposition}, so that $\pi_T$ belongs to $\mathcal{R}(\sigW[T])$.
\end{proof}

\begin{sqremark}
    Theorem~\ref{thm:ibp_sig} states that the integration by parts formula \eqref{eq:universal_delta_linearized} preserves the class $\mathcal{R}(\sigW[T])$. This implies that if $f \in C^k$ for some $k \geq 2$, then one can iterate the formula to obtain
    \begin{equation}\label{eq:ibp_recursive}
        \E\left[f^{(k)}(G)F\right] 
        = \E\left[f(G)\pi_T^k\right],
    \end{equation}
    where $\pi_T^k \in \mathcal{R}(\sigW[T])$ is given recursively by \eqref{eq:universal_weight_linearized}, replacing $\bell^F_m$ by $\bell^{\pi_T^{k - 1}}_m$ for $m \in \{1, 2\}$ and $k \geq 2$.
\end{sqremark}

\begin{sqremark}
    The formula \eqref{eq:universal_weight_linearized} can be written explicitly in the compact form
    \begin{equation}
        \pi_T = \dfrac{\bracketsigW[T]{\bell^{\pi_T}_1}}{\bracketsigW[T]{\bell^{\pi_T}_2}}, \quad \bracketsigW[T]{\bell^{\pi_T}_2} \neq 0,
    \end{equation}
    with
    \begin{align}
        \bell^{\pi_T}_1 &= \left(\sum_{i = 1}^d({\bf \Lambda}_{\word{i}} - {{\bf\Psi}}^{\word{ii}}_{\word{0}})(\bell^{h^i})\right)\shuprod\bell_1^F\shuprod\bell_2^F\shuprod\left(\bell^G \diamond\vec\bell^{h}\right) 
         - \left(\bell^F_1 \diamond \vec\bell^{h}\right)\shuprod\bell_2^F\shuprod\left(\bell^G \diamond\vec\bell^{h}\right) \\
        & + \bell^F_1\shuprod\left(\bell^F_1 \diamond\vec\bell^{h}\right)\shuprod\left(\bell^G\diamond\vec\bell^{h}\right)
         + \bell^F_1\shuprod\bell^F_2\shuprod\left(\bell^G \diamond\vec\bell^{h}\diamond\vec\bell^{h}\right),
        \\
        \bell^{\pi_T}_2 &= (\bell^F_2)\shupow{2}\shuprod\left(\bell^G \diamond\vec\bell^{h}\right)\shupow{2}.
    \end{align}
    Although this makes the implementation of the recursive formula \eqref{eq:ibp_recursive} straightforward, it has an important practical disadvantage: computing $\bracketsigW[T]{\bell^{\pi_T}_1}$ and $\bracketsigW[T]{\bell^{\pi_T}_2}$ requires evaluating $\sigW^{\leq N}$ for a much higher order $N$ than would be necessary when using the “decomposed’’ formula \eqref{eq:universal_weight_linearized}.
\end{sqremark}

\begin{sqexample}\label{ex:weights_example}
    The following examples show that two widely used choices of $h$ are covered by Theorem~\ref{thm:ibp_sig}.
\begin{enumerate}
    \item Consider first the simplest choice $h^i \equiv \mathds{1}_{[0, T]}$, $i \in I$, for some set of indices $I \subset \{1, \ldots, d\}$, as considered in \citep{al2023computation}. Note that $\mathds{1}_{[0, T]} \equiv \D^i\bracketsigW[T]{\word{i}}$, so that $\bell^{h}_i = \indic{i \in I}\word{i}$. Since, by Corollary~\ref{cor:diamond_letter}, $\bell \diamond^{\word{i}, \word{i}}_{\word{0}} \word{i} = {{\bf\Psi}}^{\word{i}}_{\word{0}}(\bell)$ for all $\bell\in\TA$, and since $({\bf \Lambda}_{\word{i}} - {{\bf\Psi}}^{\word{ii}}_{\word{0}})(\word{i}) = \word{i}$ by direct computation, the expression \eqref{eq:universal_weight_linearized} simplifies to
    \begin{equation}
    \begin{aligned}
        \pi_T &= 
        \dfrac{\bracketsigW[T]{\sum_{i \in I}\word{i}}}{\bracketsigW[T]{\sum_{i\in I}{\bf \Psi}^{\word{i}}_{\word{0}}(\bell^G )}}F 
        \;-\; \dfrac{\bracketsigW[T]{\sum_{i \in I}{\bf \Psi}^{\word{i}}_{\word{0}}(\bell^F_1)}}{\bracketsigW[T]{\bell^F_2}\,\bracketsigW[T]{\sum_{i\in I}{\bf \Psi}^{\word{i}}_{\word{0}}(\bell^G )}} \\
        &\quad + \dfrac{\bracketsigW[T]{\bell^F_1}\,\bracketsigW[T]{\sum_{i \in I}{\bf \Psi}^{\word{i}}_{\word{0}}(\bell^F_2)}}{\bracketsigW[T]{\bell^F_2}^2\,\bracketsigW[T]{\sum_{i\in I}{\bf \Psi}^{\word{i}}_{\word{0}}(\bell^G )}} 
        + \dfrac{\bracketsigW[T]{\sum_{i,j \in I}{\bf \Psi}^{\word{i}}_{\word{0}}({\bf \Psi}^{\word{j}}_{\word{0}}(\bell^G))}}{\bracketsigW[T]{\sum_{i\in I}{\bf \Psi}^{\word{i}}_{\word{0}}(\bell^G )}^2}F.
    \end{aligned}   
    \end{equation}
    Despite the relative simplicity of this formula, it is necessary to justify that $\bracketsigW[T]{\sum_{i\in I}{\bf \Psi}^{\word{i}}_{\word{0}}(\bell^G)} \neq 0$ a.s. If this cannot be guaranteed, one might prefer the second option.

    \item The second natural choice is $h^i = \D^iG$ (i.e.\ $\bell^{h}_i = \bell^G$) for $i \in I$. It is often used to establish the regularity of densities (see, for instance, \citet[Proposition 2.1.1]{nualart2006malliavin}). This choice ensures that 
    \begin{equation}\label{eq:DG_h_cond}
        \langle \D G, h\rangle_{L^2} = \sum_{i \in I}\|\D^i G\|^2 \ge 0.
    \end{equation}
    The expression \eqref{eq:DG_h_cond} is strictly positive if a.s.\ there exist $i \in I$ and $t\in[0, T]$ such that $\D_t^i G \neq 0$, which is much easier to verify than the corresponding condition in the first case. Note that, when $I = \{1, \ldots, d\}$, the condition $\langle D G, h \rangle_{L^2} \neq 0$ represents the standard non-degeneracy condition for the Malliavin covariance matrix, which reduces to a scalar in this setting. 
\end{enumerate}
\end{sqexample}

\begin{sqremark}
    The assumptions on the process $h$ in Theorem~\ref{thm:ibp_sig} allow us not only to encompass the most common choices of $h$, as shown in the example above, but also to express the Malliavin weight using the unified formula \eqref{eq:universal_weight_linearized}. However, in general, other classes of processes $h$ may be considered, and the corresponding integration by parts formulae can be derived using similar arguments.
\end{sqremark}

\section{Greeks in the signature volatility model}\label{sect:application}
We illustrate our results by computing the Greeks within the Signature Volatility model. Throughout this section, we assume that $d = 2$. The asset price under the risk-neutral measure in the Signature Volatility model is given by
\begin{align}\label{eq:model_sv}
 \dfrac{dS_t}{S_t} &= \sigma_t dB_t, \quad t \in [0, T],\\
 \sigma_t &= \bracketsigW{\bsigma}.
\end{align}
where $\bsigma\in\TA[3]$, $\rho \in [-1, 1]$, $\sigW$ denotes the signature of $\widehat{W}_t = (t, W_t^1, W_t^2)$, and $B_t = \rho W_t^1 + \bar\rho W_t^2$ with $\bar\rho = \sqrt{1 - \rho^2}$. We assume the interest rate is $r = 0$ and that the volatility process is independent of $W^2$, meaning no letter $\word{2}$ appears in $\bsigma$. To ensure the martingale property of the price process $S$ established in \citep[Theorem 3.1]{jaber2025mart}, we take odd truncation orders $N$ (that is, $\bsigma^{\word{v}} = 0$ for all $|\word{v}| > N$) and impose the condition $\rho\,\bsigma^{\word{1}\conpow{N}} < 0$, unless the volatility process is linear in $W$.

This model was shown to be universal in the sense that a large class of Markovian and non-Markovian stochastic volatility models can be exactly represented in the form \eqref{eq:model_sv} with an infinite collection of coefficients $\bsigma$, and can be well approximated by such models with finitely many coefficients $\bsigma$; see \citep{abijaber2024signature}. This highlights the universality of the approach proposed in this paper, which can be applied to any model that can be written in the form \eqref{eq:model_sv} without any modification of the pricing or Greeks formulae.

We are interested in general path-dependent options with maturity $T > 0$ that depend on the signature $\sigXhat[T]$ of the time-augmented log-price $\widehat X: t \mapsto (t, \log S_t)$ for $t\in[0, T]$. We show in Proposition~\ref{prop:log_price_sig_repr} below that $X_t$ can be represented as a linear functional $\bracketsigW{\bell^X}$ of the Brownian signature. Hence, each element of $\sigXhat[T]$ can be expressed as a linear functional of $\sigW[T]$ as well; see \citep[Theorem~3.16]{cuchiero2022theocalib} for an explicit formula. Therefore, we can assume the payoff is written in the form $f\left(G\right), \ G = \bracketsigW[T]{\bell^G}$, where $\bell^G \in \TA[3]$. The option price is given by $\E\left[f\left(G\right)\right]$. Our goal is to calculate the sensitivities
\begin{equation}\label{eq:greek_to_compute}
\dfrac{\partial}{\partial\theta}\E\left[f\left(G\right)\right],
\end{equation}
with respect to the model parameters $\theta$, where $\theta \in \{S_0, \rho, \bsigma\}$.

Assuming for a moment that $f\in C^1_b$ and that we can exchange the expectation and the differentiation, the Greek \eqref{eq:greek_to_compute} is given by
\begin{equation}\label{eq:greek_sigvol}
\E\left[f'\left(\bracketsigW[T]{\bell^G}\right) \bracketsigW[T]{\dfrac{\partial}{\partial\theta}\bell^G}\right] = \E\left[f\left(\bracketsigW[T]{\bell^G}\right)\pi_T\right].
\end{equation}
Note that $\bracketsigW[T]{\dfrac{\partial}{\partial\theta}\bell^G} \in \mathcal{P}(\sigW[T]) \subset \mathcal{R}(\sigW[T])$, so that Theorem~\ref{thm:ibp_sig} applies and the weight $\pi_T$ is given by \eqref{eq:universal_weight_linearized}. This weight is independent of the function~$f$, but it does depend on the option payoff through the coefficient~$\bell^G$.

\begin{remark}\label{rmk:approx_argument}
    The integration-by-parts result given by Theorem~\ref{thm:ibp_sig} requires the function~$f$ to be differentiable with a bounded derivative. However, functions of practical interest are typically non-differentiable, such as call option payoffs, or even discontinuous, such as digital option payoffs. In these cases, equation~\eqref{eq:greek_sigvol} no longer makes sense.
    Nevertheless, for a square-integrable function~$f$, in the sense that
    \[
    \E\left[f\left(\bracketsigW[T]{\bell^G}\right)^2\right] < \infty,
    \]
    assuming that the maps $\theta \mapsto \E\left[f\left(\bracketsigW[T]{\bell^G}\right)^2\right]$ and $\theta \mapsto \E[\pi_T^2]$ are continuous, one can prove that
    \begin{equation}\label{eq:ibp_non_smooth}
    \frac{\partial}{\partial\theta}\E\left[f\left(\bracketsigW[T]{\bell^G}\right)\right]
    =
    \E\left[f\left(\bracketsigW[T]{\bell^G}\right)\pi_T\right].
    \end{equation}
    This is achieved by approximating~$f$ in $L^2$ by a sequence $(f_n)_{n\ge 0}$ of compactly supported, infinitely differentiable functions for which \eqref{eq:ibp_non_smooth} holds, and then passing to the limit as $n \to \infty$. A rigorous proof of this approximation argument can be obtained by a straightforward adaptation of the proof of \citep[Proposition~3.2(ii)]{Fournie1999}. Note that, for the examples considered below, the continuity assumptions are trivially satisfied.
\end{remark}

Before considering concrete examples, we rewrite $\log S_T$ as a linear functional of the Brownian signature.
\begin{proposition}\label{prop:log_price_sig_repr}
    The log-price $X_t = \log S_t$ for $t \in [0, T]$ is given by
\begin{equation}\label{eq:log-price_repr}
    X_t = \bracketsigW{\bell^X}, \quad\bell^X := \log (S_0)\emptyword{-\dfrac{1}{2}\bsigma\shupow{2}\otimes\word{0} + \rho\left(\bsigma\otimes\word{1} - \dfrac{1}{2}\bsigma\proj{1}\otimes\word{0}\right) + \bar\rho\bsigma\otimes\word{2} } \in \TA,
\end{equation}
where $\bsigma\proj{i}= \sum_{n \geq 0}\sum_{|\word{v}|=v}\bsigma^{\word{vi}}\word{v}$.
\end{proposition}
\begin{proof}
We have
\begin{align}
S_T &= S_0 \exp\left(-\dfrac12\int_0^T\sigma_u^2\,du + \int_0^T\sigma_u\,\rho dW_u^1 + \int_0^T\sigma_u\,\bar\rho dW_u^2 \right) \\
&=S_0 \exp\left(\bracketsigW[T]{-\dfrac{1}{2}\bsigma\shupow{2}\otimes\word{0} + \rho\left(\bsigma\otimes\word{1} - \dfrac{1}{2}\bsigma\proj{1}\otimes\word{0}\right) + \bar\rho\bsigma\otimes\word{2} } \right).
\end{align}
where we used the identity 
$
\int_0^T \bracketsigW{\bell}\, dW_t^i 
= \bracketsigW[T]{\,\bell\otimes\word{i} - \frac{1}{2}\,\bell\proj{i}\otimes\word{0}\,},
$
which follows from the relationship between the Itô and Stratonovich integrals, as well as the dynamics of $\bracket{\bell}{\sigW}$:
\begin{align}
    d\bracket{\bell}{\sigW} &= \bracket{\bell}{\sigW\otimes\word{0}}\,dt + \bracket{\bell}{\sigW\otimes\word{1}}\circ dW_t^1  + \bracket{\bell}{\sigW\otimes\word{2}}\circ dW_t^2, \\
    &= \bracket{\bell\proj{0}}{\sigW}\,dt + \bracket{\bell\proj{1}}{\sigW}\circ dW_t^1  + \bracket{\bell\proj{2}}{\sigW}\circ dW_t^2.
\end{align}
It remains to notice that $\bsigma\proj{2} = 0$ by assumption on $\bsigma$. Taking the logarithm then concludes the proof.
\end{proof}

\subsection{European options}\label{sect:european}
In this subsection, we consider a general European option with the payoff $f(\log S_T) = f\left(\bracketsigW[T]{\bell^X}\right)$ at maturity $T$, so that $\bell^G = \bell^X$.

It is straightforward to compute $F:=\dfrac{\partial}{\partial\theta}\bracketsigW[T]{\bell^X}$ for $\theta \in \{S_0, \rho, \bsigma\}$ using \eqref{eq:log-price_repr}. Indeed,
$$
\dfrac{\partial}{\partial S_0}\bracketsigW[T]{\bell^X} = \frac{1}{S_0}, \quad \dfrac{\partial}{\partial \rho}\bracketsigW[T]{\bell^X} = \bracketsigW[T]{\bsigma\word{1} - \dfrac{1}{2}\bsigma\proj{1}\word{0}-\frac{\rho}{\bar{\rho}  }\bsigma \word{2}},
$$
and the derivatives $\dfrac{\partial}{\partial \sigma^{\word{v}}}\bracketsigW[T]{\bell^X},\, \word{v} \in V,$ are given by
$$
\dfrac{\partial}{\partial \sigma^{\word{v}}}\bracketsigW[T]{\bell^X} = \bracketsigW[T]{-(\word{v}\shuprod\bsigma)\word{0} + \rho\left(\word{v1} - \dfrac{1}{2}\word{v}\proj{1}\word{0}\right) + \bar\rho\left(\word{v2} - \dfrac{1}{2}\word{v}\proj{2}\word{0}\right)}.
$$
We will compare the numerical computation of the option's delta (i.e., $\theta = S_0$, which implies $\bell^F_1 = \emptyword, \ \bell^F_2 = S_0\emptyword$) for the following choices of the function $h$:
$$h_1 = (1, 0),\quad h_2 = (\D^1X_T, 0),\quad h_3 = (0, 1), \quad h_4 = (0, \D^2X_T).$$
In this case, we observe that since $\bell^F_1 = \emptyword, \ \bell^F_2 = S_0\emptyword$, we have
$$
\bell^F_1 \diamond^{\word{i}, \word{i}}_{\word{0}} \bell^{h^i} = \bell^F_2 \diamond^{\word{i}, \word{i}}_{\word{0}} \bell^{h^i} = 0, \quad {\bf \Psi}^{\word{i}}_{\word{0}}(\bell^F_1) = {\bf \Psi}^{\word{i}}_{\word{0}}(\bell^F_2) = 0, \quad i = 1, 2.
$$
Furthermore, since all the non-zero coefficients of $\bell^X$ correspond to the words with at most one letter $\word{2}$, we have
$$
\bell^X \diamond^{\word{2}, \word{2}}_{\word{0}} \bell^{X}\diamond^{\word{2}, \word{2}}_{\word{0}} \bell^{X} = 0, \quad {\bf \Psi}^{\word{2}}_{\word{0}}({\bf \Psi}^{\word{2}}_{\word{0}}(\bell^X)) = 0, \quad {\bf \Psi}^{\word{22}}_{\word{0}}(\bell^X) = 0.$$ 
Moreover, a straightforward computation shows that
$$
{\bf \Psi}^{\word{2}}_{\word{0}}(\bell^X ) = \bar\rho \bsigma \word{0}, \quad \bell^X \diamond^{\word{2}, \word{2}}_{\word{0}} \bell^X = \bar\rho^2\bsigma\shupow{2}\word{0}, \quad {\bf \Lambda}_{\word{2}}(\bell^X) = \bar\rho\bsigma\word{2}.
$$ This leads to a significant simplification of the Malliavin weights. Indeed, applying Theorem~\ref{thm:ibp_sig} and the formulae derived in Example~\ref{ex:weights_example} yields the Malliavin weights given in the second column of Table~\ref{tab:table_weights}.

Note that the formulas corresponding to $h^3$ and $h^4$ can be rewritten simply as 
\[
\dfrac{W_T^2}{\bar\rho S_0 \int_0^T \sigma_t\,dt}
\quad\text{and}\quad
\dfrac{\int_0^T \sigma_t\,d W_t^2}{\bar\rho S_0 \int_0^T \sigma_t^2\,dt},
\]
respectively. These weights can be obtained via standard integration by parts with respect to $W^2$, without invoking Theorem~\ref{thm:ibp_sig}. Hence, they do not depend on the choice of the volatility process as long as $\D^2 \sigma_t \equiv 0$. We include these weights as a sanity check, observing that our general result, Theorem~\ref{thm:ibp_sig}, indeed covers this particular case and leads to the same expressions. Moreover, we observe that the weights corresponding to $h_3$ and $h_4$ diverge when $\rho \in \{-1, 1\}$.

Another possible choice of $h$ is $h_t = \dfrac{1}{\D_t^2 X_T}$, as considered by \citet[Section~4.2.4]{alos2021malliavin}. While this choice simplifies the computation of the weight, it introduces the integral $\int_0^T \dfrac{1}{\sigma_t}\,dW_t^2$, which is not well defined if the volatility process $\sigma$ touches zero. Furthermore, this integral cannot be expressed as a linear functional of the signature and therefore requires discretization. This, in turn, results in a loss of numerical stability and tractability. For this reason, we do not include such estimators in our comparison.

\renewcommand{\arraystretch}{2.5}
\begin{table}[H]
\centering
\begin{tabular}{|c|c|c|c|}
\hline
$h$ & Malliavin weight $\pi_T$ & \makecell{Signature order $N$\\ ($\mathrm{deg}(\bsigma)=M$)} & \makecell{Numerical \\ stability} \\
\hline

$h_1 = (1, 0)$ & 
$\dfrac{\bracketsigW[T]{\word{1}}}{S_0\bracketsigW[T]{{\bf \Psi}^{\word{1}}_{\word{0}}(\bell^X)}} 
+ 
\dfrac{\bracketsigW[T]{{\bf \Psi}^{\word{1}}_{\word{0}}({\bf \Psi}^{\word{1}}_{\word{0}}(\bell^X))}}
{S_0\bracketsigW[T]{{\bf \Psi}^{\word{1}}_{\word{0}}(\bell^X)}^2}$ 
& 
$N = 2M + 1$ & 
{\color{Red}\ding{55}} \\
\hline

$h_2 = (\D^1X_T, 0)$ & 
$\dfrac{\bracketsigW[T]{({\bf \Lambda}_{\word{1}} - {{\bf\Psi}}^{\word{11}}_{\word{0}})(\bell^{X})}}
{S_0\bracketsigW[T]{\bell^X \diamond^{\word{1},\word{1}}_{\word{0}} \bell^{X}}}
+ 
\dfrac{\bracketsigW[T]{\bell^X \diamond^{\word{1},\word{1}}_{\word{0}} \bell^{X} 
\diamond^{\word{1},\word{1}}_{\word{0}} \bell^{X}}}
{S_0\bracketsigW[T]{\bell^X \diamond^{\word{1},\word{1}}_{\word{0}} \bell^{X}}^2}$
&
$N = 6M + 1$ & 
{\color{Green}\ding{51}} \\
\hline

$h_3 = (0, 1)$ & 
$\dfrac{\bracketsigW[T]{\word{2}}}{\bar\rho S_0\bracketsigW[T]{\bsigma\word{0}}}$
&
$N = M + 1$ & 
{\color{Red}\ding{55}} \\
\hline

$h_4 = (0, \D^2X_T)$ & 
$\dfrac{\bracketsigW[T]{\bsigma\word{2}}}{\bar\rho S_0\bracketsigW[T]{\bsigma\shupow{2}\word{0}}}$
&
$N = 2M + 1$ & 
{\color{Green}\ding{51}} ($|\rho| \neq 1$) \\
\hline

\end{tabular}
\caption{Comparison of Malliavin weights, signature order, and numerical stability for different $h$}
\label{tab:table_weights}
\end{table}

\begin{sqremark}
    Different choices of $h$ generally require computing different maximal levels of the signature in order to obtain the exact expressions. For instance, if one takes the coefficient $\bsigma$ of degree $M$ (that is, $\bsigma^{\word{v}} \neq 0$ only for $\word{v}$ such that $|\word{v}| \leq M$), then the degree of the coefficients appearing in $\pi_T$ is $6M + 1$ for $h_2$, only $2M + 1$ for $h_1$ and $h_4$, and $M + 1$ for $h_3$. However, there is no theoretical guarantee that the denominators of the weights corresponding to $h_1$ and $h_3$ do not vanish. In Appendix~\ref{sect:instabilities}, we illustrate that it is indeed possible to find model parameters such that $\pi_T$ blows up. We summarize the discussion above in Table~\ref{tab:table_weights}.
\end{sqremark}

We illustrate the convergence for standard vanilla and digital at-the-money (ATM) options 
$$
f_{\mathrm{vanilla}}(x) = (e^x - K)^+, \quad f_{\mathrm{digital}}(x) = \indic{e^x \geq K},
$$
with $K = S_0$. 

We provide in Figure~\ref{fig:path-dep-model} the convergence diagrams for \eqref{eq:greek_sigvol} under the signature volatility model with parameters 
\begin{equation}\label{eq:perfect_correl}
    \bsigma = 0.2\cdot\emptyword + 0.1\cdot(\word{1} + \word{10} + \word{110} + \word{111}), \quad \rho = -1.
\end{equation}
The signature truncation order was set to $N = 8$. Note that in this particular case of perfect correlation $\rho = -1$, the formulas corresponding to the functions $h_3$ and $h_4$ are not applicable, so we plot the results only for the first two functions $h_1$ and $h_2$. 

To reduce the variance of the Malliavin delta estimator, we apply the localization procedure proposed by \citet{Fournie1999} with $\delta = 10$,
which consists in localizing the singularity of the payoff. More precisely, the payoff function~$f$ is decomposed as
\[
f(x) = G_\delta(x) + F_\delta(x), \quad x \in \R,
\]
where $G_\delta \in C_b^1$ is a regular function, while $F_\delta$ has a singularity at $x = \log K$ but vanishes whenever $|e^x - K| > \delta$. One then applies the Malliavin method only to the function~$F_\delta$ in order to reduce the Monte Carlo variance induced by the weight~$\pi_T$, while using pathwise differentiation (i.e., differentiation under the expectation) for the function~$G_\delta$. For the explicit construction of the functions $G_\delta$ and $F_\delta$ in the case of a call option payoff, we refer to \citep[Section~4]{Fournie1999}.
 
We compare the Malliavin-based Monte Carlo approach using $n_{\mathrm{max}} = 10^5$ trajectories with the option's delta obtained from the characteristic function of the log-price (see Appendix~\ref{sect:char_func}), and with the standard finite-difference approach. The latter consists of approximating the option's delta by
$$
\frac{\partial}{\partial s}\mathbb{E}\left[f\left(\log S_T\right) \mid S_0=s\right] \approx \frac{\mathbb{E}\left[f\left(\log S_T\right) \mid S_0=(1 + \varepsilon)s\right] - \mathbb{E}\left[f\left(\log S_T\right) \mid S_0=(1 - \varepsilon)s\right]}{2\varepsilon s},
$$
where the expectations are estimated via the Monte Carlo method using common random numbers. Throughout the experiments, we fix $\varepsilon = 0.01$.

\begin{figure}[H]
    \centering
    \begin{subfigure}[b]{1\textwidth}
        \centering
        \includegraphics[width=\textwidth]{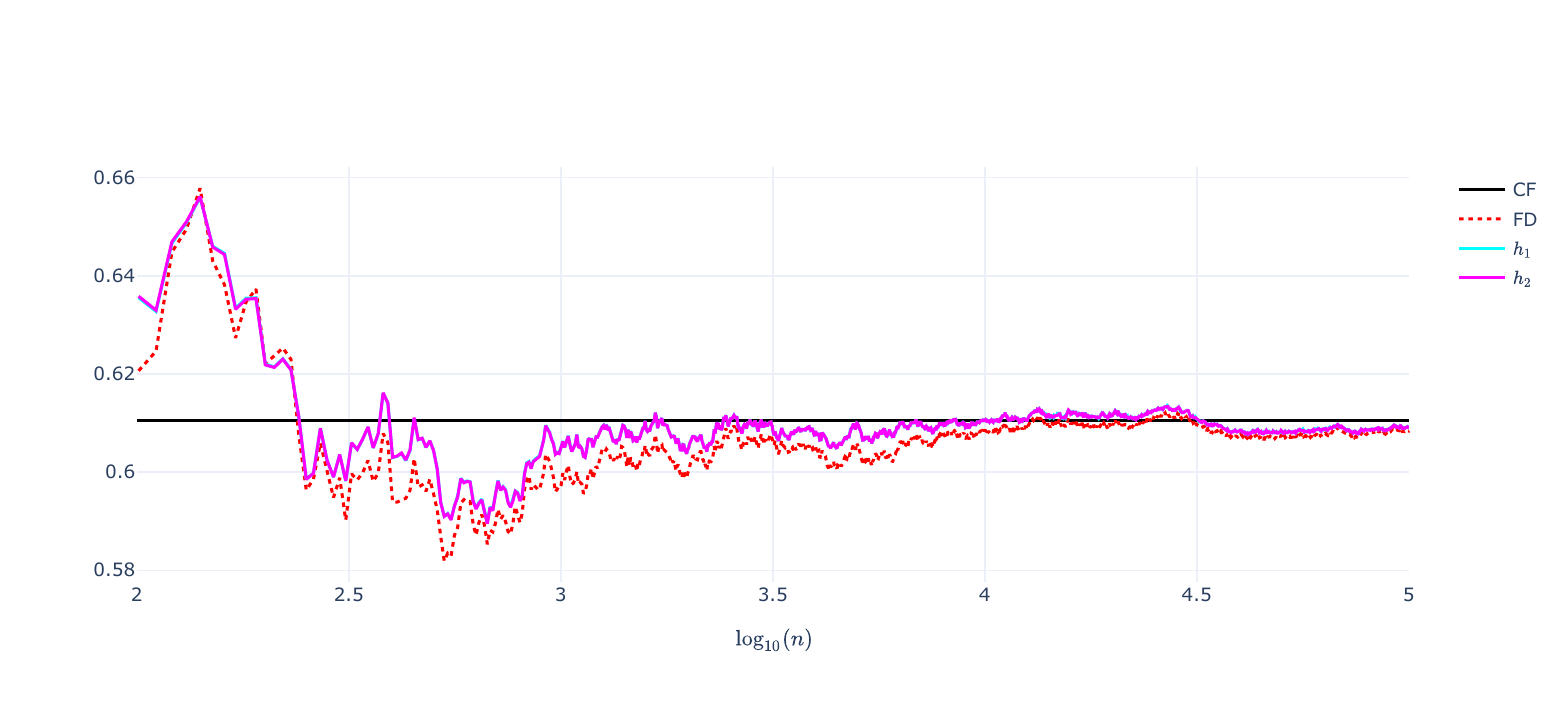}
        \caption{Delta of the vanilla ATM call option.}
    \end{subfigure}

    \vspace{1em}

    \begin{subfigure}[b]{1\textwidth}
        \centering
        \includegraphics[width=\textwidth]{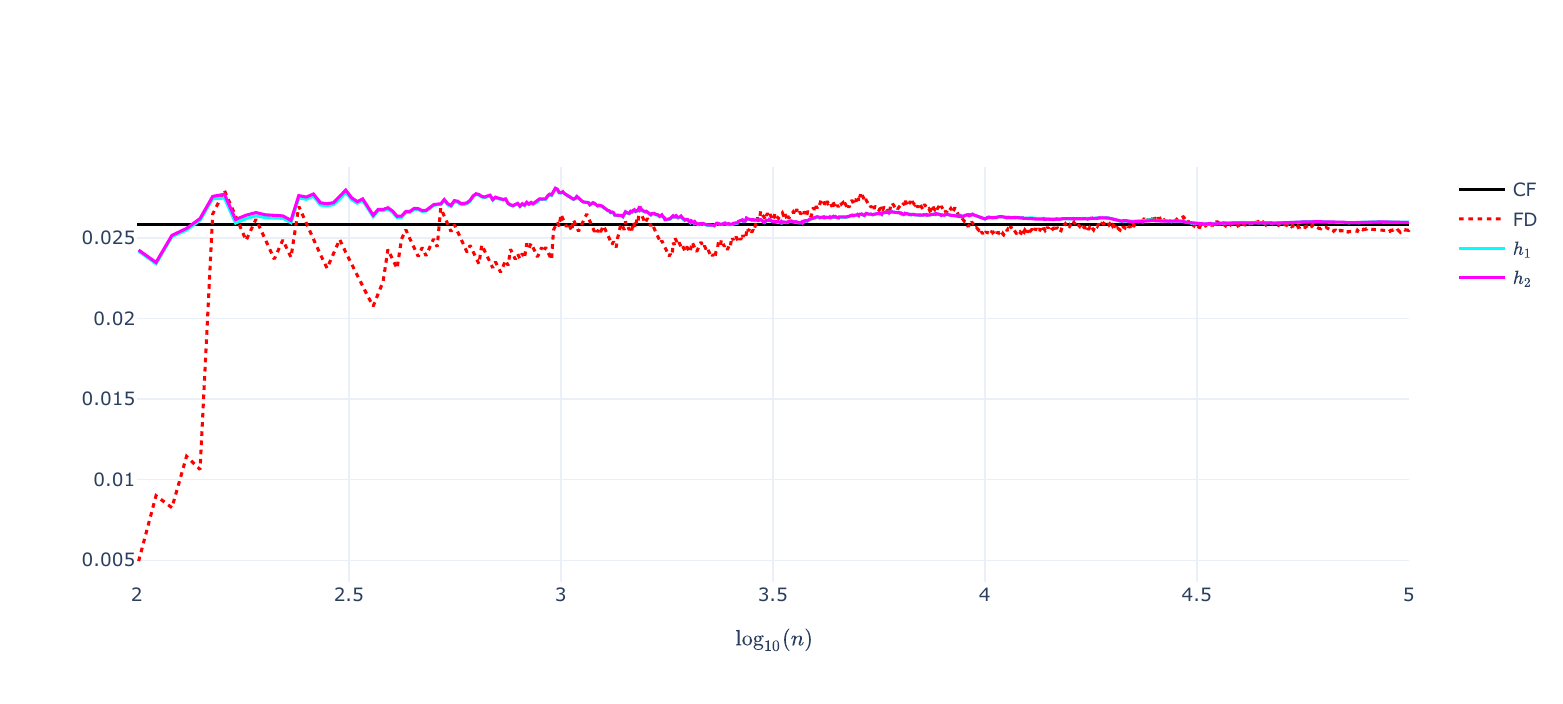}
        \caption{Delta of the digital ATM call option.}
    \end{subfigure}

    \caption{Convergence diagram for vanilla and digital ATM options in the perfect-correlation model \eqref{eq:perfect_correl}. The black horizontal line represents the delta computed using the characteristic function. The dotted red curve shows the results of the finite-difference method, while the two solid curves correspond to the Malliavin method with functions $h_1$ and $h_2$, respectively.}
    \label{fig:path-dep-model}
\end{figure}

In Figure~\ref{fig:conv-diagram-europ-stochvol}, we show the convergence diagram corresponding to the model
\begin{equation}\label{eq:ex_stoch_vol_eur}
    \bsigma = 0.25\cdot\emptyword + 0.04\cdot(\word{1} + \word{01} + \word{110} + \word{111}), \quad \rho = -0.9,
\end{equation}
with the signature truncated at order $N = 7$, as the number of simulations $n$ varies from $0$ to $n_{\mathrm{max}} = 10^5$. As expected, in both cases the Malliavin method exhibits faster convergence than the finite-difference method when the payoff function $f$ is less regular (discontinuous in our case).

\begin{figure}[H]
    \centering
    \begin{subfigure}[b]{1\textwidth}
        \centering
        \includegraphics[width=\textwidth]{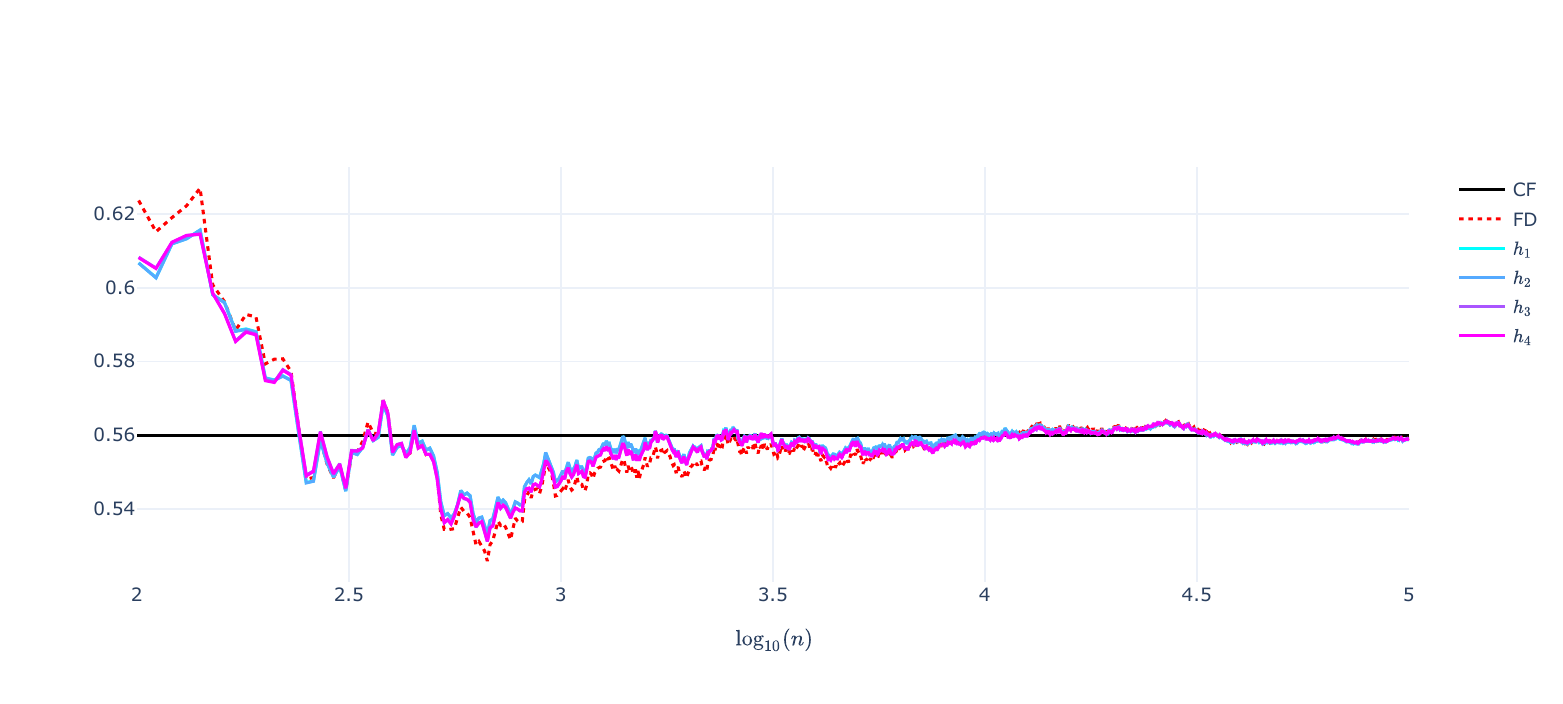}
        \caption{Delta of the vanilla ATM call option.}
    \end{subfigure}

    \vspace{1em}

    \begin{subfigure}[b]{1\textwidth}
        \centering
        \includegraphics[width=\textwidth]{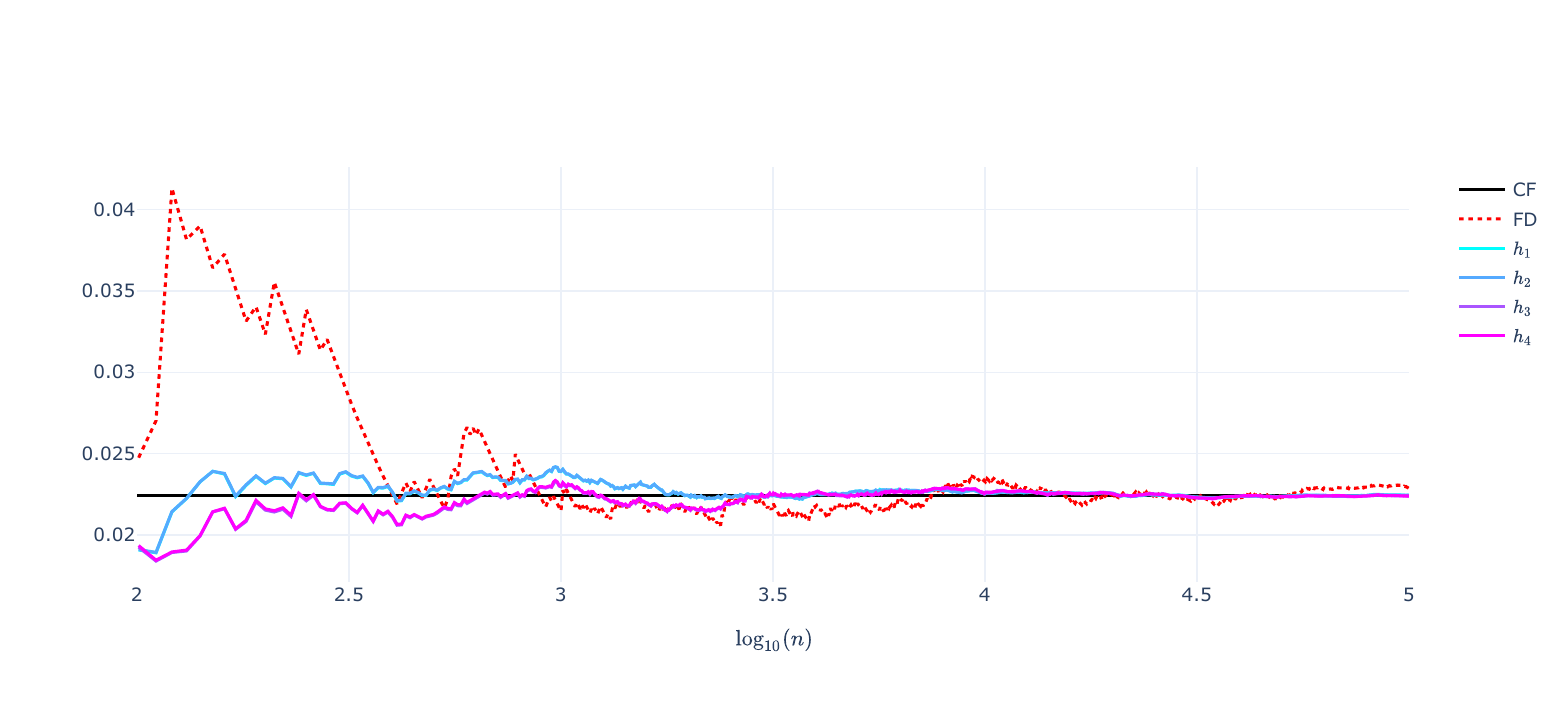}
        \caption{Delta of the digital ATM call option.}
    \end{subfigure}

    \caption{Convergence diagram for vanilla and digital ATM options in the stochastic volatility model \eqref{eq:ex_stoch_vol_eur}. The dotted red curve shows the results of the finite-difference method, while the four solid curves correspond to the Malliavin method with functions $h_1, h_2, h_3$, and $h_4$, respectively.}
    \label{fig:conv-diagram-europ-stochvol}
\end{figure}

\subsection{Asian options}
We now consider a genuinely path-dependent Asian option with payoff function $(S_t)_{t \in [0, T]} \mapsto f\left(\frac{1}{T}\int_0^T S_t\,dt\right)$, which depends on the average underlying price over $[0, T]$. Note that since $\log S_T = \bracketsigW[T]{\bell^X}$, we have $S_T = e^{\bracketsigW[T]{\bell^X}} = \bracketsigW{\exp^{\shuprod}(\bell^X)}$ and
\begin{equation}
    Y_T = \frac{1}{T}\int_0^T S_t\,dt = \bracketsigW[T]{\frac{1}{T}\exp^{\shuprod}({\bell^X})\otimes\word{0}},
\end{equation}
recall the definition of the shuffle exponential \eqref{eq:shuexp_def}.
We apply Theorem~\ref{thm:ibp_sig} to standard vanilla and digital payoff functions
$$
f_{\mathrm{vanilla}}(x) = (x - K)^+, \quad f_{\mathrm{digital}}(x) = \indic{x \geq K},
$$
with $K = S_0$ and the coefficients
$$
\bell^{G, \mathrm{Asian}} = \bell_1^F = \frac{1}{T}\exp^{\shuprod}({\bell^X})\otimes\word{0}, \quad \bell_2^F = S_0\emptyword. 
$$ 
Unlike the European option case, where a significant simplification of the Malliavin weight was possible for the functions $h_k$, $k = 1,2,3,4$, defined in the previous subsection, we now use the generic formula \eqref{eq:universal_weight_linearized} for all four functions.

\begin{figure}[H]
    \centering
    \begin{subfigure}[b]{1\textwidth}
        \centering
        \includegraphics[width=\textwidth]{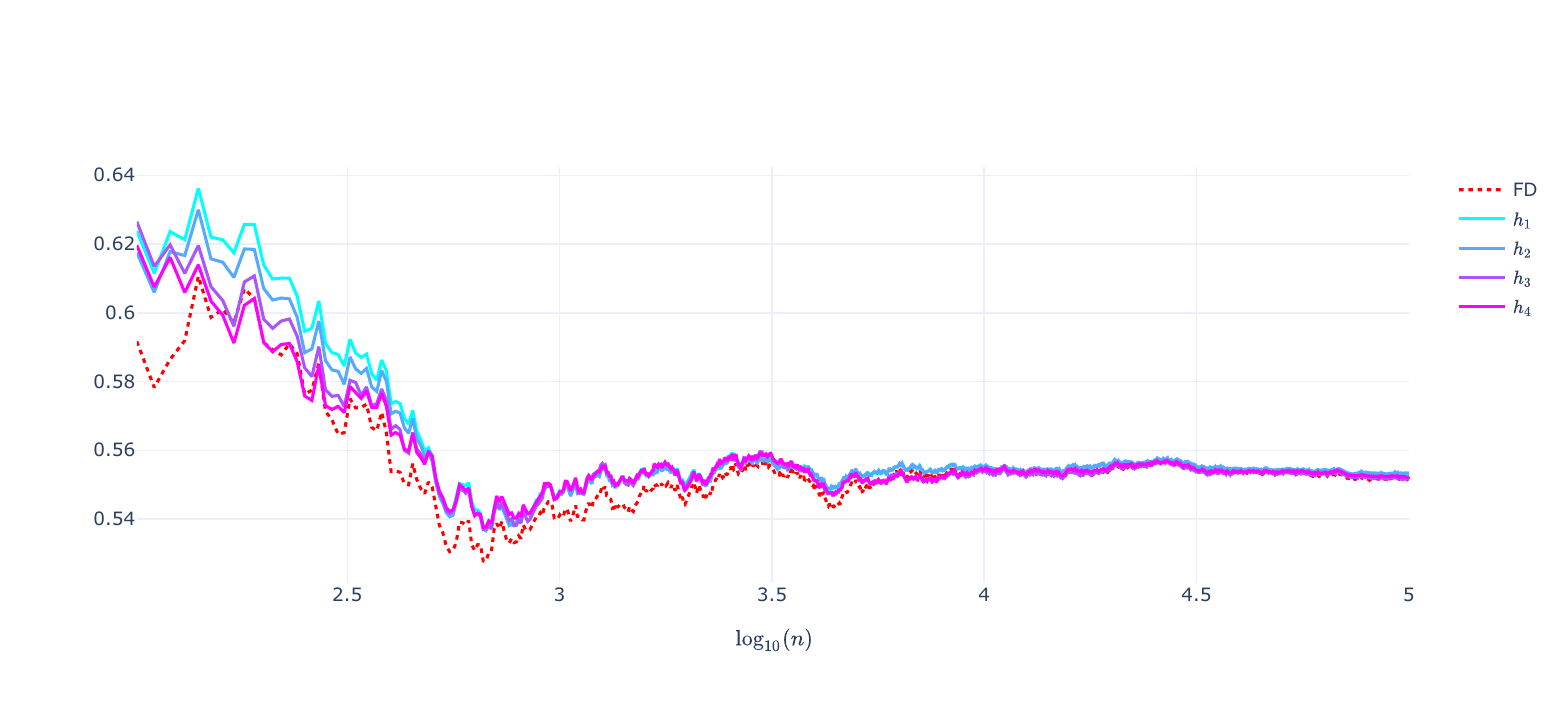}
        \caption{Delta of the Asian ATM call option.}
    \end{subfigure}

    \vspace{1em}

    \begin{subfigure}[b]{1\textwidth}
        \centering
        \includegraphics[width=\textwidth]{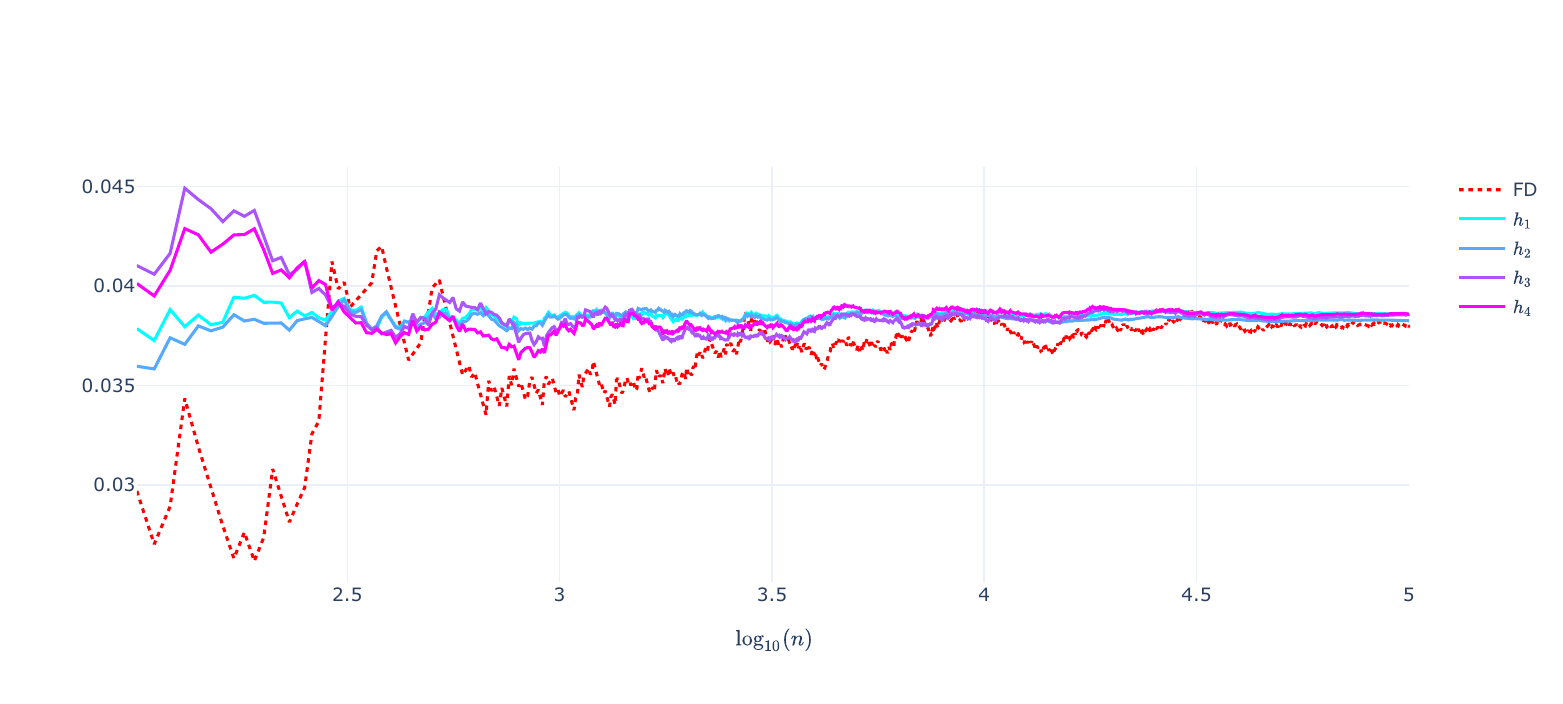}
        \caption{Delta of the digital Asian ATM call option.}
    \end{subfigure}

    \caption{Convergence diagram for vanilla and digital Asian ATM options in the stochastic volatility model \eqref{eq:ex_sto_vol_asian}. The dotted red curve shows the results of the finite-difference method, while the four solid curves correspond to the Malliavin method with functions $h_1, h_2, h_3$, and $h_4$, respectively.}
    \label{fig:asian}
\end{figure}

We consider the signature volatility model with parameters
\begin{equation}\label{eq:ex_sto_vol_asian}
    \bsigma = 0.25\cdot\emptyword + 0.1\cdot\word{1} + 0.05 \cdot \word{10}, \quad \rho = -0.9,
\end{equation}
truncated at order $N = 8$. 

The convergence diagrams are shown in Figure~\ref{fig:asian}. Again, we observe faster convergence for the discontinuous payoff, while for the vanilla Asian option the estimator remains close to the one given by the finite-difference approach.
We also note that $\bell^{G,\mathrm{Asian}}$ is an infinite sequence, so truncation error is inevitable, especially when performing two ``diamond'' products in the weight formula \eqref{eq:universal_weight_linearized}. Despite this, we observe that the Malliavin method provides sufficiently accurate results even with a relatively small truncation order, namely $N = 8$.

\appendix
 
\section{Malliavin differentiability of Stratonovich integrals}\label{sect:strato_diff}

For simplicity of notation, we present the proof of Theorem~\ref{thm:strato_mal_der} for the process
$$
X_t = \int_0^t a_s \circ dW_s,
$$
where $a$ satisfies \eqref{eq:norm_condition}--\eqref{eq:strato_reg_assump_second} and $W$ is a one-dimensional Brownian motion. The general case can be treated analogously.

The proof is an adaptation of the results in \cite[Section 3.1.1]{nualart2006malliavin} to the case of Stratonovich integrals considered as processes $\int_0^\cdot a_s \circ dW_s$, rather than random variables $\int_0^T a_s \circ dW_s$, and also establishes the existence of the derivative $(r,t) \mapsto \D_r \int_0^t a_s \circ dW_s$ as an element of $L^2(\Omega \times [0,T]^2)$. Our goal is therefore to show that
$$
X \in \mathbb{L}^{1,2}
\quad \text{and} \quad
\D_r X_t = a_r\mathds{1}_{[0, t]}(r) + \int_0^t \D_r a_s \circ dW_s.
$$

Recall that for a square-integrable process $h$, the Stratonovich integral $\int_0^T h_s \circ dW_s$, when it exists, is defined as the limit in probability of the Riemann sums
$$
S^\pi(h_\cdot) =
\sum_{j=0}^{n-1}
\dfrac{1}{t_{j+1}-t_j}
\left( \int_{t_j}^{t_{j+1}} h_s \, ds \right)
\left( W_{t_{j+1}} - W_{t_j} \right),
$$
along partitions $\pi = \{0 = t_0 < t_1 < \ldots < t_n = T\}$ as their diameter
$|\pi| := \sup_{j=0,\ldots,n-1} |t_{j+1}-t_j| \to 0$; see, e.g., \citet[Definition 3.1.1]{nualart2006malliavin}.

Let $(Y,\mathcal{Y},\nu)$ be a measurable space with a $\sigma$-finite atomless measure $\nu$, and let $h \in L^2(\Omega \times Y \times [0,T])$. In what follows, we will need to define the Stratonovich integrals $\int_0^T h(y,s)\circ dW_s$, depending on a parameter $y \in Y$, as elements of $L^2(\Omega \times Y)$.

We denote by $P^\pi h$ the stepwise approximation of $h$ corresponding to the partition $\pi$:
$$
P^\pi h(y,\cdot)
=
\sum_{j=0}^{n-1}
\dfrac{1}{t_{j+1}-t_j}
\int_{t_j}^{t_{j+1}} h(y,s)\, ds \;
\mathds{1}_{(t_j,t_{j+1}]}(\cdot),
$$
which, for almost every $(\omega,y) \in \Omega \times Y$, is the orthogonal projection of $h(y,\cdot)$ onto the space of step functions generated by $\{\mathds{1}_{(t_j,t_{j+1}]}\}_{j=0}^{n-1}$ associated with the partition $\pi$.

The following lemma extends \cite[Lemma 3.1.2]{nualart2006malliavin}.

\begin{lemma}\label{lem:L2_proj}
If $h \in L^2(\Omega \times Y \times [0,T])$, then
$$
\|P^\pi h - h\|_{L^2(\Omega \times Y \times [0,T])}
\xrightarrow{|\pi|\to 0} 0.
$$
\end{lemma}

\begin{proof}
For almost every $(\omega,y) \in \Omega \times Y$, we have $h(y,\cdot) \in L^2([0,T])$, and by the density of step functions,
$$
\int_0^T |P^\pi h(y,s) - h(y,s)|^2 \, ds
\xrightarrow{|\pi| \to 0} 0.
$$
Moreover,
$$
\int_0^T |P^\pi h(y,s) - h(y,s)|^2 \, ds
\le
2 \int_0^T |P^\pi h(y,s)|^2 \, ds
+
2 \int_0^T |h(y,s)|^2 \, ds
\le
4 \int_0^T |h(y,s)|^2 \, ds,
$$
where we used the projection property
$\|P^\pi h(y,\cdot)\|_{L^2([0,T])}
\le
\|h(y,\cdot)\|_{L^2([0,T])}$.
Integrating this bound with respect to $\mathbb{P} \otimes \nu(dy)$ yields
$$
\|P^\pi h - h\|_{L^2(\Omega \times Y \times [0,T])}^2
=
\E\!\left[
\int_Y \int_0^T |P^\pi h(y,s) - h(y,s)|^2 \, ds \, \nu(dy)
\right]
\le
4 \|h\|_{L^2(\Omega \times Y \times [0,T])}^2.
$$
The dominated convergence theorem then implies
$$
\|P^\pi h - h\|_{L^2(\Omega \times Y \times [0,T])}^2
\xrightarrow{|\pi| \to 0} 0.
$$
\end{proof}

    \begin{lemma}\label{lem:Skorokhod_conv}
        If $h_n \to 0$ in $L^2(\Omega \times Y\times[0, T])$ and $\D h_n \to 0$ in $L^2(\Omega \times Y\times[0, T]^2)$ as $n \to \infty$, then
        $$
        \|\delta(s \mapsto h_n(\cdot, s))\|_{L^2(\Omega \times Y)} \xrightarrow{n \to \infty} 0,
        $$
        where we recall that $\delta$ is the Skorokhod integral defined in Subsection~\ref{sect:skorokhod_and_repr}.
    \end{lemma}
    \begin{proof}
        The proof relies on the standard $L^2$-bound for $\delta(u)$, where $u \in \mathbb{L}^{1, 2}$, that is,
        $$
        \E[\|u\|_{L^2([0, T])}^2] + \E[\|\D u\|_{L^2([0, T]^2)}^2] < \infty.
        $$
        In this case, the following bound holds
        \begin{equation}\label{eq:skorokhod_bound}
            \E[|\delta(u)|^2] \leq\E[\|u\|_{L^2([0, T])}^2] + \E[\|\D u\|_{L^2([0, T]^2)}^2].
        \end{equation}
        For the proof, we refer to \citep[Proposition 1.3.1]{nualart2006malliavin}.

        Applying \eqref{eq:skorokhod_bound} to $\delta(h_n(y, \cdot))$, integrating with respect to $\nu(dy)$ and applying Fubini's theorem, we obtain
        $$
        \|\delta(s \mapsto h_n(\cdot, s))\|_{L^2(\Omega \times Y)}^2 \leq \|h_n\|_{L^2(\Omega \times Y\times[0, T])}^2 + \|\D h_n\|_{L^2(\Omega \times Y\times[0, T]^2)}^2 \xrightarrow{n \to \infty} 0.
        $$
    \end{proof}

    \begin{lemma}\label{lem:L2_strato}
Suppose that $h \in L^2(\Omega \times Y \times [0,T])$ and $\D h \in L^2(\Omega \times Y \times [0,T]^2)$, and that there exist processes $\D^+ h, \D^- h \in L^2(\Omega \times Y \times [0,T])$ satisfying
\begin{equation}
\begin{aligned}\label{eq:D_plus_h_assump}
\lim_{n \to \infty} \int_Y \int_0^T
\sup_{\substack{s < u < T \\ |u-s| \le 1/n}}
\E\!\left[ \big| \D_s h(y,u) - (\D^{+} h)(y,s) \big|^2 \right]
ds\, \nu(dy) &= 0, \\
\lim_{n \to \infty} \int_Y \int_0^T
\sup_{\substack{0 < u < s \\ |s-u| \le 1/n}}
\E\!\left[ \big| \D_s h(y,u) - (\D^{-} h)(y,s) \big|^2 \right]
ds\, \nu(dy) &= 0.
\end{aligned}
\end{equation}
Then, the Stratonovich integral $\int_0^T h(y,s)\circ dW_s$ is well-defined in $L^2(\Omega \times Y)$ as the limit of the Riemann sums $S^\pi(s \mapsto h(\cdot,s))$ as $|\pi| \to 0$.
\end{lemma}

\begin{proof}
We first note that the Riemann sums can be rewritten as 
\begin{equation}
\begin{aligned}\label{eq:riem_sums_l2}
S^\pi(s \mapsto h(y,s))
&= \sum_{j=0}^{n-1}
\dfrac{1}{t_{j+1}-t_j}
\left( \int_{t_j}^{t_{j+1}} h(y,s)\, ds \right)
(W_{t_{j+1}} - W_{t_j}) \\
&= \delta(P^\pi h(y,\cdot))
+ \sum_{j=0}^{n-1}
\dfrac{1}{t_{j+1}-t_j}
\int_{t_j}^{t_{j+1}} \int_{t_j}^{t_{j+1}}
\D_t h(y,s)\, ds\, dt.
\end{aligned}
\end{equation}
The final equality follows from the fact that $P^\pi h(y, \cdot)$ is a step function with random coefficients, in conjunction with the formula \eqref{eq:IBP_formula}.

By Lemma~\ref{lem:L2_proj}, the process $s \mapsto P^\pi h(\cdot,s)$ converges to $h$ in $L^2(\Omega \times Y \times [0,T])$, and $\D P^\pi h$ converges to $\D h$ in $L^2(\Omega \times Y \times [0,T]^2)$. It then follows from Lemma~\ref{lem:Skorokhod_conv} that the Skorokhod integral term in \eqref{eq:riem_sums_l2} converges to $\delta(h(y,\cdot))$ in $L^2(\Omega \times Y)$.

The convergence of the second term in \eqref{eq:riem_sums_l2} to
\[
\frac12 \int_0^T (\D^+ h)(y,s)\, ds
+
\frac12 \int_0^T (\D^- h)(y,s)\, ds
\]
in $L^2(\Omega \times Y)$ is obtained by repeating the proof of \cite[Theorem~3.1.1]{nualart2006malliavin}, replacing the $L^1(\Omega)$ norms with $L^2(\Omega \times Y)$ norms.
\end{proof}

\subsection{Proof of Theorem~\ref{thm:strato_mal_der}}

\paragraph{Step 1. Well-posedness of the Stratonovich integral.}

To show that the Stratonovich integral
$$
t \mapsto \int_0^t a_s \circ dW_s
$$
is well-defined as an element of $L^2(\Omega \times [0,T])$, we define the process $h(t,s) = a_s \mathds{1}_{[0,t]}(s)$ and verify the assumptions of Lemma~\ref{lem:L2_strato} with $Y = [0,T]$ and $\nu(dy) = dy$. Clearly,
$$
\|h\|_{L^2(\Omega \times [0,T]^2)}^2
=
\E\!\left[
\int_0^T \int_0^T
|a_s \mathds{1}_{[0,t]}(s)|^2 \, ds\, dt
\right]
\le
T \, \E\!\left[\int_0^T |a_s|^2 \, ds \right]
< \infty,
$$
and
$$
\|\D h\|_{L^2(\Omega \times [0,T]^3)}^2
=
\E\!\left[
\int_0^T \int_0^T \int_0^T
|\D_r a_s \mathds{1}_{[0,t]}(s)|^2
\, ds\, dt\, dr
\right]
\le
T \, \E\!\left[
\int_0^T \int_0^T |\D_r a_s|^2 \, ds\, dr
\right]
< \infty,
$$
by \eqref{eq:norm_condition}.

We set $(\D^+ h)(t,s) := (\D^+ a)_s \mathds{1}_{[0,t]}(s)$. Condition \eqref{eq:D_plus_h_assump} becomes
\begin{align*}
\lim_{n \to \infty}
&\int_0^T \int_0^T
\sup_{\substack{s < u < T \\ |u-s| \le 1/n}}
\E\!\left[
\big| \D_s h(t,u) - (\D^{+} h)(t,s) \big|^2
\right] ds\, dt \\
&=
\lim_{n \to \infty}
\int_0^T \int_0^T
\sup_{\substack{s < u < T \\ |u-s| \le 1/n}}
\E\!\left[
\big| \D_s a_u \mathds{1}_{[0,t]}(u)
- (\D^+ a)_s \mathds{1}_{[0,t]}(s)
\big|^2
\right] ds\, dt \\
&\le
\lim_{n \to \infty}
\left(
T \int_0^T
\sup_{\substack{s < u < T \\ |u-s| \le 1/n}}
\E\!\left[
\big| \D_s a_u - (\D^+ a)_s \big|^2
\right] ds
+
\frac{1}{n}
\int_0^T \E[|(\D^+ a)_s|^2]\, ds
\right)
= 0,
\end{align*}
by \eqref{eq:strato_reg_assump}. The condition for $(\D^- h)(t,s) := (\D^- a)_s \mathds{1}_{[0,t]}(s)$ is verified similarly. Lemma~\ref{lem:L2_strato} therefore yields convergence of the Riemann sums $S^\pi(s \mapsto h(\cdot,s))$ to the Stratonovich integral in $L^2(\Omega \times [0,T])$.
    
 \paragraph{Step 2. Malliavin derivative of the Stratonovich integral.}
We next show that the $L^2(\Omega \times [0,T]^2)$ limit of $\D S^\pi(s \mapsto h(\cdot,s))$ exists and is equal to
\begin{equation}\label{eq:D_S_pi_limit}
\lim_{|\pi| \to 0}
\D S^\pi(s \mapsto h(\cdot,s))
=
a \mathds{1}_{[0,\cdot]}
+
\int_0^T \D a_s \mathds{1}_{[0,\cdot]}(s)\circ dW_s.
\end{equation}
By the closability of the operator
$\D \colon L^2(\Omega \times [0,T]) \to L^2(\Omega \times [0,T]^2)$,
this implies
$$
\D_r \int_0^t a_s \circ dW_s
=
a_r \mathds{1}_{[0,t]}(r)
+
\int_0^t \D_r a_s \circ dW_s,
$$
for a.e.\ $(r,t) \in [0,T]^2$, and consequently
$\int_0^\cdot a_s \circ dW_s \in \mathbb{L}^{1,2}$.

To prove \eqref{eq:D_S_pi_limit}, we compute
\begin{equation}
\begin{aligned}\label{eq:D_S_pi}
\D_r S_t^\pi(s \mapsto h(t,s))
&=
\sum_{j=0}^{n-1}
\dfrac{1}{t_{j+1}-t_j}
\left(
\int_{t_j}^{t_{j+1}} a_s \mathds{1}_{[0,t]}(s)\, ds
\right)
\mathds{1}_{(t_j,t_{j+1}]}(r)
\\
&\quad +
\sum_{j=0}^{n-1}
\dfrac{1}{t_{j+1}-t_j}
\left(
\int_{t_j}^{t_{j+1}} \D_r a_s \mathds{1}_{[0,t]}(s)\, ds
\right)
(W_{t_{j+1}} - W_{t_j}).
\end{aligned}
\end{equation}

By Lemma~\ref{lem:L2_proj}, the first term converges to
$a_r \mathds{1}_{[0,t]}(r)$ in $L^2(\Omega \times [0,T]^2)$ as $|\pi| \to 0$.

The second term in \eqref{eq:D_S_pi} is a Riemann sum corresponding to the process
$\tilde{h} \in L^2(\Omega \times [0,T]^3)$ defined by
$\tilde{h}(r,t,s) = \D_r a_s \mathds{1}_{[0,t]}(s)$.
It follows from \eqref{eq:norm_condition} that
$$
\|\tilde{h}\|_{L^2(\Omega \times [0,T]^3)}^2 < \infty,
\qquad
\|\D \tilde{h}\|_{L^2(\Omega \times [0,T]^4)}^2 < \infty.
$$

Define
$(\D^+ \tilde h)(r,t,s) := \D_r(\D^+ a)_s \mathds{1}_{[0,t]}(s)$
and
$(\D^- \tilde h)(r,t,s) := \D_r(\D^- a)_s \mathds{1}_{[0,t]}(s)$.
Using the same argument as in Step~1 together with
\eqref{eq:strato_reg_assump_second}, we verify condition
\eqref{eq:D_plus_h_assump} for $\tilde h$.
Applying Lemma~\ref{lem:L2_strato} then yields convergence in
$L^2(\Omega \times [0,T]^2)$ of the second term in \eqref{eq:D_S_pi}
to
$\int_0^t \D_r a_s \circ dW_s$.
This completes the proof.
    
\section{Greeks via the characteristic function}\label{sect:char_func}

If the following two conditions are satisfied:
\begin{enumerate}
    \item the model admits a computable characteristic function $\phi(u) := \E\left[e^{iu\log\frac{S_T}{S_0}}\right],\ u\in\R$;
    \item for a given payoff $f$, a Fourier inversion formula is available,
\end{enumerate}
then one can obtain a quasi-analytic Fourier representation for Greeks with respect to $S_0$, computable via numerical integration without the need for Monte Carlo simulation. The characteristic function $\phi(u)$ in the signature volatility model can be computed by solving an infinite-dimensional Riccati equation; see \citep*[Section~5]{abijaber2024signature}.

For example, for a vanilla call option, Lewis’ formula yields
\begin{equation}\label{eq:lewis_vanilla_call}
    C(T, K) = \E[(S_T - K)^+] = S_0 - \frac{K}{\pi}\int_0^\infty \Re\left(e^{i\left(u - \frac{i}{2}\right)\log\frac{S_0}{K}}\phi\!\left(u - \frac{i}{2}\right)\right)\frac{d u}{u^2 + \frac{1}{4}},
\end{equation}
so that the option Delta is given by formal differentiation with respect to $S_0$:
$$
\Delta(T, K) = 1 - \frac{1}{\pi}\frac{K}{S_0}\int_0^\infty \Re\left(i\left(u - \frac{i}{2}\right)e^{i\left(u - \frac{i}{2}\right)\log\frac{S_0}{K}}\phi\!\left(u - \frac{i}{2}\right)\right)\frac{d u}{u^2 + \frac{1}{4}},
$$
which can be computed much faster than Monte Carlo methods.

Similarly, using the relationship $C_{\mathrm{digital}}(T, K) = -\partial_K C(T, K)$ between digital and vanilla call options, one obtains a Fourier representation for the digital option price by differentiating \eqref{eq:lewis_vanilla_call} with respect to $K$:
\begin{equation}
     C_{\mathrm{digital}}(T, K) = \E[\indic{S_T \ge K}] = \frac{1}{\pi}\int_0^\infty \Re\left(\left(\frac{1}{2} - iu\right)e^{i\left(u - \frac{i}{2}\right)\log\frac{S_0}{K}}\phi\!\left(u - \frac{i}{2}\right)\right)\frac{d u}{u^2 + \frac{1}{4}}.
\end{equation}
Taking the derivative with respect to $S_0$, we obtain the Delta of the digital option:
\begin{equation}
     \Delta_{\mathrm{digital}}(T, K) = \frac{1}{\pi S_0}\int_0^\infty \Re\left(e^{i\left(u - \frac{i}{2}\right)\log\frac{S_0}{K}}\phi\!\left(u - \frac{i}{2}\right)\right)\, d u.
\end{equation}

This approach is rather formal, since we do not justify the differentiation of the integrals with respect to parameters, and no rigorous result concerning the existence of solutions to the infinite-dimensional Riccati equations is currently available. Hence, we emphasize that the Greeks obtained by this method are used only as sanity checks to demonstrate consistency with other methods.

\section{Numerical instabilities in the Malliavin weights}\label{sect:instabilities}

We consider the simplest possible stochastic volatility model,
\begin{equation}\label{eq:model_instab}
    \bsigma = 0.02\cdot\emptyword + 0.05\cdot \word{2}, \qquad \rho = -0.95,
\end{equation}
with signature truncation order $N = 7$. Equivalently, the volatility is an affine function of Brownian motion, $\sigma_t = 0.02 + 0.05\, W_t$. 
Even in this elementary setting, the denominators $\bracket{\D G}{h}$ appearing in the Malliavin weights of Theorem~\ref{thm:ibp_sig} may vanish, causing the method to become unstable for $h = h_1$ and $h = h_3$ from Table~\ref{tab:table_weights}. Figure~\ref{fig:instab_hist} shows that the support of $\bracket{\D G}{h}$ is contained in $\mathbb{R}^+$ for $h_2$ and $h_4$, whereas for $h_1$ and $h_3$, zero lies in the interior of the support. As a consequence, the corresponding Greek estimators exhibit large jumps, as seen in Figure~\ref{fig:instab_diag}, ultimately leading to divergence of the Monte Carlo estimator.

\begin{figure}[H]
    \centering
        \includegraphics[width=\textwidth]{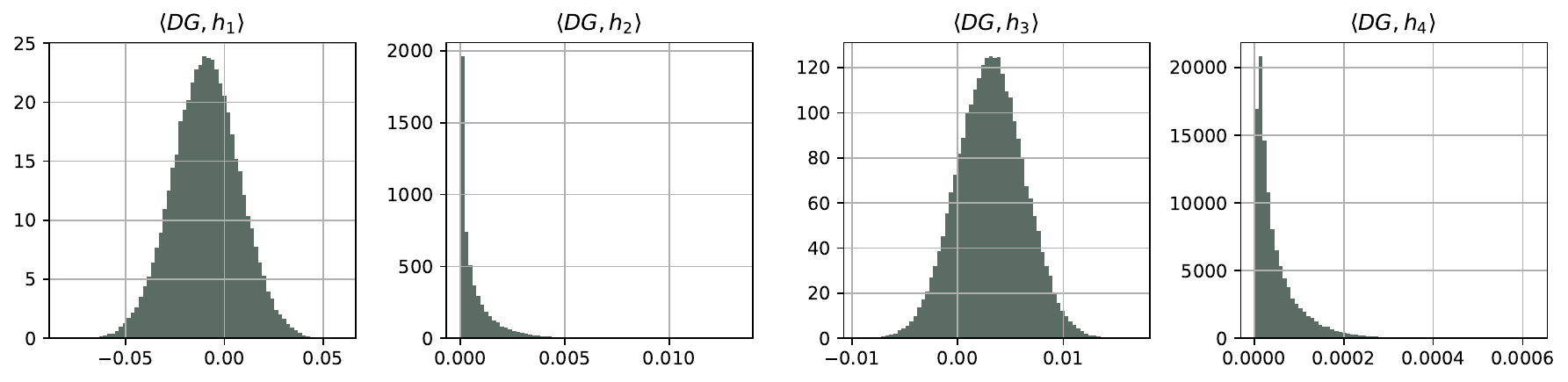}
    \caption{Empirical distribution of \(\bracket{\D G}{h}\) for the functions $h_k,\, k = 1,2,3, 4,$ listed in Table~\ref{tab:table_weights}.}
    \label{fig:instab_hist}
\end{figure}

\begin{figure}[H]
    \centering
        \includegraphics[width=\textwidth]{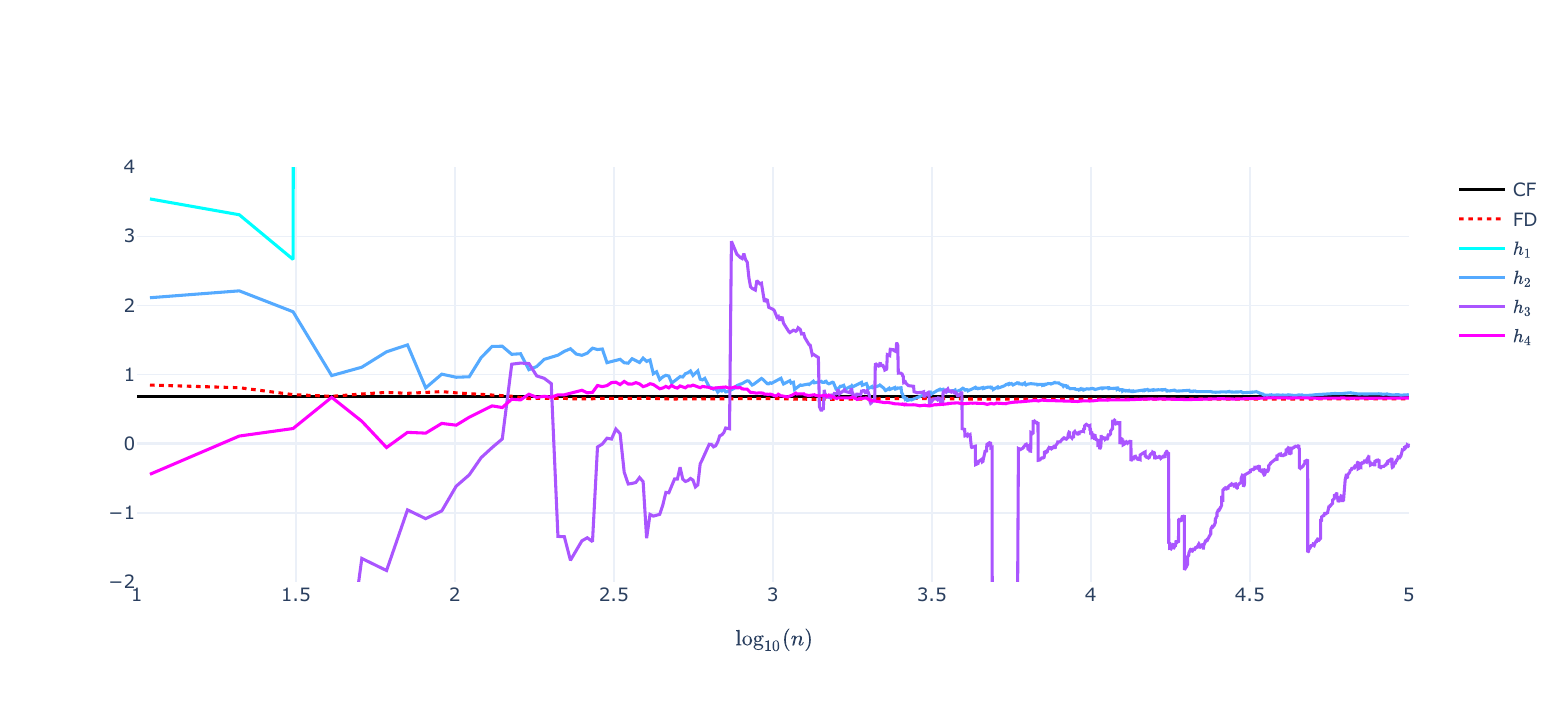}
    \caption{``Convergence'' diagram for vanilla ATM options in the stochastic volatility model \eqref{eq:model_instab}}
    \label{fig:instab_diag}
\end{figure}
\bibliographystyle{plainnat}
\bibliography{bibl.bib}

\end{document}